\documentclass[11pt,reqno]{amsart}
\usepackage{amsmath,amsthm,amsfonts,amssymb,amscd,mathrsfs}
\usepackage{paralist}
\usepackage{graphicx}
\usepackage{mathrsfs}
\usepackage{subcaption}
\usepackage{footmisc}
\DeclareSymbolFontAlphabet{\mathrsfs}{rsfs}
\usepackage{bm}
\usepackage{color}
\usepackage{hyperref}
\usepackage[myheadings]{fullpage}

\theoremstyle{plain}
\newtheorem{theorem}{Theorem}[section]

\newtheorem{lemma}[theorem]{Lemma}

\newtheorem{remark}[theorem]{Remark}
\newtheorem{prop}[theorem]{Proposition}

\definecolor{azure}{rgb}{0, 0.5, 1}

\def\cX{{\mathcal X}}
\def\cY{{\mathcal Y}}

\def\cF{{\mathcal F}}

\def\E{\mathbb{E}}
\def\N{\mathbb{N}}
\def\P{\mathbb{P}}

\def\R{\mathbb{R}}

\def\I{\mathbb{I}}

\def\al{\alpha}

\def\ka{\kappa}
\def\la{\lambda}

\def\de{\delta}

\def\ga{\gamma}

\def\var{\varepsilon}
\def\bs{\boldsymbol}

\begin{document}
	\title[EBH and a probabilistic approach in characterizing potential landscape] {Essential barrier height and a probabilistic approach in characterizing potential landscape}
	
	\author{Yao Li}\address{Yao Li: Department of Mathematics and Statistics, University of Massachusetts Amherst, Amherst, MA, 01002, USA}\email{yaoli@math.umass.edu}
	
	\author{Molei Tao}\address{Molei Tao:
		School of Mathematics, Georgia Institute of Technology, Atlanta, GA 30332 USA}\email{mtao@gatech.edu}

	\author{Shirou Wang*}
	\address{Shirou Wang:  School of Mathematics, Jilin University, Changchun 130012, China, and Department of Mathematical and Statistical Sciences, University of Alberta, Edmonton, AB T6G 2G1, Canada
	} \email{shirou@jlu.edu.cn}
	
	\begin{flushleft}
		{\it Stochastic Processes Appl.},  to appear
		
		\vspace{5mm}
	\end{flushleft}
 
 \subjclass[2020]{Primary {37H10}; Secondary { 60H10, 60J22, 60J60}}
 
\keywords{Potential landscape, coupling method, overdamped Langevin dynamics, essential barrier height, non-convexity.}

\thanks{* Corresponding author.}

\thanks{Y. Li was partially supported by NSF DMS-2108628. 
M. Tao is grateful for partial support from NSF DMS-1847802, Cullen-Peck Scholarship, and GT-Emory AI.Humanity Award. S. Wang was partially supported by NSFC (12201244, 12271204), and a faculty development grant from Jilin University.}
 
\begin{abstract}
This paper proposes a probabilistic approach to investigate the shape of landscapes of multi-dimensional potential functions. 
Under a suitable coupling scheme, two copies of the overdamped Langevin dynamics associated with the potential function are coupled, and the coupling times are collected. Assuming a set of intuitive yet technically challenging  conditions on the coupling scheme, it is shown that the tail distributions of the coupling times exhibit qualitatively different dependencies on the noise magnitude for single-well  versus multi-well potential functions.  
More specifically, for convex single-well potentials, the negative tail exponent of the coupling time distribution is uniformly bounded away from zero  by the convexity parameter and is independent of  the noise magnitude. In contrast, for multi-well potentials, the negative tail exponent decreases exponentially as the noise vanishes,  with the decay rate governed by the {\it essential barrier height}, a quantity introduced in this paper to characterize the non-convex nature of the potential function. 
Numerical investigations are conducted for a variety of examples, including the Rosenbrock function, interacting particle systems, and loss functions arising in artificial neural networks. These examples not only illustrate the theoretical results in various contexts but also provide crucial numerical validation of the  conjectured assumptions, which are  essential to the theoretical analysis yet lie beyond the reach of standard technical tools. 
\end{abstract}
\maketitle

\section{Introduction}
The concept of potential functions is fundamental  in both continuous and discrete time dynamics. 
In continuous-time dynamics, it arises in both  conservative systems (e.g., Hamiltonian dynamics) and dissipative systems (e.g., gradient flow and damped mechanical systems). In discrete-time dynamics, it often corresponds to the objective function of an optimization algorithm or, more generally, to a variational inequality. In all these contexts, characterizing the landscape of the potential function, particularly  in high dimensions, is often crucial. For example, understanding the existence, locations, and connections of local minima, saddle points, and global minima of neural network training objectives is essential for comprehending both the training dynamics and the generalization capabilities of machine learning models (e.g., \cite{soltanolkotabi2018theoretical, allen2019learning, du2019gradient, cooper2021global, wang2021large, liu2022loss}). 

This paper proposes a probabilistic approach to understanding how local minima are globally connected in a potential landscape. 
Let $U$ be a  smooth function defined on  a regular  domain $D\subseteq\R^k (k\ge1)$ with finitely many local minima $x_1,\dots,x_L.$ 
Generically, denote by $(\varphi^t)_{t\ge0}$ the 
negative gradient flow of $U$. Then each $x_i$ is a stable equilibrium of $(\varphi^t)_{t\ge0}$ with the basins (of attraction) given by
	\[B_i=\{x\in D:\varphi^t(x)\to x_i\ \text{as} \ t\to\infty\}.\]

Call  $U$  a single-well potential  if it has only one local minimum $x_1$ (i.e., $L=1$) such that $D=B_1$, and call $U$ a  multi-well potential  if 
$L\ge 2$ and $D=\cup_{1\le i\le L}B_i$ up to a Lebesgue null set. A multi-well potential is, 
in particular, called a  double-well potential if $L=2.$ Throughout the paper, the following is always assumed for $U$:

\medskip

\noindent{\bf (U1)}
The potential function  $U\in C^3(D),$ where $D$ is open, convex and connected,  such that
$\lim\nolimits_{x\to\partial D}U(x)=\infty$, and if $D$ is unbounded, it further holds that 
\begin{eqnarray*}
	\lim\nolimits_{x\to\partial D}|\nabla U|=\infty,\ \lim\nolimits_{x\to\partial D}|\nabla U(x)|-2\Delta U(x)=\infty, 
\end{eqnarray*}
where $|\cdot|$ denotes the Euclidean norm.

\medskip

\noindent{\bf Remark}:
In the single-well setting, {\bf (U1)} ensures the existence of a global strong solution of \eqref{SDE1}. In the multi-well setting,  further assumptions on the finiteness and non-degeneracy of the saddle points and local minima, as stated in {\bf (U2)} or  {\bf (U3)}(iii), guarantee this \cite{Baur,metastability_book}.

\medskip

Our approach makes strong use of the coupling idea in probability. Given two stochastic processes $\bs X=\{\cX_t;t\ge0\}, \bs Y=\{\cY_t;t\ge0\}$ on $\R^k$, a {\it coupling} of $\bs X$ and $\bs Y$ is a stochastic process $\{(X_t,Y_t); t\ge0\}$ on $\R^{2k}$ satisfying the following:

 (i) For any $t>0,$ 
	$X_t$ ({\it resp.} $Y_t$) has the same law as $\cX_t$ ({\it resp.} $\cY_t$);
	
(ii) If $X_s =Y_s$ for certain $s>0,$ then $X_t =Y_t$ for all $t\ge s.$

\noindent 
The coupling time $\tau_c$ 
is defined to be the first meeting time between $X_t$ and $Y_t$, i.e., \begin{eqnarray}\label{coupling_time1}
		\tau_c=\inf\nolimits\{t\ge0:X_t=Y_t\}.
	\end{eqnarray}
 A coupling is said to be successful if $\tau_c<\infty$ almost surely. Henceforth, a coupling is denoted by $(X_t,Y_t)$ for simplicity and clarity.
 
 Coupling is a classical tool for comparing two  probability measures and, in the context of stochastic processes, provides a probabilistic approach to investigate the distributional convergence of the process \cite{reflection1986,chen1989, lindvall_book,eberle2019coupling,hairer_note}.  In this paper, the coupling method is utilized to characterize the landscape of a potential  function $U$.  The two stochastic processes being coupled are the overdamped Langevin dynamics,  which satisfy the  stochastic differential equation (SDE)
	\begin{eqnarray}\label{SDE1}
		dZ_t=-\nabla U(Z_t)dt+\var dB_t,
	\end{eqnarray}
	where $\{B_t;t\ge0\}$ is a $k$-dimensional Brownian motion and  $\var>0$ is the  noise magnitude.
Under effective coupling methods, the coupling time distribution for  Langevin dynamics
usually exhibits exponential tails  (e.g., \cite{eberle2019coupling, li-wang2020}), indicating  intuitive connections with the characteristics of the potential function $U.$

We will focus on how the exponential tails of  the coupling time distributions depend on the noise magnitude $\var$.
The main message is that, under a {\it suitable} coupling scheme,  this dependence exhibits both {quantitatively} and  {qualitatively} different behaviors between potential functions with only a single well and  those with multiple wells. 
More specifically, if denote
$r(\var)=-\limsup_{t\to\infty}\frac{1}{t}\log\P[\tau_c>t],$
then for a single-well potential $U$, $r(\var)$ is uniformly bounded away from zero,  independent of $\var$ (see Theorem \ref{thm:1}); whereas for a multi-well potential $U$, $r(\var)$
decreases exponentially with respect to $\var,$ leading to the emergence of a quantity called {\it essential barrier height}, which quantifies the level of non-convexity of the potential $U$ in a certain sense (see Theorem \ref{thm:2}-Theorem \ref{thm:5}).

Various coupling methods have been developed in  different contexts since the pioneer work of Doeblin \cite{doeblin1938}. 
In this paper, for the purpose of coupling efficiency, we use  a {\it mixture} of two particular coupling methods: reflection coupling and maximal coupling (see Section 2 for details on these two methods). Specifically, for a certain threshold distance $d>0,$  the coupling method between $X_t$ and $Y_t$ is switched between the reflection  and maximal coupling 
 in such a way that  $(X_t, Y_t)$ evolves according to the reflection ({\it resp.} maximal) coupling whenever $|X_t-Y_t|> d$  ({\it resp.} $|X_t-Y_t|\le d$) until a successful coupling is attained (i.e., $X_t=Y_t$ for some $t$). This coupling scheme is referred to as the {\it reflection-maximal coupling}.
It was developed in \cite{li-wang2020} to compute the geometric  convergence rate of stochastic dynamics, and a similar scheme is  utilized to compute the convergence rate for Markov processes \cite{eberle2019quantitative}.
 
How should the threshold $d$ be chosen? 
We note that the maximal coupling is defined in the discrete-time setting, specifically for the time-$h$ sampled chain of the SDE. The choice of $d$ should be chosen so that, if  $|X^h_{n-1}-Y^h_{n-1}|<d$, then the distributions of the time-$h$ sampled chains $X_{n}^h$ and $Y^h_{n}$ have sufficient overlap to ensure that the probability of successful coupling, $\P[X^h_{n}=Y^h_{n}]$, is of order $\mathcal{O}(1)$.
Lemma \ref{lem:max_coupling} shows that by taking $d=\mathcal O(\var\sqrt{h}),$ both the coupling probability and the expected distance between $X^h_n$ and $Y^h_n$ can be controlled suitably. Hereafter,  we refer to the scheme as the ``$h$-reflection-maximal coupling"
when emphasizing  the time step size $h$; otherwise, we simply refer to it as  the reflection-maximal coupling, typically assuming  a small $h$ without specifying its exact value.

Although the theoretical results established in this paper do not depend on the choice of discretization scheme, 
in numerical simulations, 
the Euler-Maruyama scheme is adopted for all numerical examples. This is because 
its probability density function at any given point can be explicitly computed, which is required for the implementation of the maximal coupling.
With additional effort to evaluate  the relevant densities, the reflection-maximal coupling  can also be adapted to other numerical schemes, such as the Milstein scheme. Since the primary goal of this paper is to demonstrate the effectiveness of the reflection-maximal coupling method 
in characterizing the potential landscape, the Euler-Maruyama scheme is used throughout the  numerical examples. 

The first main result of this paper concerns the single-well potential. Let $U$ be a single-well potential on a convex domain $D$. The function $U$ is said to be {\it strongly convex} (with constant $m_0>0$) if	\begin{eqnarray}\label{def:unif_convex}
		\langle\nabla U(x)-\nabla U(y), x-y\rangle\ge m_0|x-y|^2,\quad \forall x,y\in D,
	\end{eqnarray}
	where $\langle\cdot,\cdot\rangle$ denotes the standard inner product in $\R^k$.	
	The supremum of all positive values of $m_0$ satisfying \eqref{def:unif_convex} is called the convexity parameter of $U$. Henceforth, $m_0$  always denotes the convexity parameter. 
	
\begin{theorem}\label{thm:1}
Let $U$ be a single-well potential  satisfying {\bf (U1)} and strongly convex  with  constant $m_0>0.$ Given any $\delta>0$, there exists $h_0>0$ such that for any $h\in(0,h_0)$, if $(X_t,Y_t)$ is an  $h$-reflection-maximal coupling  of two solutions of \eqref{SDE1}  satisfying $\E[|X_0-Y_0|]<\infty,$ then for 
any $\varepsilon > 0$, it  holds that 
\begin{eqnarray*}
\limsup_{t\to\infty}\dfrac{1}{t}	\log\P[\tau_c>t]\le -m_0+\delta.
\end{eqnarray*}	
\end{theorem}
\medskip

\noindent{\bf Remark}:
Theorem \ref{thm:1} provides only an upper bound for the coupling time in the single-well case, in contrast to the asymptotic characterizations established for the  multi-well case in Theorems \ref{thm:2} and \ref{thm:5} below. Deriving a lower bound would require identifying a mechanism by which two coupled trajectories fail to meet within a sufficiently long time. In the absence of energy barriers, as in the single-well case, such a mechanism is not straightforward. Even under the simplifying assumption that the potential is quadratic, estimating the probability of near-coupling without success involves estimates on the first hitting times of the Ornstein–Uhlenbeck process, for which explicit formulas are generally not  available \cite{lipton2018}. In practice, the exponential tail of the coupling time distribution for the single-well case is expected to be governed by the smallest eigenvalue of the Hessian at the global minimum; see  Section 5.2.

\medskip
	
When $U$ has multiple wells,  a crucial quantity  is the  least barrier height of any continuous path connecting  two local minima of $U$.  More specifically, given  two subsets $A,B\subseteq D,$ the communication height between $A$ and $B$ is defined as
\begin{eqnarray}\label{tilde U}
\Phi(A,B)=\inf_{\substack{{\phi\in C([0,1],D)},\\ \phi(0)\in A,\ \phi(1)\in B}}\sup\nolimits_{t\in[0,1]}U(\phi(t)),
\end{eqnarray}
where the infimum is taken over all continuous paths in $D.$ It is straightforward to observe that $\Phi(A,B)=\Phi(B,A).$

For a double-well potential $U$ with  two local minima $x_1,x_2$, define the {\it essential barrier height} as
\begin{eqnarray}\label{barrier_height}
		H_U=\min\big\{\Phi(x_1,x_2)-U(x_1), \Phi(x_1,x_2)-U(x_2)\big\}, 
	\end{eqnarray}
which represents the lower of the two barrier heights that must be crossed when transitioning from one local minimum to the other.
In the double-well setting, the potential function is assumed to satisfy the following generic conditions.

	\medskip
	
	\noindent{\bf (U2)}
	Let  $U:D\to\R$ 
	be a double-well potential function  satisfying  {\bf (U1)}  with  two local minima $x_1$ and $x_2.$ The following hold:
	
	\medskip
	{(i)} The communication height between $x_1$ and $x_2$ is attained at a unique saddle point $z^*(x_1,x_2),$ i.e., 
	\begin{eqnarray*}
		U(z^*(x_1,x_2))=\Phi(x_1,x_2); 
	\end{eqnarray*}
	
	\medskip

	{(ii)} $U$ is non-degenerate (i.e., the Hessian of $U$  has only non-zero eigenvalues)  at the two local minima $x_1,x_2,$ and at the saddle point $z^*(x_1,x_2).$ 

	\medskip

 In the multi-well setting, in addition to assumptions on the potential function, several key properties of the coupling scheme are also required; see {\bf (H1)}-{\bf (H3)} in Section 4. 
These property assumptions, while technical in form, are supported by intuitive reasoning and are numerically validated in Section 5.

When multiple  wells are present,
the coupling process is assumed to be initially related to all  basins, ensuring that all  typical scenarios are considered. More specifically, a probability measure $\mu$ on $D\times D$ is said to be  fully supported (with respect to  all local minima) if for any $\delta>0$, 
		\begin{eqnarray*}
			\mu(B_\delta(x_i)\times B_\delta(x_j))>0,\quad i,j=1,\dots,L,
		\end{eqnarray*}
	where $B_\delta(x)$ denotes the ball centered at $x$ with radius $\delta.$
A coupling  $(X,Y)$ is said to be {fully supported} if its distribution is fully supported. Analogously, a probability measure $\mu$ on $D$ is said to be fully supported if for any $\delta>0,$ 
\begin{eqnarray*}
	 	\mu(B_\delta(x_i))>0,\quad i=1,\dots,L.
 \end{eqnarray*} 
A random variable $X$ is said to be fully supported if its distribution is fully supported. Note that any probability measure  equivalent to the Lebesgue measure is fully supported. 

Throughout this paper, the notation $x\lesssim y$  ({\it resp.} $x \gtrsim y$) indicates that $x$ is bounded from above ({\it resp.} below) by a constant, which is independent of $t$ and $\var$, multiplied by $y$. The notation $x\simeq y$ means that both $x\lesssim y$ and $y \gtrsim x$ hold. 	
	\begin{theorem}\label{thm:2}
		Let $U$ be a double-well potential  satisfying {\bf (U2)}, and   $(X_t,Y_t)$  be a 
		coupling of two solutions of  \eqref{SDE1} 		
		such that $(X_0,Y_0)$ is  fully supported. Then, if the coupling $(X_t,Y_t)$ satisfies  {\bf (H1)}-{\bf (H2)},
		for any $\var>0$ sufficiently small, it holds that 
		\begin{eqnarray*}
\limsup\limits_{t\to\infty}\dfrac{1}{t}
			\log\P[\tau_c>t]\simeq -C_{\var}e^{-2H_U/{\var^2}},
		\end{eqnarray*} 
		where $H_U$ is defined in \eqref{barrier_height}, and $C_\var>0$ is a constant such that the limit
  $\lim_{\var\to0}C_\var$ exists and depends only on $U.$
\end{theorem}

In the general setting of multi-well potentials, in addition to the degeneracy of the  critical points and the uniqueness of the saddle, as specified in  {\bf (U2)}, the potential function  $U$ is also assumed to exhibit distinct  potential values and depths corresponding to the different  local minima.
	
	\medskip
	\noindent{\bf (U3)} Let $U:D\to\R$ be a multi-well potential function satisfying {\bf (U1)} with  local minima $x_1,\dots,x_L.$ The following hold: 
	
	\medskip
{(i)}  $U$ has different potential values at the different local minima. In particular, $U$ admits a unique global minimum, denoted by $x_1$;
	\medskip  	 
	
{(ii)} The different basins of potential $U$ admit different depths. More precisely,  there exists some $\de>0$ such that the $L$ local minima of $U$ can be labeled in such a way that
\begin{eqnarray}\label{ordering}
\Phi(x_i,\mathcal M_{i-1})-U(x_i)\le \min\nolimits_{\ell<i}\{\Phi(x_\ell,\mathcal M_{i}\backslash x_\ell)-U(x_\ell)\}-\de,\quad i=1,\dots,L, 
\end{eqnarray}
where $\mathcal M_0=D^c,$ $\mathcal M_i=\{x_1,\dots,x_{i}\},$ $ i=1,\dots,L;$
	\medskip
	
{(iii)} Let $\mathcal M_{i}$ be  as in (ii). Then for each $i\in\{1,\dots,L\},$   
	the communication height between $x_i$ and $\mathcal M_{i-1}$  is reached at the unique saddle point $z^*(x_i,\mathcal M_{i-1}),$ i.e., 
	\begin{eqnarray*}
		U(z^*(x_i,\mathcal M_{i-1}))=\Phi(x_i,\mathcal M_{i-1}).
	\end{eqnarray*}
Moreover, $U$ is non-degenerate at all the local minima $x_1,\ldots,x_L,$ and at the associated saddle points $z^*(x_i,\mathcal M_{i-1}), 1\le i\le L.$  
	
\bigskip	
	
Note that {\bf (U3)}(iii) reduces to {\bf (U2)} when $L=2$. We refer to Figure \ref{FigU} for an example of the potential function $U(x)$ in one dimension, which illustrates the local minima, the communication heights, the essential barrier height, the relative depths $\Phi(x_i,\mathcal M_{i}\backslash x_i)-U(x_i)$, and their relationships. It should be noted that the essential barrier height $H_U$ is {\it not} the highest communication height among the local minima, as indicated by $H_1$.

Condition
{\bf (U3)} comes from a nice work on metastability  \cite{metastability2004,metastability2005}, in which a sharp  estimate of the first hitting time from a local minimum to an {\it appropriate set} is rigorously proved. We will extensively apply this result to derive an estimate of the first hitting time to the basin of the global minimum (see Lemma \ref{lem:expon_tail_multi_well}), naturally introducing the notion of essential barrier height defined in \eqref{barrier_height_multi-well} below. This ultimately yields the coupling time estimate  for the multi-well case. 
	
	\begin{figure}
		\centering
		\includegraphics[width=0.8\linewidth]{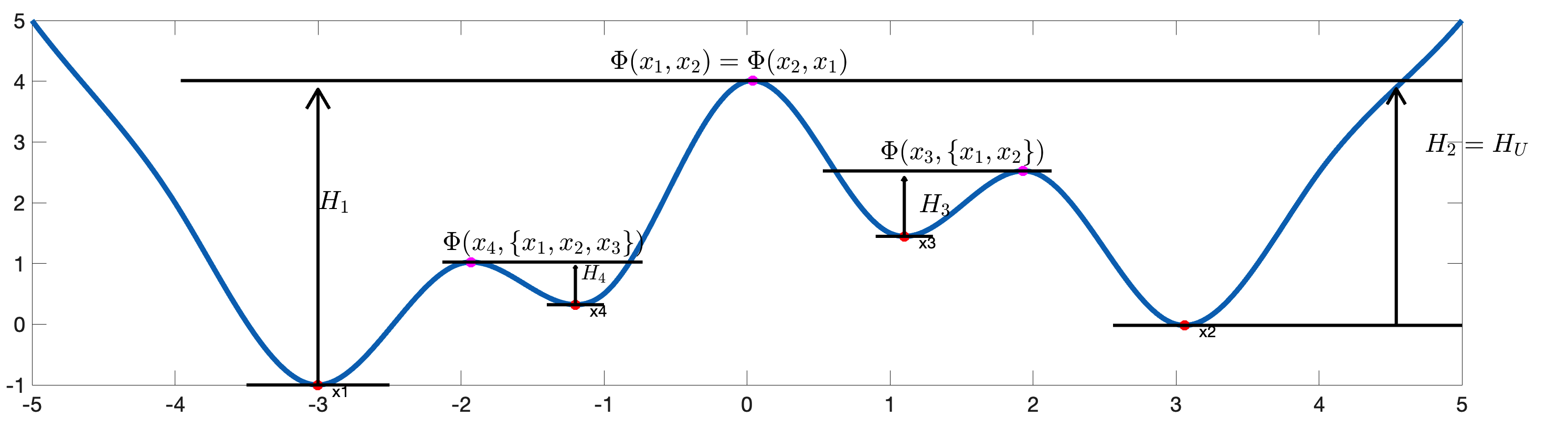}
		\caption{Example of 1D potential function, local minima, communication height, and essential barrier height $H_U$. Four relative depths $H_1 = \Phi(x_1, x_2) - U(x_1)$, $H_2= \Phi(x_2,  x_1) - U(x_2)$, $H_3 = \Phi(x_3, \{ x_1, x_2\}) - U(x_3)$, and $H_4 = \Phi(x_4, \{ x_1, x_2, x_3\}) - U(x_4)$ are demonstrated. Note that $H_2>H_3>H_4.$ In this example $H_U=H_2<H_1.$}
		\label{FigU}
	\end{figure}	
	
We now define the {\it essential barrier height} in the  general context. 
Let $U$ be a multi-well potential  satisfying {\bf (U3)}, with  $x_1$ denoting the (unique) global minimum.
The essential barrier height of $U$ is defined as
\begin{eqnarray}\label{barrier_height_multi-well}
H_U=\max\nolimits_{2\le i\le L}\big\{\Phi(x_i,x_1)-U(x_i)\big\}.
\end{eqnarray}
Note that when $L=2,$ \eqref{barrier_height_multi-well} reduces to \eqref{barrier_height}, so the definitions of essential barrier height for double- and multi-well potentials coincide. 
	
We note that the essential barrier height defined in \eqref{barrier_height_multi-well} differs from the usual notion of barrier height in the literature. The latter is  a local characterization of the potential landscape by focusing  only on the relevant barriers that must be crossed when transitioning from   one local minimum to another. In contrast, the essential barrier height considered in this paper is a {\it global} characterization, as it captures  the greatest height of the barriers that must be passed by any continuous path going towards the global minimum from any of the local minima. 
An equivalent characterization of $H_U$ will be given in Section 2.3.

	\begin{theorem}\label{thm:5}
		Let $U$ be a  multi-well potential  
		satisfying  {\bf (U3)}, and  let $(X_t,Y_t)$  be a coupling of two solutions of \eqref{SDE1} such that  $(X_0,Y_0)$ is fully supported.  Then, if  the coupling $(X_t,Y_t)$ satisfies {\bf (H1)}-{\bf (H3)}, for any $\var>0$ sufficiently small, it holds that 
\begin{eqnarray*}
\limsup_{t\to\infty}\dfrac{1}{t}
			\log\P[\tau_c>t]\simeq C_\var e^{-2H_U/{\var^2}},
		\end{eqnarray*}
		where $H_U$ is given in \eqref{barrier_height_multi-well}, and $C_\var>0$ is a constant such that the limit
  $\lim_{\var\to0}C_\var$ exists and depends only on $U.$
	\end{theorem}

The intuitive ideas underlying Theorem \ref{thm:2} and Theorem \ref{thm:5}, which relate the coupling times to the essential barrier height $H_U$, are as follows. In the double-well case, the typical scenario is that both processes enter the basin associated with the global minimum and be coupled within that basin, as they overcome a lower barrier when transitioning from the local minimum to the global one than in the reverse direction. This intuition analogously extends to  multi-well cases: when the two coupled processes start in different basins, 
the minimal height of the barriers they must overcome to reach the same basin is always no greater than $H_U$. Specifically, it is 
no greater than the lower barrier when transitioning to the basin of the global minimum. Such height can be attained when the initial basins of  the two processes 
are sufficiently ``distant'' from each other (see Section \ref{sec:hitting time}).
 
The essential barrier height, in a certain sense, quantifies the  ``global non-convexity" of multi-well potentials, which is of crucial importance in non-convex optimization problems arising in various fields. In Section 5,  we propose a numerical algorithm to compute the essential barrier height, based on the linear extrapolation of the exponential tails of coupling time distributions. The computed values are validated for both a one-dimensional double-well potential and a  multi-dimensional interacting particle system, with numerical results shown to closely match the theoretical values.
We further apply this algorithm to detect the loss landscapes of artificial neural networks.  In a two-layer neural network model, it is shown that the loss functions of large artificial neural networks (over-parameterized) have lower essential barrier heights than that of small ones (under-parameterized).
This is largely consistent with observations in the machine learning community, suggesting a promising criterion for training artificial  
neural networks based on the essential barrier height of the training loss function. 
	
This paper is organized as follows. Section 2 presents basic facts and results that will be used in the subsequent sections,  including estimates for reflection and maximal couplings,  first hitting times of Langevin dynamics under multi-well potentials, as well as probability generating functions. 
 Section 3 studies the case of the single-well potential and proves Theorem \ref{thm:1}. Section 4 investigates both double-well  and multi-well potentials, and proves Theorem \ref{thm:2} and Theorem \ref{thm:5}.  Section 5  explores various examples of single- and multi-well potentials,  in which both the theoretical findings and the assumptions on the coupling scheme  are numerically verified.

	\bigskip
	
	\section{Preliminary}

	This section prepares  key preliminary results that will be used in the rest of the paper.

	\subsection{Reflection coupling  and single-well potential}
	
	Consider two stochastic processes $\bs X,\bs Y$ satisfying  the  following stochastic differential equation
	\begin{eqnarray}\label{SDE2}
		dZ_t=g(Z_t)dt+\var dB_t, \quad Z_t\in\R^k 
	\end{eqnarray}
	with initial conditions $\mu$ and $\nu$ respectively.   
	Assume that $g:\R^k\to\R^k$ is Lipschitz continuous and satisfies additional conditions,  ensuring the unique existence of non-explosive strong solutions of \eqref{SDE2}
from any initial condition.

	A {\it reflection coupling} of $\bs X$ and $\bs Y$ is a stochastic process $\{(X_t,Y_t); t\ge0\}$ taking values in $\R^{2k}$ such that 
	$X_0\sim\mu, Y_0\sim\nu,$ and 
	\begin{eqnarray}\label{reflection_coupling}
		&&dX_t=g(X_t)dt+\var dB_t,\nonumber\\
		\quad 
		&&dY_t=g(Y_t)dt+\var P_tdB_t,\quad 0<t<\tau_c; \quad 
		Y_t=X_t,\quad t\ge \tau_c, 
	\end{eqnarray}
	where $P_t=I_k-2e_te^\top_t$ is the  orthogonal matrix in which   $e_t=(X_t-Y_t)/|X_t-Y_t|$,  and $\tau_c$ is the coupling time defined in \eqref{coupling_time1}.	
	
	The reflection coupling, as its name suggests, is to make the noise terms in $X_t$ and $Y_t$ the mirror reflection 
	of each other with respect to the middle hyperplane between $X_t$ and $Y_t$ \cite{reflection1986}. It is a  particularly efficient coupling method in high-dimension, achieved by  only keeping the noise along the vertical direction (which is one-dimensional) of the hyperplane  with noise in other directions being cancelled out. 

The following proposition states that under the method of reflection coupling, the distributions of coupling time of the overdamped Langevin dynamics along a strongly convex  single-well potential have exponential tails,  bounding away from zero by the convexity parameter. 
\begin{prop}\label{prop:sup_mart} 
Let $U$ be a single-well potential satisfying {\bf (U1)}. Assume that $U$ is strongly convex with  constant $m_0>0$.  Then given any $t_0>0$,  there exists  $c_0>0$ 
such that, if  $(X_t,Y_t)$ is a reflection coupling of two solutions of  \eqref{SDE1} with initial conditions $X_0=x_0, Y_0=y_0,$ for any $\var>0$, it holds that
\begin{eqnarray*}
\P[\tau_c>t]\le c_0\big({|x_0-y_0|}/{2\var}\big) e^{-m_0t},\quad\forall t \ge t_0.
\end{eqnarray*}
\end{prop}
	\begin{proof}
		Denote  $R_t=|X_t-Y_t|/2\var.$
		It is not hard to see that $\{R_t;t\ge0\}$ is a one-dimensional stochastic process satisfying
		\begin{eqnarray}\label{eq:2}
			dR_t=-R_t^{-1}\langle\nabla U(X_t)-\nabla U(Y_t),X_t-Y_t\rangle dt+2\var d\bar B_t,\ 0\le t<\tau_c,
		\end{eqnarray}
		where $\{\bar B_t;t\ge0\}$ is a one-dimensional Brownian motion. 
		
		By the strong convexity of $U$, the drift term in \eqref{eq:2} is upper bounded by  $-m_0R_t.$ Thus, for $t\in[0,\tau_c),$ $R_t$ is always  bounded by  the following one-dimensional Ornstein-Uhlenbeck process $\{S_t;t\ge0\}$ 
		\begin{eqnarray}\label{SDE:compare}
			dS_t=-m_0S_tdt+ d\bar  B_t,\quad S_0=|x_0-y_0|/2\var.
		\end{eqnarray}
Let $\tau_0=\inf\{t\ge0:S_t=0\}.$
It is now sufficient  to estimate $\P[\tau_0>t].$  		

By Proposition 1  in \cite{correction2000} (see also \cite{lipton2018}), 
the probability density function of $\tau_0$ has an analytic expression as follows
\begin{eqnarray}\label{pdf}
p(t)= \dfrac{S_0}{\sqrt{2\pi}}\Big(\dfrac{m_0}{\sinh(t)}\Big)^{3/2}\exp\Big\{\frac{m_0(t-S_0^2)}{2}-\frac{m_0S_0^2}{2}\coth(m_0t)\Big\},\quad t\ge0.
		\end{eqnarray}
Note that
	\begin{eqnarray*}
p(t)\le \dfrac{S_0}{\sqrt{2\pi}}\Big(\dfrac{m_0}{\sinh(t)}\Big)^{3/2}\exp\Big\{\frac{m_0(t-S_0^2)}{2}\Big\}=c_0S_0e^{\frac{m_0t}{2}}\big/(e^{m_0t}-e^{-m_0t})^{3/2}
\end{eqnarray*}
where $c_0$ is a constant independent of $t,$ $x_0,y_0$ and $\var$.
Thus, we have 
		\begin{eqnarray*}
		\P[\tau_0\ge t]
			=\int_{t}^\infty p(s)ds
			\le \dfrac{c_0|x_0-y_0|}{2\var}\int_{t}^\infty\dfrac{e^{2m_0s}}{(e^{2m_0s}-1)^{3/2}}ds.
		\end{eqnarray*}

Note that for any $p\in(0,1)$,  \[e^{2m_0s}-1\ge pe^{2m_0s},\quad\forall s\ge|\ln(1-p)|/2m_0.\] Thus, for any given $t_0>0,$ by letting 
$p\in(0,1)$ be such that $p\ge 1-e^{-2m_0t_0}$ and suitably enlarging the constant  $c_0$,
one obtains  
\begin{eqnarray*}
\P[\tau_0\ge t]\le\dfrac{c_0|x_0-y_0|}{2\var}\int_{t}^\infty\dfrac{e^{2m_0s}}{(e^{2m_0s})^{3/2}}ds
\le c_0\big({|x_0-y_0|}/{2\var}\big)e^{-m_0t},\quad\forall t\ge t_0,
\end{eqnarray*}
where $c_0$ is independent of $t,$ $x_0,y_0$ and $\var$.\end{proof}

	

\subsection{Maximal coupling and   estimations}
Let $\mu_1$ and $\mu_2$  be two probability  distributions on $\R^k.$ Call 
$(X,Y)$ a coupling of  $\mu_1$ and $\mu_2$ if	$X\sim\mu_1, Y\sim\mu_2.$
	By the well-known coupling inequality (see, for instance, Lemma 3.6 in \cite{aldous1983random}), 
	\begin{eqnarray}\label{coupling_ineq}
		\text{TV}(\mu_1,\mu_2)\le2\P[X\neq Y],
	\end{eqnarray}
	where $\text{TV}(\mu_1,\mu_2):=2\sup_{A\subseteq\R^k}|\mu_1(A)-\mu_2(A)|$ denotes the total variation distance between probability measures on $\R^k.$ 
	A coupling $(X,Y)$ is said to be a {\it maximal coupling} if the equality in \eqref{coupling_ineq} is attained, i.e., the probability $\P[X=Y]$ is maximized. 
	
		A particular way to obtain maximal coupling is as follows: Denote the  ``minimum" distribution  of $\mu_1$ and $\mu_2$ by $\nu(\cdot)=\al^{-1}\min\{\mu_1(\cdot),\mu_2(\cdot)\},$ where
	$\al$ is the normalizer satisfying $\alpha=\P[X=Y].$
    With probability $\al,$ let  $X=Y\sim\nu,$ and with probability $(1-\al),$ let $X$ and $Y$ be independently  sampled such that
	\begin{eqnarray}\label{max_coupling}
		X\sim(1-\al)^{-1}(\mu_1-\al\nu),\quad  Y\sim(1-\al)^{-1}(\mu_2-\al\nu).
	\end{eqnarray}
It is not hard to verify that $\P[X\neq Y]=\text{TV}(\mu_1,\mu_2)/2$ (see \cite{Sason}, Theorem 1).

In the context of stochastic processes,   the maximal coupling is defined  in terms of  conditional distributions of the associated discrete-time  chains. Let $\{(X^h_n,Y^h_n);n\ge0\}$ be the  time-$h$ sampled chain of a  coupling of two solutions of \eqref{SDE2}. Assume that at step $n-1,$ $(X^h_{n-1},Y^h_{n-1})$ takes the value $(x,y)\in\R^k\times\R^k,$
Then $(X^h_n,Y^h_n)$ is a  maximal coupling at step $n$ if 	\begin{eqnarray}\label{coupling_ineq1}
		\text{TV}(\mu_x,\mu_y)=2\P[X^h_{n}\neq Y^h_{n}|X^h_{n-1}=x,Y^h_{n-1}=y].
	\end{eqnarray}	
where $\mu_x$ and $\mu_y$ denote the probability distribution of $X^h_n$ and $Y^h_n$ conditioning on $X^h_{n-1}=x$  and $Y^h_{n-1}=y$, respectively. 

In the proof of Theorem \ref{thm:1}, a key step is to bound the expected distance between two coupled processes under the maximal coupling. This estimate can be derived using the independent coupling, in which the two coupled random variables are independent.

\begin{prop}\label{prop:max<ind}
Let $(X,Y)$ be a coupling of  
two random variables such that $X\sim\mu_1, Y\sim\mu_2$. Assume that $\alpha:=\P[X=Y]<1.$ Then 
\begin{eqnarray*}
	\E_{\text{\rm max}}[|X-Y|^2]\le\dfrac{2}{1-\alpha}
	\E_{\text{\rm ind}}[|X-Y|^2]
\end{eqnarray*}
where $\E_{\text{\rm max}}$ and $\E_{\text{\rm ind}}$ denote expectations with respect to the maximal coupling and independent coupling, respectively. 		
\end{prop}
\begin{proof}
By the construction of the maximal coupling, we have 
\begin{eqnarray*}
\E_{\text{max}}[|X-Y|^2]&=&(1-\alpha)\int_{\R^{2k}}|x-y|^2
\dfrac{(\mu_1-\alpha\nu)(dx)}{1-\alpha}\cdot\dfrac{(\mu_2-\alpha\nu)(dy)}{1-\alpha}\\
&\le&\dfrac{1}{1-\alpha}\int_{\R^{2k}}|x-y|^2\mu_1(dx)\mu_2(dy)+\dfrac{1}{1-\alpha}\int_{\R^{2k}}|x-y|^2(\alpha\nu)(dx)(\alpha\nu)(dy)
\end{eqnarray*}
where $\nu=\alpha^{-1}\min\{\mu_1,\mu_2\}.$ Hence, 
\begin{eqnarray*}
\E_{\text{max}}[|X-Y|^2]
\le\dfrac{2}{1-\alpha}\int_{\R^{2k}}|x-y|^2\mu_1(dx)\mu_2(dy)
=\dfrac{2}{1-\alpha}\E_{\text{ind}}[|X-Y|^2].
\end{eqnarray*}
\end{proof}

As shown in Proposition \ref{prop:max<ind}, in the context stochastic processes, obtaining an upper bound of $\E_{\text{max}}[|X^h_n-Y^h_n|^2]$ requires that $\alpha_n:=\P[X^h_n=Y^h_n]$ remains uniformly bounded away from 1, independent of  $n$. 
Although in practical simulations, it is rarely observed that 
$\P[X^h_n=Y^h_n]$ exceeds 0.8,  
a rigorous theoretical verification of this uniform bound remains  challenging.  
To address this issue, we introduce a modified construction of  maximal coupling. For any given  $\alpha_0\in(0,1]$, define $\tilde\alpha_n=\min\{\alpha_n,\alpha_0\}$
where $\alpha_n=\P[X^h_n=Y^h_n].$ Let  $\mu_{1,n}$ and $\mu_{2,n}$ denote the distributions of $X^h_n$ and $Y^h_n,$ respectively, and define $\nu_n=\al_n^{-1}\min\{\mu_{1,n}(\cdot),\mu_{2,n}(\cdot)\}.$ Then with probability $\tilde\al_n,$ let  $X^h_n=Y^h_n\sim\nu_n,$ 
and with probability $(1-\tilde\al_n),$ let $X^h_n$ and $Y^h_n$ be independently sampled according to
\begin{eqnarray}\label{max_coupling2}
	X^h_n\sim(1-\tilde\al_n)^{-1}(\mu_{1,n}-\tilde\al_n\nu_n),\quad  Y^h_n\sim(1-\tilde\al_n)^{-1}(\mu_{2,n}-\tilde\al_n\nu_n)
\end{eqnarray}
such that $X^h_n\neq Y^h_n$. 
This modification ensures  
$\P[X^h_n=Y^h_n]=\tilde\alpha_n\le\alpha_0,$ so the coupling probability is uniformly bounded by $\alpha_0.$ Note that \eqref{max_coupling2} reduces to the standard maximal coupling \eqref{max_coupling} when $\al_0\ge\al_n.$

The modified construction of the maximal coupling in \eqref{max_coupling2} is referred to as the $\alpha_0$-maximal coupling for $\al_0\in(0,1].$ Henceforth, 
the term ``maximal coupling" refers to the $\al_0$-maximal coupling with  $\al_0$ fixed at $0.8.$ 
	
Under the reflection-maximal coupling scheme, a maximal coupling is implemented whenever triggered in the previous step. More specifically, 
$(X^h_n,Y^h_n)$ is a maximal coupling, if at the previous step $n-1,$ the distance between 
$X^h_{n-1}$ and 
$Y^h_{n-1}$ does not exceed a threshold $d$. In the numerical implementation, the distance between $X^h_n$ and $Y^h_n$ is evaluated at each step $n$ to determine whether  maximal coupling should be triggered for the next step $n+1$. If the condition is not met, reflection coupling is applied instead at step $n+1.$

The triggering of maximal coupling is  a crucial mechanism, especially in the numerical schemes, for ensuring a successful  coupling. It guarantees a positive success rate of coupling in the following step when the two processes are  sufficiently close. In the absence of maximal coupling,  numerical errors may cause  two  processes to ``miss" each other,  even if  they should theoretically be coupled successfully. Moreover, maximal coupling exhibits robustness to small perturbations, making it a reliable method in numerical simulations.

The following lemma shows that in the single-well setting, choosing the threshold $d=\mathcal O(\var\sqrt h)$ ensures both an  $\mathcal O(1)$ coupling probability and a uniform bound on the expected one-step distance between the two processes. 
It provides crucial estimates for the proof of Lemma \ref{lem:2} in Section 3. 
	
\begin{lemma}\label{lem:max_coupling}
Let $U$ be a  single-well potential  satisfying {\bf (U1)} which is strongly convex, and let $(X^h_n,Y^h_n)$ be a coupling  of the time-$h$ sampled chains of two solutions of  \eqref{SDE1}. 
Assume for  $n\ge1$,  $(X^h_n,Y^h_n)$ is a maximal coupling  conditional on $X^h_{n-1}=x_0, Y^h_{n-1}=y_0,$ where $x_0,y_0\in\R^k$ satisfy $|x_0-y_0|\le d=2\var\sqrt{h}.$ Then the following hold: 
		
		(i) There exists a constant $\ga\in (0, 1)$ such that for any $n\ge1$ and $h>0$,
        \[\P[|X^h_{n}-Y^h_{n}|>0|X^h_{n-1}=x_0, Y^h_{n-1}=y_0]
		\le\ga.\]
		
		(ii) 
		For any $n\ge1$ and any $h>0$ sufficiently small, 
		\begin{eqnarray*}
	\E[|X^h_{n}-Y^h_{n}|\big|X^h_{n-1}=x_0,Y^h_{n-1}=y_0]\le  
    c_1\var\sqrt{h}.
		\end{eqnarray*}
		where the constant $c_1>0$ is independent of $h,\var,$ and $n$.
	\end{lemma}
\begin{proof}
(i) By definition, 
the one-step conditional probability
\begin{eqnarray}\label{probability}
\P[X^h_n=Y^h_n|X^h_{n-1}=x_0, Y^h_{n-1}=y_0]
\end{eqnarray}
is maximized under the standard maximal coupling (i.e., the 1-maximal coupling). In particular,  
for any alternative coupling method, such as the reflection coupling, the probability in \eqref{probability} 
is no greater than that achieved under the maximal coupling. Hence, 
it suffices to establish that under the reflection coupling, with initial condition $X_0=x_0,Y_0=y_0$ and $|x_0-y_0|\le d=2\var\sqrt{h},$ 
the probability $\P[X_h=Y_h]$ remains uniformly away from 0
for all sufficiently small $h>0.$

Denote by $m_0$ the convexity parameter of $U$. From the proof of Proposition \ref{prop:sup_mart}, in the single-well case, the coupling time of the reflection coupling is bounded by that of one-dimensional  Ornstein-Uhlenbeck process $\{S_t\}$ governed by   \eqref{SDE:compare}, whose probability density function $p(t)$ is given by \eqref{pdf}.
Without loss of generality, assume 
$m_0=1.$
Then for any sufficiently small $h>0$ and $0 < t < h$,  
\begin{eqnarray*}
p(t)&\ge& a_0S_0\exp\Big\{\frac{t}{2}-\frac{S_0^2}{2}\cdot\frac{e^{2t}+1}{e^{2t}-1}\Big\}\Big/\big(e^t-e^{-t}\big)^{3/2}\\
&\ge&a_0S_0\exp\Big\{-\frac{2S_0^2}{e^{2t}-1}\Big\}\Big/\big(e^{\frac{2t}{3}}-e^{-\frac{4t}{3}}\big)^{3/2}\\
&\ge&a_0S_0e^{-\frac{S_0^2}{t}}\big/t^{\frac{3}{2}},
\end{eqnarray*}
for some constant $a_0>0$ independent of $t$ and $h$. Integrating over $[0,h]$ yields 
\begin{eqnarray*}	
\P[X_h=Y_h]=\int_0^h p(t)dt\ge a_0S_0\int_0^he^{-\frac{S_0^2}{t}}t^{-\frac{3}{2}}dt.
\end{eqnarray*}
Applying the change of variable  $u=S_0^2\big/t$ and using the assumption $S_0=|x_0-y_0|/2\var\le\sqrt h$ yields
\begin{eqnarray}\label{reflection}
\P[X_h=Y_h]\ge a_0\int_{S_0^2/h}^\infty e^{-u}u^{-\frac{1}{2}}du\ge
a_0\int_1^\infty e^{-u}u^{-\frac{1}{2}}du:=a_1>0.
\end{eqnarray}

Now, \eqref{reflection} implies that 
under the $1$-maximal coupling, 
\begin{eqnarray}\label{1-probaiblity}
\P[X^h_n=Y^h_n|X^h_{n-1}=x_0, Y^h_{n-1}=y_0]\ge a_1>0.
\end{eqnarray}
Thus, for the 
$\alpha_0$-maximal coupling, 
by setting  
$\tilde a_1:=\min\{a_1,\al_0\}<1$, \eqref{1-probaiblity} yields 
\begin{eqnarray*}
\P[X^h_n=Y^h_n|X^h_{n-1}=x_0, Y^h_{n-1}=y_0]\ge \tilde a_1>0, 
\end{eqnarray*}
and hence 
\[\P[|X^h_{n}-Y^h_{n}|>0|X^h_{n-1}=x_0, Y^h_{n-1}=y_0]\le 1-\tilde a_1:=\ga.\]
In particular, $\ga\in(0,1)$ is independent of $h$ and $n.$ 

\medskip
(ii) Let $X_t$ and $Y_t$ be solutions of \eqref{SDE1} 
with the initial condition $X_0=x_0$ and $Y_0=y_0$. 
Under the independent coupling, where the noise terms
driving $X_t$ and $Y_t$ are independent, the following  holds
\begin{eqnarray}\label{independent_est}
\E_{\text{ind}}[|X_h-Y_h|^2]\le 6\var^2h.
\end{eqnarray}
To verify \eqref{independent_est}, apply Dynkin's formula to obtain
\begin{eqnarray*}
\E_{\text{ind}}[|X_h-Y_h|^2]&=&|x_0-y_0|^2+\E_{\text{ind}}\Big[\int_0^h
 \big( -2\langle\nabla U(X_s)-\nabla U(Y_s), X_s-Y_s\rangle+2\var^2 \big)
ds\Big],
\end{eqnarray*}
which, together with the strong convexity of $U$ in \eqref{def:unif_convex}, 
leads to
\begin{eqnarray*}
\E_{\text{ind}}[|X_h-Y_h|^2]\le|x_0-y_0|^2+2\var^2h. 
\end{eqnarray*}
Since  $|x_0-y_0|\le2\var\sqrt{h},$ it follows that $|x_0-y_0|^2\le4\var^2 h,$ and thus 
\eqref{independent_est} follows. 

Now, consider $\E[|X^h_{n}-Y^h_{n}|\big|X^h_{n-1}=x_0,Y^h_{n-1}=y_0]$
under the ($\al_0$-)maximal coupling. 
Applying Proposition \ref{prop:max<ind} with $\alpha=\alpha_0,$ 
together with the bound in  
\eqref{independent_est},
it follows that 
\begin{eqnarray*}
	\E[|X^h_{n}-Y^h_{n}|^2\big|X^h_{n-1}=x_0,Y^h_{n-1}=y_0]\le  \dfrac{2}{1-\alpha_0}\E_{\text{ind}}[|X_h-Y_h|^2]\le \dfrac{12}{1-\alpha_0}\var^2h.
\end{eqnarray*}
Applying H\"older's inequality, one obtains
\[\E[|X^h_{n}-Y^h_{n}|\big|X^h_{n-1}=x_0,Y^h_{n-1}=y_0]\le 
c_1\var\sqrt{h}\]
where  $c_1>0$ is a constant independent of $h, n$ and $\var.$ 
This completes the proof of (ii). 
\end{proof}

In concluding this subsection, we remark that the maximal coupling, 
as employed for numerical efficiency, is formulated for discrete-time processes. However, 
the theoretical results  in this paper are presented in the continuous-time setting.
To ensure consistency between the discrete-time numerical scheme and its continuous-time  theoretical counterpart, we assume that when the maximal coupling  is applied, the intermediate values of the  processes between the discrete steps are disregarded. That is, only the values at times $t=nh$ are relevant, and the behavior of the coupling process at times between the discrete steps has no influence on the analysis.

\subsection{An equivalent  
characterization of essential barrier height}
Let $U:D\to\R$ be a multi-well potential satisfying {\bf (U3)}. 
Throughout the paper, 
let the $L$ local minima of $U$ be labeled according to \eqref{ordering}, with $x_1$ being the unique global minimum.

The following proposition provides an equivalent characterization of the essential barrier height $H_U$ defined in \eqref{barrier_height_multi-well}.

\begin{prop}\label{prop:H_U}
Let $U$ be a multi-well potential on $D$ with $L$ local minima $x_i,i=1,\ldots,L.$ Then 
\begin{eqnarray}\label{equivalent_characterize_H_U}
		H_U=\max\nolimits_{2\le i
		\le L}\big\{\Phi(x_i,\mathcal M_{i-1})- U(x_{i})\big\},
	\end{eqnarray}
where $\mathcal M_i$ is defined as in \eqref{ordering}. 
In particular, 
\begin{eqnarray}\label{H_U=2}
H_U=\Phi(x_2,x_1)- U(x_{2}).    
\end{eqnarray}
\end{prop}
\begin{proof}
Since $x_1\in\mathcal M_{i-1}$ for all $i\in\{2,\ldots,L\},$ it follows that 
	\begin{eqnarray}\label{eq:3}
		H_U\ge\Phi(x_i,x_1)-U(x_i)\ge \Phi(x_i,\mathcal M_{i-1})-U(x_{i}),
	\end{eqnarray}
which yields
	\begin{eqnarray}\label{eq:1}
		H_U\ge \max\nolimits_{2\le i\le L}\big\{\Phi(x_i,\mathcal M_{i-1})- U(x_{i})\big\}.
	\end{eqnarray}

It remains to prove that the inequality in \eqref{eq:1} is in fact an equality. 
Suppose, by contradiction, that the inequality is strict; that is, 
	\begin{eqnarray}\label{eq:4}
		\Phi(x_i,\mathcal M_{i-1})-U(x_i)<H_U,\quad\forall i\in\{2,\ldots,L\}.
	\end{eqnarray}
Under this assumption, we {\it claim} that for each $i\in\{2,\ldots,L\}$, one has
\begin{eqnarray}\label{to1_small}
 \Phi(x_i,x_1)-U(x_i)<H_U,  
\end{eqnarray}
which further implies 
 \[H_U=\max\nolimits_{2\le i\le L}\{\Phi(x_i,x_1)-U(x_i)\}<H_U,\]
yielding a contradiction. 

Now, it only needs to prove \eqref{to1_small}. Fix $i_0\in\{2,\ldots,L\}.$ By \eqref{eq:4}, 
there exists $x_{i_1}\in\mathcal M_{i_0-1}$ such that 
\begin{eqnarray}\label{contradiction}
\Phi(x_{i_0},x_{i_1})-U(x_{i_0})=  \Phi(x_{i_0},\mathcal M_{i_0-1})-U(x_{i_0})<H_U.
\end{eqnarray}
Since $x_{i_1}\in\mathcal M_{i_0-1}$, it follows that $i_1<i_0$. Moreover, the ordering in \eqref{ordering} yields
\begin{eqnarray*}
\Phi(x_{i_0},\mathcal M_{i_{0}-1})-U(x_{i_{0}})<\Phi(x_{i_{1}},\mathcal M_{i_0}\backslash x_{i_{1}})-U(x_{i_{1}}).
\end{eqnarray*}
Using the fact that  
$\Phi(x_{i_0},\mathcal M_{i_0-1})=\Phi(x_{i_0},x_{i_{1}})$ and $ \Phi(x_{i_{1}},\mathcal M_{i_0}\backslash x_{i_{1}})\le\Phi(x_{i_{1}},x_{i_0}),$ we obtain 
\begin{eqnarray*}
\Phi(x_{i_0},x_{i_1})-U(x_{i_0})<\Phi(x_{i_{1}},x_{i_0})-U(x_{i_1}).
\end{eqnarray*}
Hence, $U(x_{i_1})<U(x_{i_0}).$

If $i_1=1,$ then \eqref{to1_small} follows directly from \eqref{contradiction}. 
Otherwise, the same argument can be applied recursively: 
for $i_1\in\{2,\ldots,L\},$
there exists $i_2\in\mathcal M_{i_1-1}$ such that 
$\Phi(x_{i_1},x_{i_2})-U(x_{i_1})<H_U$, with $i_2<i_1$ and $U(x_{i_2})<U(x_{i_1}).$ Continuing inductively, a finite sequence of indices 
$
i_0>i_1>\cdots>i_k=1,\ \text{with finite $k\le L,$} 
$
is obtained such that
\begin{eqnarray}\label{U_order}
U(x_{i_k})<\cdots<U(x_{i_0}).
\end{eqnarray}
Hence,
\begin{eqnarray*}
\Phi(x_{i_0},x_1)-U(x_{i_0})&\le&\max\nolimits_{0\le j< k}\Phi(x_{i_j},x_{i_{j+1}})-U(x_{i_0})\\
&\le&\max\nolimits_{0\le j<k}\big\{\Phi(x_{i_j},x_{i_{j+1}})-U(x_{i_j})\big\}<H_U,
\end{eqnarray*}
where the final inequality follows from \eqref{contradiction}
 and \eqref{U_order}. This obtains  \eqref{to1_small}.
    
Since the local minima $x_i$ are labeled according to \eqref{ordering}, identity \eqref{H_U=2} follows directly. 
\end{proof}

\begin{remark}\label{rem:unique_B2}
{\rm 
In fact, $x_2$ is the unique local minimum such that 
\eqref{H_U=2} is satisfied. In other words,
\begin{eqnarray}\label{small_H_U}
\Phi(x_i,x_1)- U(x_{i})<H_U,\quad \forall i>2.   
\end{eqnarray}
To see this, suppose for the sake of contradiction that there exists $i_0>2$ such that \eqref{small_H_U} does not hold. Then it follows that 
$\Phi(x_{i_0},x_1)- U(x_{i_0})=H_U,$ as 
it always holds that $\Phi(x_{i_0},x_1)- U(x_{i_0})\le H_U.$ 
Since, by \eqref{ordering},   $\Phi(x_{i_0},\mathcal M_{i_0-1})-U(x_{i_0})<H_U,$ there exists an index $1<i_1<i_0$  such that 
\begin{eqnarray*}
\Phi(x_{i_0},x_{i_1})-U(x_{i_0})=\Phi(x_{i_0},\mathcal M_{i_0-1})-U(x_{i_0})<H_U.
\end{eqnarray*}
This leads to a contradiction, as this argument can be applied repeatedly until eventually arriving at $i_k=1$ for some finite $k.$
}
\end{remark}

\subsection{Multi-well potential and first hitting time}\label{sec:hitting time}
Given a multi-well potential 
$U:D\to\R$ satisfying  {\bf (U3)}, let $\bs Z=\{Z_t;t\ge0\}$ be a solution of \eqref{SDE1} and $A\subseteq D$ be a subset. Denote the {\it first hitting time} of $Z_t$ to $A$ as
\begin{eqnarray}\label{kappa}
\kappa_{\bs Z}(A)=\inf\{t>0:Z_t\in A\}.
\end{eqnarray}

It is well known, from large deviation theory,  that the first hitting time from a local minimum  $x_i$ to an appropriate subset is asymptotically exponentially distributed, with the exponent determined  by the associated (i.e., {\it local}\ ) barrier heights 
(\cite{day1983,freidlin1998random,metastability2004, metastability2005}).
The essential barrier height $H_U$ plays  a similar role in a {\it global} sense, characterizing the first hitting time to the basin of the global minimum from any local minimum.

\begin{lemma}\label{lem:expon_tail_multi_well}
	Let $\bs Z=\{Z_t;t\ge0\}$ be a solution of \eqref{SDE1} with initial condition $Z_0=z.$
	Then for any $t>0$ and any $\var>0$ sufficiently small, 
	\begin{eqnarray}\label{eq:7}
		\P_z[\ka_{\bs Z}(B_1)>t]\le A_{z,\var}\exp\Big\{-C_\var e^{-2H_U/\var^2}t\Big\},
	\end{eqnarray}
	where $C_\var>0$ and $A_{z,\var}>0$ are constants  
    such that the limit
  $\lim_{\var\to0}C_\var$ exists,
and $A_{z,\var}$ depends on both the initial value $z$ and the noise strength $\var$, but is independent of $t$. 
Moreover, if the initial condition $Z_0 $ is fully supported with distribution $\mu$, then 
    \begin{eqnarray}
        \label{eq:8}
        \P_z[\ka_{\bs Z}(B_1)>t]\simeq A_{\mu,\var}\exp\Big\{-C_\var e^{-2H_U/\var^2}t\Big\},
    \end{eqnarray}
where constant $A_{\mu,\var}>0$ depends on both $\mu$ and $\var$.  \end{lemma}

The proof of Lemma \ref{lem:expon_tail_multi_well}  follows closely the approach in \cite{metastability2004, metastability2005}, relying on  estimates for the eigenvalues and eigenfunctions of the generator. 
As it is not directly relevant to the main focus of the paper, the proof is deferred to Appendix \ref{appendix_1}. 

\medskip
Define
\begin{eqnarray}\label{I}
\mathcal{I} =\big\{1 \leq i \leq L : \Phi(x_2, x_i) = \Phi(x_2, x_1)\big\} 
\end{eqnarray}
to be the set of indices corresponding to local minima whose communication height with $x_2$  equals that between $x_2$ and $x_1.$
Let \begin{eqnarray}\label{collection_B1}
\bs B_1 = \bigcup\nolimits_{i \in \mathcal{I}} B_i.
\end{eqnarray}
Clearly, \( B_1 \subseteq \bs B_1 \) since \( 1 \in \mathcal{I} \). 
As argued in Remark \ref{rem:hitting_time}, the conclusion in Lemma \ref{lem:expon_tail_multi_well} remains valid if the set $B_1$ is replaced by the larger set $\bs B_1$, that is, 
\begin{eqnarray}\label{genaralized_8}
\P_z[\ka_{\bs Z}(\bs B_1)>t]\simeq A_{\mu,\var}\exp\Big\{-C_\var e^{-2H_U/\var^2}t\Big\},
\end{eqnarray}
for process $\bs Z$ with initial distribution $\mu$.

The following proposition establishes a property for indices not belonging to $\mathcal I$.
\begin{prop}\label{prop:collection_B1}
For any $j\in\{1,\ldots,L\}\backslash\mathcal I$, 
$\Phi(x_1,x_j)-U(x_1)>H_U$.    
\end{prop}

\begin{proof}
By the definition of $\mathcal I$, it  follows directly from  
\eqref{H_U=2} that 
\begin{eqnarray*}
\Phi(x_2,x_i)-U(x_2)=H_U,\quad \forall i\in\mathcal I.
\end{eqnarray*}
Thus, for  $j\notin\mathcal I,$   it follows that either 
(i) $\Phi(x_2,x_j)-U(x_2)<H_U$, or 
(ii) $\Phi(x_2,x_j)-U(x_2)>H_U$.

First, consider case (i). We claim that 
\begin{eqnarray}\label{j}
    \Phi(x_j,x_1)>\Phi(x_j,x_2).
\end{eqnarray}
Indeed, if \eqref{j} fails, then 
\[\Phi(x_2,x_1)\le\max\{\Phi(x_2,x_j),\Phi(x_j,x_1)\}\le\Phi(x_2,x_j),\] which implies 
\[H_U=\Phi(x_2,x_1)-U(x_2)\le \Phi(x_2,x_j)-U(x_2)<H_U,\]
a contradiction. 
Thus, \eqref{j} holds. It then follows that 
\[\Phi(x_1,x_2)-U(x_1)\le
\max\{\Phi(x_1,x_j),\Phi(x_j,x_2)\}-U(x_1)
=\Phi(x_1,x_j)-U(x_1).\]
Since $\Phi(x_1,x_2)-U(x_1)>\Phi(x_2,x_1)-U(x_2)=H_U$, we conclude that 
\[\Phi(x_1,x_j)-U(x_1)>H_U.\]

Next, consider case (ii). Assume that $\Phi(x_2,x_j)-U(x_2)>H_U$. 
Suppose, by contradiction, that $\Phi(x_1,x_j)-U(x_1)\le H_U.$ Then  
\begin{eqnarray*}
\Phi(x_2,x_j)-U(x_2)&\le&\max\{\Phi(x_2,x_1),\Phi(x_1,x_j)\}-U(x_2)\\
&\le&\max\{\Phi(x_2,x_1)-U(x_2),\Phi(x_1,x_j)-U(x_1)\}\le H_U, 
\end{eqnarray*}
contradicting the assumption that $\Phi(x_2,x_j)-U(x_2)>H_U.$
Hence, $\Phi(x_1,x_j)-U(x_1)>H_U.$
\end{proof}
From the proof of Proposition \ref{prop:collection_B1}, we see that
$i\in\mathcal I$ if and only if 
\begin{align}\label{I_equivalent}
\Phi(x_i,x_1)<\Phi(x_i,x_2)\quad\text{and}\quad
\Phi(x_2,x_i)-U(x_2)\le H_U.
\end{align}
That is, among the local minima that can be reached from $x_2$ via a barrier  not exceeding $H_U,$ 
the set $\mathcal I$ consists precisely of the indices for which the corresponding minima are more  accessible to the global minimum $x_1$ than to the local minimum $x_2.$ Accordingly, each basin  in the collection $\bs B_1$
is referred to as a {\it nearby} basin (relative to $x_1$), whereas  $B_2$ is referred to as the {\it distant} basin. 

\subsection{An upper bound in the form of probability generating function}
Given  $C_0>0,\la_0>1,$ and $m\in\N$, define
\begin{eqnarray}\label{g}
	g(\la; C_0,\la_0,m)=\lambda^m+ C_0\sum_{n=m}^{\infty}(\lambda^{n+1}-\lambda^{n})\la_0^{-n},\quad \la\in\R.
\end{eqnarray}
It is not hard to see that $g(\la; C_0,\la_0,m)<\infty,$ $\forall\lambda\in(1,\la_0)$.

The right-hand side of \eqref{g}  is motivated by the probability generating function. The following 
proposition states that if a  random variable $T$ exhibits exponential decay,  then
$\E[\la^T]$ is bounded above by a certain function $g.$

\begin{prop}\label{prop:prob_generate_fun2} Let $T$ be a  random variable taking positive real values. Assume that  for some constant $t_0>0,$
	\begin{eqnarray}\label{T}
\P[T>t]\le C_0\la_0^{-t},\quad\forall t\ge t_0.
	\end{eqnarray}
Then for any $\lambda\in(1,\la_0),$ it holds that 
\begin{eqnarray*}
\E[\lambda^T]\le g(\la; C_0,\la_0,n_0)<\infty
\end{eqnarray*}
where $n_0=\lfloor t_0\rfloor+1.$ 
In particular, if $t_0\in(0,1),$  then \[\E[\lambda^T]\le g(\la; C_0,\la_0,1)<\infty.\]
\end{prop}
\begin{proof}
Note that
\begin{eqnarray*}
\E[\lambda^T]\le
	\sum_{n=0}^\infty
	\lambda^{n+1}\mathbb P[n<T\le n+1]
\le\la+\sum_{n=1}^\infty(\la^{n+1}-\la^n)\P[T>n],
\end{eqnarray*}
where, for $n_0>1,$ the right-hand side can be rewritten as
\begin{eqnarray*}
\la+\sum\limits_{n=1}^{n_0-1}(\lambda^{n+1}-\la^n)\mathbb P[T>n]+\sum\limits_{n=n_0}^\infty
(\lambda^{n+1}-\lambda^n)\mathbb P[T>n].
\end{eqnarray*}
By \eqref{T}, for any $n\ge n_0$, $\P[T>n]\le C_0\la_0^{-n}.$ It then  follows that 
\begin{eqnarray*}
\E[\la^T]\le\la^{n_0}+
C_0\sum_{n=n_0}^\infty
(\lambda^{n+1}-\la^n)\la_0^{-n}
=g(\lambda;C_0,\lambda_0,n_0)<\infty.
\end{eqnarray*}
\end{proof}

Note that for $m=1,$
\begin{eqnarray}\label{g_limit}
	g(\la;C_0,\la_0,1)\to1,\quad \text{as}\  \la\to1.
\end{eqnarray}
Given $\rho>1,\lambda_0>1,$ and $C_0>0$,  define
\begin{eqnarray}\label{beta}
\beta(\rho;C_0,\la_0)=\min\big\{1,\beta^*\big\},
\end{eqnarray}
where 
\begin{eqnarray}\label{beta1}
\beta^*=
\begin{cases}
\dfrac{-(\lambda_0+\rho+C_0-2)+\sqrt{(\lambda_0+\rho+C_0-2)^2+4(C_0-1)(\la_0-1)(\rho-1)}}{2(C_0-1)(\lambda_0-1)},&\text{if}\ C_0\neq 1,\\
\dfrac{\rho-1}{\la_0+\rho-1},&\text{if}\ C_0=1.
\end{cases}
\end{eqnarray}

\medskip
The following Proposition \ref{prop:prob_generate_fun1} provides a quantitative characterization of the 
approximation in \eqref{g_limit}. 
\begin{prop}\label{prop:prob_generate_fun1}
Given  $\rho>1, \lambda_0>1$ and $C_0>0,$
let $\beta=\beta(\rho;C_0,\la_0)$ be as in \eqref{beta}. Then for any $\la\in\big(1,1+\beta(\la_0-1)\big),$ it holds that 
	\begin{eqnarray*}
		g(\la;C_0,\la_0,1)<\rho.
	\end{eqnarray*} 
	\end{prop}
\begin{proof}
	Write $\la=1+\beta(\la_0-1)$ with $\beta\in(0,1),$ we have 
	\begin{eqnarray*}
	g(\la;C_0,\la_0,1)&=&\la+C_0(\la-1)\sum_{n=1}^\infty(\la/\la_0)^n\\
        &=&(1+\beta(\lambda_0-1))\Big(1+C_0\frac{\beta}{1-\beta}\Big)
	\end{eqnarray*}
To have $g(\la;C_0,\la_0,1)<\rho,$ it suffices
\begin{eqnarray*}
(C_0-1)(\la_0-1)\beta^2+(\lambda_0+C_0+\rho-2)\beta+1-\rho<0.
\end{eqnarray*}
This specifies the definition of $\beta$ in \eqref{beta}. Proposition \ref{prop:prob_generate_fun1} is proved.
\end{proof}

\section{Single-well potential and proof of Theorem \ref{thm:1}}
Throughout this section, $U$ is a  strongly convex single-well potential  satisfying   {\bf (U1)}. Let  $(X_t,Y_t)$ denote an $h$-reflection-maximal coupling of two solutions of  \eqref{SDE1}. 
The threshold  $d=\mathcal O(\var\sqrt{h})$, at which the coupling $(X_t,Y_t)$  switches between the reflection and maximal couplings,  is  set to $2\var\sqrt{h}$. 

Define 
\begin{eqnarray*}
	\tau_h^{(1)}=\inf\Big\{t\ge0:|X_t-Y_t|\in(0,2\var\sqrt{h}],\ {\text{and for some $s\in(0,t),$}}\
	|X_s-Y_s|>2\var\sqrt{h}\Big\}, 
\end{eqnarray*}
with the convention that  $\tau_h^{(1)}=\infty$ if the set is empty. 	Note that $\tau_h^{(1)}$ is the infimum time at which the distance between $X_t$ and $Y_t$ attains the threshold $d=2\var\sqrt{h}$ from a distance greater than this value.  

Since it is possible for  $|X_t-Y_t|$ to never exceed the threshold $d$ before a successful  coupling occurs, in which case $\tau_h^{(1)}=\infty,$ define
	\begin{eqnarray*}
		\tau_h=\tau_h^{(1)}\wedge\tau_c. 
	\end{eqnarray*}
Note that  $\tau_h<\infty$ holds almost surely.
It  will be shown later that the coupling time $\tau_c$ is almost surely a finite iteration of $\tau_h$.

	\subsection{Estimation of $\tau_h$.} 
	In this subsection,  estimates of $\tau_h$ are provided under the two  initial conditions $|X_0-Y_0|>2\var\sqrt{h}$ and 
 $|X_0-Y_0|\le2\var\sqrt{h},$ respectively. 
		
If $|X_0-Y_0|>2\var\sqrt{h},$ then $(X_t,Y_t)$ remains a reflection coupling until $t=\tau_h^{(1)}$, when  $(X_t,Y_t)$  switches to the maximal coupling. Proposition \ref{prop:sup_mart} immediately yields the following. 
	
\begin{lemma}\label{lem:1}
		Assume $|X_0-Y_0|/2\var=r_0>\sqrt{h}.$
		Then $\tau_h=\tau_h^{(1)}$ holds $\P{\text{-a.s.}}$, and for any $t_0>0$, there exists a constant $c_0>0$ such that		\begin{eqnarray}\label{ineq:4}
			\P[\tau_h>t]\le  c_0r_0e^{-m_0t},\quad \forall t\ge t_0.
		\end{eqnarray}
In particular, by letting $0<t_0<1$ and applying Proposition \ref{prop:prob_generate_fun2}, for any $\la\in(1,e^{m_0}),$
\begin{eqnarray}\label{ineq:5}
			\E[\lambda^{\tau_h}]\le g(\lambda; c_0r_0, e^{m_0}, 1).
\end{eqnarray}
\end{lemma}
	\medskip

\begin{remark}\label{rem:approximate} 
{\rm The estimation in \eqref{ineq:4} is for the continuous-time process instead of its time-$h$ sampled chain, which the numerical scheme truly  approximates. 
Let $\tau_h^0$ ({\it resp.} $\tau_h^h$) be the first passage time of the coupling process  $(X_t,Y_t)$  ({\it resp.} its time-$h$ sampled chain $(X_n^h,Y_n^h)$) to the set $\{(x,y)\in\R^k\times\R^k:|x-y|\le2\var\sqrt{h}\}.$ It is obvious that $\tau_h^h\ge\tau_h^0,$ 	and it is intuitive that their difference, which is generally difficult to theoretically estimate,  should  approach to zero as $h$ tends to zero, i.e.,  
		\begin{eqnarray}\label{approximate}
			\lim\nolimits_{h\to0}(\tau_h^h-\tau^0_h )=0,\quad\text{$\P$-a.s.}
		\end{eqnarray}
	Throughout this section, \eqref{approximate} is  always assumed and will be numerically verified  in Section 5  for the example of symmetric quadratic potential functions. Therefore,  the estimation \eqref{ineq:4}   applies to the time-$h$ sampled chain $(X^h_n,Y^h_n)$ (with a possible slight enlargement of $c_0$  if necessary) whenever $h$ is sufficiently small. }
\end{remark}

 The analysis becomes more intricate  for the initial condition $|X_0-Y_0|\le 2\var\sqrt{h}$, as the coupling method between $X_t$ and $Y_t$ may switch during the time interval $(0,\tau_h).$  Specifically, there exists $n>0$
such that the coupling between the time-$h$ sampled chains $X_{ih}$ and $Y_{ih}$ remains a maximal coupling for $0\le i<n$. At the step $i=n$, either 
 $X_{nh}=Y_{nh}$, indicating a successful coupling, or  $|X_{nh}-Y_{nh}|>2\var\sqrt{h}.$ 
In the former case, $\tau_h=\tau_c$; in the latter, $\tau_h=\tau_h^{(1)}$, and $(X_t,Y_t)$ evolves under a  reflection coupling until the condition $|X_t-Y_t|\le2\var\sqrt{h}$ is satisfied again.

\begin{lemma}\label{lem:2}
Given any $t_1>0$, there exist $h_0>0$ and $C_0>0$ such that  for all $h\in(0,h_0),$ 	if $|X_0-Y_0|\le2\var\sqrt{h},$ then
	\begin{eqnarray*}
		\P[\tau_h>t]\le C_0\sqrt{h} e^{-m_0t},\quad\forall t\ge t_1.
	\end{eqnarray*}
In particular, by choosing $0<t_1<1$ and applying Proposition \ref{prop:prob_generate_fun2}, for any $\la\in(1,e^{m_0})$ and $h>0$ sufficiently small,	\begin{eqnarray}\label{ineq:10_2}
		\E[\lambda^{\tau_h}]\le g(\lambda; C_0\sqrt{h}, e^{m_0},1)<\infty.
	\end{eqnarray}
\end{lemma}
\begin{proof}
	Recall from the proof of Proposition \ref{prop:sup_mart} that the process $R_t=|X_t-Y_t|/2\var$ is a one-dimensional stochastic process  induced by the coupling $(X_t,Y_t)$. 
Let $n=\lfloor t/h\rfloor\in\N$. Based on the coupling behaviors between $X_t$ and $Y_t$ before  the stopping time $\tau_h,$ one has
	\begin{eqnarray*}
\P[\tau_h>t]&\le&\P[\tau_h>nh]\\
&=&\sum_{j=1}^n\P[R_{ih}\in(0,\sqrt{h}], 0\le i\le j-1]\\
		&&\cdot\ \Big(\P[R_{jh}>\sqrt{h}|R_{(j-1)h}\in(0, \sqrt{h}]]
		\cdot\P[\tau_h^{(1)}\circ\theta^{jh}> t-jh|R_{jh}> \sqrt{h}]\Big)\\
		&&+\ \P[R_{ih}\in(0,\sqrt{h}], 0\le i\le  n],
	\end{eqnarray*}
where $\theta$ is the usual shift operator. 	

For any $i\ge1,$ since $(X_{ih},Y_{ih})$ is a maximal coupling whenever $|X_{(i-1)h}-Y_{(i-1)h}|\le2\var\sqrt{h},$ it follows from Lemma \ref{lem:max_coupling} (i) that 
\[\P\big[R_{ih}>0|R_{(i-1)h}\in(0,\sqrt{h}]\big]\le\ga,\] where $\ga\in(0,1)$ is independent of $i$ and $h.$
Thereofor, by the Markov property, for any $2\le j\le n,$
	\begin{eqnarray*}
		\P\big[R_{ih}\in(0,\sqrt{h}], 0\le i\le j-1\big]=\prod_{i=1}^{j-1}\P\big[R_{ih}\in(0,\sqrt{h}]|R_{(i-1)h}\in(0,\sqrt{h}]\big]\le \gamma^{j-1},
	\end{eqnarray*}
and this also holds trivially for $j=1$.
Consequently,
	\begin{eqnarray}
\P[\tau_h>t]
\le\sum_{j=1} ^{n}\gamma^{j-1}\Big(\P[R_{jh}>\sqrt{h}|R_{(j-1)h}\in(0, \sqrt{h}]]
		\cdot\P[\tau_h\circ\theta^{jh}> t-jh|R_{jh}> \sqrt{h}]\Big)
		+\ga^n	\label{ineq:1}.
	\end{eqnarray}

Now, for $1\le j\le n,$ consider estimating
\begin{eqnarray}\label{expression:1}
		\P[R_{jh}> \sqrt{h}|R_{(j-1)h}\in(0, \sqrt{h}]]
		\cdot\P[\tau_h\circ\theta^{jh}>t-jh|R_{jh}> \sqrt{h}],
	\end{eqnarray}
which equals
	\begin{eqnarray}\label{eq:P}	\int_{\sqrt{h}}^{\infty}\P\big[\tau_h\circ\theta^{jh}>t-jh|R_{jh}= r\big]\P\big[R_{jh}=dr|R_{(j-1)h}\in(0, \sqrt{h}]\big].
	\end{eqnarray}

Fix $t_0\in(0,t_1).$ Then for any $1\le j\le \lfloor\frac{t-t_0}{h}\rfloor$, it holds that $t-jh\ge t_0$.
Since $R_0=r>\sqrt{h}$, Lemma \ref{lem:1} and the Markov property implies that 
\begin{eqnarray*}
\P\big[\tau_h\circ\theta^{jh}>t-jh|R_{jh}= r\big]
=\P[\tau_h>t-jh|R_0=r]\le c_0re^{-m_0(t-jh)},
\end{eqnarray*}
where $c_0>0$ is the constant given in Lemma \ref{lem:1}.
Therefore, for any such $j,$
\begin{eqnarray*}
	\eqref{eq:P} &\le& c_0e^{-m_0(t-jh)}
		\int_{\sqrt{h}}^\infty r\P[R_{jh}=dr|R_{(j-1)h}\in(0, \sqrt{h}]]\\
	 &\le&c_0e^{-m_0(t-jh)}\cdot\E[R_{jh}|R_{(j-1)h}\in(0,\sqrt{h}]].
	\end{eqnarray*}
Since $R_{(j-1)h}\le \sqrt{h},$ and hence $|X_{(j-1)h}-Y_{(j-1)h}|\le2\var\sqrt{h}$, Lemma \ref{lem:max_coupling} (ii) implies that for sufficiently small $h>0$, 
	\begin{eqnarray*}
		\E\big[R_{jh}\big|R_{(j-1)h}\in(0,\sqrt{h}]\big] &=&\frac{1}{2\var}\E\big[|X_{jh}-Y_{jh}|\big||X_{(j-1)h}-Y_{(j-1)h}|
		\le2\var\sqrt{h}\big]\\
		&\le&\frac{1}{2}c_1\sqrt{h}.
	\end{eqnarray*}
Thus,
	\begin{eqnarray}\label{ineq:45}
	\eqref{expression:1}\le  C_0\sqrt{h}e^{-m_0(t-jh)},\quad\ \forall 1\le j\le \lfloor\frac{t-t_0}{h}\rfloor,
		\end{eqnarray}
for some constant $C_0>0$ independent of $h, \var,$ and $j.$ Moreover, 
\begin{eqnarray}\label{ineq:46}
\eqref{expression:1}\le1,\quad\ \text{$\lfloor\frac{t-t_0}{h}\rfloor< j\le \lfloor \frac{t}{h}\rfloor:=n.$}
\end{eqnarray}

Combining \eqref{ineq:1}, \eqref{ineq:45}, and \eqref{ineq:46},  it follows that
\begin{eqnarray}
\P[\tau_h>t]&\le&
C_0\sqrt{h}\sum_{j=1}^{\lfloor \frac{t-t_0}{h}\rfloor}\gamma^{j-1}e^{-m_0(t-jh)}+\sum_{j=\lfloor \frac{t-t_0}{h}\rfloor}^{n}\gamma^{j}\nonumber\\
&\le&C_0\sqrt{h}\dfrac{e^{m_0h}}{1-\ga e^{m_0h}}e^{-m_0t}+\dfrac{\ga^{\lfloor \frac{t-t_0}{h}\rfloor}}{1-\ga}\label{ineq:7}
\end{eqnarray}
Let $0<h_0\le |\ln\gamma|/m_0$ be sufficiently small so that for any $h\in(0,h_0),$ 
	\begin{eqnarray*}
		\ga^{\lfloor \frac{t-t_0}{h}\rfloor}\le\sqrt{h}e^{-m_0t},\quad\forall t\ge t_1.
	\end{eqnarray*}
Since ${e^{m_0h}}/(1-\ga e^{m_0h})\to 1/(1-\gamma)$ as $h\to0,$ by enlarging  $C_0$ in \eqref{ineq:7} if necessary, it follows that
\begin{eqnarray*}
	\P[\tau_h>t]\le C_0\sqrt{h}e^{-m_0t},\quad\forall t\ge t_1.
\end{eqnarray*}
\end{proof}

Combining Lemma \ref{lem:1} and Lemma \ref{lem:2}, the following holds. 
\begin{lemma}\label{lem:3}
	Assume  $\E[|X_0-Y_0|]<\infty.$ Then for any $h\in(0,h_0)$ where $h_0>0$ is as in Lemma \ref{lem:2}, for any $\la\in(1,e^{m_0}),$ it holds that 
	\begin{eqnarray*}
		\E[\la^{\tau_h}]<\infty. 
	\end{eqnarray*}
\end{lemma}
\begin{proof}
Recall  the  one-dimensional stochastic process $R_t=|X_t-Y_t|/(2\var)$, $t\ge0.$ 
Let $\E_r[\cdot]$ denote the expectation with respect to the initial condition $R_0=r.$ Then
	\begin{eqnarray}\label{ineq:two terms}
		\E[\la^{\tau_h}]&=&\int_0^ {\sqrt{h}}\E_r[\la^{\tau_h}]\mu(dr)+\int_{\sqrt{h}}^\infty\E_r[\la^{\tau_h}]\mu(dr)
	\end{eqnarray}
where $\mu$ denotes the distribution  of $R_0$. 

By Lemma \ref{lem:2}, for sufficiently small $h>0$ , the first term
on the right-hand side 
of \eqref{ineq:two terms} satisfies
\begin{eqnarray}\label{term1}
	\int_0^ {\sqrt{h}}\E_r[\la^{\tau_h}]\mu(dr)\le\int_0^{\sqrt{h}} g(\la; C_0\sqrt{h}, e^{m_0},1)\sqrt{h}\mu(dr)\le g(\la; C_0\sqrt{h},e^{m_0},1)\sqrt{h}. 
\end{eqnarray}
By \eqref{ineq:5}, the second term on the right-hand side of \eqref{ineq:two terms} is bounded as
\begin{eqnarray}
\int_{\sqrt{h}}^\infty\E_r[\la^{\tau_h}]\mu(dr)
&\le&\int_{\sqrt{h}}^\infty g(\lambda; c_0r, e^{m_0},1)\mu(dr)\nonumber\\
&=&
g(\la; c_0\int_{\sqrt{h}}^\infty r\mu(dr),e^{m_0},1)\le g(\la; c_0\E[R_0],e^{m_0},1)\label{term2}.
\end{eqnarray}

Combining \eqref{term1} and \eqref{term2}, we have
\begin{eqnarray*}
\E[\la^{\tau_h}]\le g(\la; C_0\sqrt{h},e^{m_0},1)\sqrt{h}+g(\la; c_0\E[R_0],e^{m_0},1).
\end{eqnarray*}	
Since $\E[R_0]=\E[|X_0-Y_0|]/2\var<\infty,$ the lemma is proved. 
\end{proof}

\subsection{Iteration of $\tau_h$ and  coupling times}
The coupling time $\tau_c$ is in fact a finite iteration of $\tau_h.$ To see this, define
\begin{eqnarray*}
	\tau_h^0=0,\quad \tau_h^k=\tau_h^{k-1}+\tau_h\circ\theta^{\tau_h^{k-1}},\  k\ge1
\end{eqnarray*} 
where $\theta$ is the usual shift operator, and let
\begin{eqnarray*}
	\eta=\inf\big\{k\ge1: X_{\tau_h^k}=Y_{\tau_h^k}\big\}.
\end{eqnarray*}

The following proposition immediately follows from the definition of $\tau_h.$ 
\begin{prop}\label{prop:1}
	Given any $h>0$ and $k\ge1.$ The following hold: 
	\begin{itemize}	
		\item[(i)]  $|X_{\tau_h^k}-Y_{\tau_h^k}|=2\var\sqrt{h}\ \text{or}\ 0,$ where $X_{\tau_h^k}=Y_{\tau_h^k}$ if and only if $k\ge\eta;$
		\item[(ii)] If  $k>1,$ then \[\P[|X_{\tau_h^k}-Y_{\tau_h^k}|>0| \cF_{\tau_h^{k-1}}]<\gamma.\]
		where $\gamma$ is as in Lemma \ref{lem:max_coupling}. 
		\end{itemize}
\end{prop}

By Proposition \ref{prop:1} (i), 
\begin{eqnarray*}
	\tau_c=\tau_h^\eta,\quad \P\mbox{-a.s.}
\end{eqnarray*}
Hence, the estimation of 
$\tau_c$ is reduced to the estimation of $\tau^\eta_h.$
\begin{theorem}\label{thm:3}
	Assume 
	$\E[|X_0-Y_0|]<\infty.$
	Then for any $\delta>0,$ there exists $ h_0>0$ such that for any $h\in(0,h_0)$ and any $\la\in(1,e^{m_0-\delta}),$ it holds that
	\begin{eqnarray*}
		\E[\lambda^{\tau_h^\eta}]<\infty.
	\end{eqnarray*}
\end{theorem}
\begin{proof}
	The proof follows the approach of Lemma 2.9 in \cite{nummelin1982}. Note that
	\begin{eqnarray}
		\E[\lambda^{\tau_h^\eta}]
		&\le&\sum_{k=1}^\infty\E[\lambda^{\tau_h^k}\I_{\eta\ge k}]\\
		&=&\E\big[\lambda^{
			\tau_h}\I_{\eta\ge1}\big]+\sum_{k=2}^\infty\E\big[\I_{\eta\ge k}\lambda^{\tau_h^{k-1}}\E[\lambda^{\tau_h\circ\theta^{\tau_h^{k-1}}}|\cF_{\tau_h^{k-1}}]\big],\nonumber
	\end{eqnarray}
	where the last equality follows from the fact that $\lambda^{\tau_h^{k-1}}\in\cF_{\tau_h^{k-1}}$  and $\{\eta\ge k\}\in\cF_{\tau_h^{k-1}}.$ 
	
We retain the notation $R_t=|X_t-Y_t|/2\var$ for $t\ge0$, and let $\E_{r}$ denote the expectation with respect to the initial condition $R_0=r.$
	By Proposition \ref {prop:1} (i),  $R_{\tau_h^{k-1}}=\sqrt{h}$ for $2\le k\le\eta.$ By  \eqref{ineq:10_2} and the strong Markov property, 
	\begin{eqnarray*}		\E\big[\lambda^{\tau_h\circ\theta^{\tau_h^{k-1}}}|\cF_{\tau_h^{k-1}}\big]\le\E_{\sqrt{h}}[\lambda^{\tau_h}]\le g(\lambda;C_0\sqrt{h},e^{m_0},1)<\infty,\quad \forall k\ge1,
	\end{eqnarray*}
	where $C_0>0$ is as in Lemma \ref{lem:2}.
	Thus, 	\begin{eqnarray}\label{ineq:9}		\E[\lambda^{\tau_h^\eta}]\le\E[\lambda^{
    \tau_h}\I_{\eta\ge1}]+  g(\lambda;C_0\sqrt{h},e^{m_0},1)\sum_{k=2}^\infty\E[\I_{\eta\ge k}\lambda^{\tau_h^{k-1}}].
	\end{eqnarray}	
	
	Now, for  $k\ge2,$ we estimate $\E[\I_{\eta\ge k}\lambda^{\tau_h^{k-1}}]$. 
Write
	$\I_{\eta\ge k}=\I_{\eta\ge k-1}\I_{R_{\tau_h^{k-1}}>0}$.
	Then we have
	\begin{eqnarray}
		\E\big[\I_{\eta\ge k}\lambda^{\tau_h^{k-1}}\big]
		&=&\E\big[\I_{\eta\ge k-1}\la^{\tau_h^{k-2}}\E\big[\I_{R_{\tau_h^{k-1}}>0}\lambda^{\tau_h\circ\theta^{\tau_h^{k-2}}}|\cF_{\tau_h^{k-2}}\big]\big]\nonumber\\
		&=&\E\big[\I_{\eta\ge k-1}\la^{\tau_h^{k-2}}\E_{R_{\tau_h^{k-2}}}[\I_{R_{\tau_h}>0}\lambda^{\tau_h}]\big]\label{eq:E}.
	\end{eqnarray}
	Note that for $k=2,$
    \begin{eqnarray}\label{k=2}
        \E_{R_{\tau_h^{k-2}}}[\I_{R_{\tau_h}>0}\lambda^{\tau_h}]=\E_{R_0}[\I_{R_{\tau_h}>0}\lambda^{\tau_h}]\le\E[\la^{\tau_h}].
    \end{eqnarray}
 For $k>2,$ since  $R_{\tau_h^{k-2}}=\sqrt{h}$, the strong Markov property implies that \begin{eqnarray}\label{ineq:E}
		\E_{R_{\tau_h^{k-2}}}[\I_{R_{\tau_h}>0}\lambda^{\tau_h}]\le		
			\E_{\sqrt{h}}[\I_{R_{\tau_h}>0}\lambda^{\tau_h}]
			\end{eqnarray}
By  the H\"older's inequality, for any $p\in(0,1),$
	\begin{eqnarray*}
		\E_{\sqrt h}[\I_{R_{\tau_h}>0}\lambda^{\tau_h}]&\le&(\E_{\sqrt{h}}[\I_{R_{\tau_h}>0}])^{1-p}\cdot(\E_{\sqrt{h}}[\lambda^{\tau_h/p}])^{p}\\
        &=&(\P_{\sqrt{h}}[{R_{\tau_h}>0}])^{1-p}\cdot(\E_{\sqrt{h}}[\lambda^{\tau_h/p}])^{p}.	
	\end{eqnarray*}
	Then it follows from  Proposition \ref{prop:1} (ii) and \eqref{ineq:10_2} that
\begin{eqnarray}\label{ineq:holder}
\E_{\sqrt{h}}[\I_{R_{\tau_h}>0}\lambda^{\tau_h}]\le\gamma^{1-p}g(\la^{1/p}; C_0\sqrt{h},e^{m_0},1)^{p}.
	\end{eqnarray}	

Substituting \eqref{k=2}-\eqref{ineq:holder} into \eqref{eq:E}, we obtain
\begin{eqnarray*}
\E[\I_{\eta\ge k}\lambda^{\tau_h^{k-1}}]\le
\begin{cases}
\E[\la^{\tau_h}], & k=2\\
\gamma^{1-p}g(\lambda^{1/p}; C_0\sqrt{h},e^{m_0},1)^p\cdot\E[\I_{\eta\ge k-1}\lambda^{\tau_h^{k-2}}],  & k>2.
\end{cases}
\end{eqnarray*}
By induction, for $k\ge2,$
	\begin{eqnarray*}
		\E[\I_{\eta\ge k}\lambda^{\tau_h^{k-1}}]\le\gamma^{(1-p)(k-2)}g(\lambda^{1/p}; C_0\sqrt{h},e^{m_0},1)^{p(k-2)}\cdot\E[\la^{\tau_h}]
	\end{eqnarray*}
	Therefore,  \eqref{ineq:9}  yields 
	\begin{eqnarray*}		\E[\lambda^{\tau_h^\eta}]\le \E[\la^{\tau_h}]\Big(1+ g(\lambda;C_0\sqrt{h},e^{m_0},1)\sum_{k=2}^\infty \Big(\gamma^{1-p}g(\lambda^{1/p}; C_0\sqrt{h},e^{m_0},1)^{p}\Big)^{k-2}\Big)
	\end{eqnarray*}		

By Lemma \ref{lem:3}, $\E[\la^{\tau_h}]<\infty.$ Thus, to guarantee $\E[\lambda^{\tau_h^\eta}]<\infty$, it suffices
\begin{eqnarray}\label{ineq:6}
		g(\la^{1/p};C_0\sqrt{h},e^{m_0},1)<\gamma^{-(1-p)/p}
	\end{eqnarray}
By Proposition \ref{prop:prob_generate_fun1}, \eqref{ineq:6} holds for any $\la>1$ satisfying $\la^{1/p}\in\big(1,1+\beta(e^{m_0}-1)\big)$,
where \[\beta=\min\{1,\beta^*\},\] and $\beta^*$ is given by
\eqref{beta1} with 
$\lambda_0=e^{m_0}$, $\rho=\ga^{-(1-p)/p},$ and $C_0$ replaced by $C_0\sqrt{h}.$

Note that  
$\beta^*\to1,$ and hence $\beta\to1,$ as $h\to0$.
Since $p$ can be  arbitrarily close to $1$,  by choosing $h>0$  sufficiently small, 
 it follows that 
	\[\E[\lambda^{\tau_h^\eta}]<\infty,\quad\forall \la\in(1,e^{m_0-\de}),\] 
where  $\de>0$ is arbitrarily small. 
\end{proof}

\medskip

The proof of Theorem \ref{thm:1} is  now straightforward. 

\medskip

\noindent{\it Proof of Theorem \ref{thm:1}}: 
Note that for any $\la\in(1,\infty)$ satisfying $\E[\la^{\tau_h^\eta}]<\infty,$ we have 
\begin{eqnarray*}
	\P[\tau_h^\eta>t]\le \E[\la^{\tau_h^\eta}]\la^{-t},\quad\forall t>0.
\end{eqnarray*}
By Theorem \ref{thm:3}, for any $\la\in(1,e^{m_0-\delta})$ and any $\delta>0,$
\begin{eqnarray*}
	\limsup_{t\to\infty}\dfrac{1}{t}\log\P[\tau_h^\eta>t]\le 
	-m_0+m_0\delta.
\end{eqnarray*}

Theorem \ref{thm:1} is proved by  taking $\delta$ as $m_0\delta.$

\section{Multi-well potentials and proof of Theorem \ref{thm:2}-Theorem \ref{thm:5}}\label{sec:double-multiple well} Throughout this section, let $U$ be a multi-well potential, and let $(X_t,Y_t)$ be a coupling of two solutions of \eqref{SDE1}. 
Section 4.1-4.3 focus primarily on the  double-well potential  $U$, while the general case of more than two wells 
 is discussed in Section 4.4. 
 
\subsection{Key stopping times for double-well potential}\label{subsec:stopping_time}
Assume that $U$ is a double-well potential satisfying  {\bf (U2)} with two basins $B_1$ and $B_2$. 
Let
\begin{eqnarray*}
\tau^{(1)}_\var=\inf\Big\{t>h:(X_t,Y_t)\in B_1\times B_1\ \text{or}\ B_2\times B_2\Big\}
\end{eqnarray*}
denote the infimum time when $X_t$ and $Y_t$ lie in the {\it same} basin of $U.$ Here, the subscript ``$\var$" emphasizes the  role of the noise magnitude $\var$ in determining the stopping times in the multi-well setting. We note that $\tau^{(1)}_\var$ is finite   
$\P$-almost surely. 

\begin{remark}
{\rm 
When $X_t$ or $Y_t$ is initiated near a basin boundary, 
repeated boundary crossings within an infinitesimal time interval may occur, making the analysis cumbersome. 
To circumvent these non-essential complications, $\tau_\var^{(1)}$ is defined after a small positive time. Specifically, this positive time is chosen as the numerical step size $h$ to  ensure compatibility with the numerical simulations. This convention is adopted for all stopping times defined in this section.  
}
\end{remark}

If initially,  $X_t$ and $Y_t$  already belong to the same basin, then $\tau_{\var}^{(1)} = h$ with probability close to $1$. Now, assume that $X_0$ and $ Y_0$    belong to  different basins. Without loss of generality, let $Y_0\in B_1.$ Then we have
  \begin{eqnarray*}
			\tau_\var^{(1)}=\ka_{\bs X}(B_1)\wedge\ka_{\bs Y}(B_2),
		\end{eqnarray*}
 where recall that 
$\ka_{\bs X}(B_1)$  ({\it resp.} $\ka_{\bs Y}(B_2)$) denotes the first hitting time of the process $X_t$  ({\it resp.} $Y_t$) to the basin $B_1$ ({\it resp.} $B_2$). 

Throughout this section, let\begin{eqnarray}\label{la_var}
\la_\var=\exp\big\{C_\var e^{{-2H_U}/{\var^2}}\big\}, \end{eqnarray} 
where, in the double-well case, $H_U$ is the essential barrier height defined in \eqref{barrier_height},
and 
$C_\var>0$ is any constant that is not uniquely determined  and satisfies $\lim_{\var\to0}C_\var>0$, with the limit depending only on $U.$ By Lemma \ref{lem:expon_tail_multi_well}, 
\begin{eqnarray}\label{ineq:11}
\P[\tau_\var^{(1)}>t]\leq\P[\ka_{\bs X}(B_1)>t]\lesssim \la_\var^{-t},\quad\forall t>0. 
\end{eqnarray}

\medskip

Before proceeding, recall from Section \ref{sec:hitting time}  that for the multi-well potential, an enlarged set $\bs B_1$, defined in \eqref{collection_B1}, 
consists of all nearby basins,  in particular, including $B_1.$  
By Proposition \ref{prop:collection_B1}, for any local minimum $x_j$ 
with the corresponding basin $B_j$ not contained in $\bs B_1,$ one has 
\begin{eqnarray}\label{eq:11}
\Phi(x_1,x_j)-U(x_1)>\Phi(x_2,x_i)-U(x_2)=H_U,\quad\forall i\in\mathcal I.
\end{eqnarray}
This suggests that a process starting from the global minimum $x_1$ must overcome a higher barrier to exit the enlarged set  $\bs B_1$ than a process  starting from $x_2$ has to overcome to enter it. 

\medskip
In light of \eqref{eq:11}, in the multi-well setting, which in particular  includes the double-well case, we assume the following 
 {\bf (H1)}  for the coupling scheme.
 
\medskip

\noindent{\bf (H1)}
There exist constants $\delta_0>0$ and $\ga_0>0$ such that if
the coupling $(X_t,Y_t)$ satisfies the initial condition  $X_0\in B_{\delta_0}(x_2), Y_0\in B_{\delta_0}(x_1),$ then 
for any $\var>0$ sufficiently small, 
\begin{eqnarray}\label{ineq:14}
	\P\big[Y_s\in \bs B_1\ \text{for all}\  s\in[0,t]\big| \ka_{\bs X}(\bs B_1)>t\big]\ge\ga_0,\quad\forall t>0.
\end{eqnarray}

\vspace{1.5mm}

The assumption {\bf (H1)} states  that when $X_t$ and $Y_t$ start from the bottom of the basins $B_2$ and  $B_1$, respectively, 
the probability that $Y_t$  remains in the enlarged set $\bs {B}_1$, given that $X_t$ has not yet entered $\bs {B}_1$,  is uniformly positive and independent of $\var$ and $t$. In Section 5, {\bf (H1)} is numerically verified. At present, although  {\bf (H1)}, along with the forthcoming assumptions {\bf (H2)}-{\bf (H3)}, cannot be rigorously verified, Section~\ref{numerical verification} provides numerical evidence supporting their validity in an interacting particle system with multiple local minima. 

\begin{remark}
{\rm 
The assumption {\bf (H1)} naturally arises in the context of reflection coupling. For simplicity, consider the one-dimensional  double-well potential. Let $\varphi(t), t \in [0, T],$ be a $C^1$ function satisfying $\varphi(0)=0$, and let $\delta>0$ be a small constant such that 
\begin{eqnarray*}
|(X_t - X_0) - \varphi(t)| < \delta, \quad\forall t\in [0,T].
\end{eqnarray*}    
Then the Wiener process $B^x_t$ associated with  $X_t$ must stay in the neighborhood of  $\varphi
(t)+\int_0^t\nabla U(X_0 + \varphi(s))ds$. Due to the reflection, the Brownian motion terms in $X_t$ and $Y_t$ are symmetric. Consequently, the corresponding Wiener process $B^y_t$ of $Y_t$ must stay in the small neighborhood of 
$-\varphi
(t) - \int_0^t\nabla U(X_0 + \varphi(s))ds,$
which has the same action functional as $\varphi (t)+\int_0^t\nabla U(X_0 + \varphi(s))ds.$
If $X_t$ exits the shallower well $B_2$ at some $T > 0$, 
then with high probability,
the final segment of $X_t$ remains in a small neighborhood of the minimum energy path, denoted by $\phi(t)$ (see, for instance, Theorem 2.3 in Chapter 4 of  \cite{freidlin1998random}). This implies that the Brownian motion term $B^x_t$ stays in the neighborhood of \[\phi(t) + \int_0^t\nabla U(\phi(s)) \mathrm{d}s,\] whose action functional equals $2 (\Phi(x_1, x_2) - U(x_2) )$, which is strictly less than $2 H_U$. On the other hand, for $Y_t$ to exit $B_1$ from the neighborhood of $x_1$, its trajectory must have an action functional of at least $2 H_U$. Therefore, when $X_t$ exits from the shallower well, it is highly likely that $Y_t$ remains in $B_1$. 
Unfortunately, to the best of our knowledge, this argument is difficult to establish  rigorously, as the Freidlin-Wentzell large deviation theory applies only to a {\it fixed} time span $[0, T]$ as $\varepsilon \rightarrow 0$. However, the tail estimates required in this paper necessitate estimates that hold for arbitrarily large $t$. 
}
\end{remark}

In the double-well setting, $\bs B_1$ is simply $B_1$. Hence,  \eqref{ineq:14} is  reduced to 
\begin{eqnarray}\label{ineq:19}
\P[\ka_{\bs Y}(B_2)>t|\ka_{\bs X}(B_1)>t]\ge\ga_0.
\end{eqnarray}
Under {\bf (H1)},  the reverse of \eqref{ineq:11} holds if $X_0\in B_{\de_0}(x_1)$ and $Y_0\in B_{\de_0}(x_2)$: 
\begin{eqnarray}
\P[\tau^{(1)}_\var>t]&\ge&\P[\ka_{\bs X}(B_1)>t,\ka_{\bs Y}(B_2)>t]\nonumber\\
&=&\P[\ka_{\bs Y}(B_2)>t|\ka_{\bs X}(B_1)>t]\cdot\P[\ka_{\bs X}(B_1)>t]\nonumber\\
&\ge&\ga_0\cdot\P[\ka_{\bs X}(B_1)>t]\ 
\simeq\ \lambda_\var^{-t}\label{ineq:15}
\end{eqnarray}	
where the last ``$\simeq$" follows from Lemma \ref{lem:expon_tail_multi_well}, by choosing $\delta_0>0$  sufficiently small. 

		\medskip
		
In contrast to $\tau^{(1)}_\var,$  define another stopping time 
\begin{eqnarray*}
\tau^{(2)}_\var=&\inf&\Big\{t> h:(X_t,Y_t)\in B_1\times B_2\ \text{or}\ B_2\times B_1, \ \text{and for some}\ s\in(h,t),\\ &&(X_s,Y_s)\in B_1\times B_1 \ \text{or}\  B_2\times B_2\Big\},
\end{eqnarray*}
		and let $\tau^{(2)}_\var=\infty$ if the set is empty.  
		Note that  $\tau^{(2)}_\var$ captures the infimum time when 
		$X_t,Y_t$  are  separated  (again) by {\it different} basins, where ``again" applies if $X_t$ and $Y_t$ 
        already belong to  different basins at the very beginning.
			
Let 
\begin{eqnarray*}
\tau_\var=\tau^{(2)}_\var\wedge\tau_c. 
\end{eqnarray*}
We note that $\tau_\var=\tau^{(2)}_\var<\tau_c$ if $X_t,Y_t$ are {\it not} coupled while staying in the same basin; otherwise, $\tau_\var=\tau_c$ and $\tau_\var^{(2)}=\infty$.  As will be seen in Section 4.3, the coupling time $\tau_c$ is  $\P$-a.s. a finite iteration of $\tau_\var$.

\subsection{Estimation of $\tau_\var$.}\label{subsec:4.2}
The following assumption {\bf (H2)} is  made in both double and  multi-well settings, characterizing  {\it local} coupling properties when $X_t$ and $Y_t$ lie in the same basin. 
		
		\medskip
		
		\noindent{\bf  (H2)} 
		Let $(X_t,Y_t)$ be  a coupling of two solutions of \eqref{SDE1} such that  $(X_0,Y_0)\in\bigcup_{1\le i\le L}\overline{ B_i\times B_i}$. The  following hold: 
		\medskip
		
		(i) There exists $\ga_1 \in(0,1)$ such that \[\P[X_{\tau_\var}\neq Y_{\tau_\var}]<\ga_1;\]
				
(ii)  For any $\var>0$ sufficiently small, there exists $r_0(\varepsilon)=\mathcal O(-1/\log \varepsilon)>0$ such that
\begin{eqnarray*}
\P[\tau_\var>t]\lesssim e^{-r_0(\var) t},\quad\forall t>0.
	\end{eqnarray*}

\medskip

Assumption {\bf (H2)}(i) asserts that 
when $X_t$ and $Y_t$ belong to the same basin, there is a positive probability of  successful coupling.  Assumption {\bf (H2)}(ii) states that as $\var$ tends to zero, the exponential tail of $\tau_\var$ 
vanishes at the rate $\mathcal O(-1/\log\var)$. This is expected, since in the limiting case
$\var=0,$ one process may be trapped at the saddle point on the boundary and cannot couple with the other one.  The rate $\mathcal{O}(-1/\log \var$) can be derived by explicitly solving the linearized dynamics near the saddle point; see also \cite{monter2011normal, bakhtin2011noisy} for rigorous results on the passage time of a small-noise perturbation of a deterministic dynamical system through a hyperbolic equilibrium.

The following proposition provides sufficient conditions for  
{\bf (H2)}, which will be numerically verified in Section 5.

\begin{prop}\label{pfH2}
Assume that for each $i\in\{1,\ldots,L\}$, the following conditions hold:
\begin{itemize}
\item[(a)] There exist constants $T_0=\mathcal O(-\log \var)$, $\delta > 0$, and $\gamma_0 > 0$ such that $$
\mathbb{P}[(X_{t}, Y_{t}) \in B^i_\delta \times
B^i_\delta,\text{for all}\ h\le t\le T_0\,|\, (X_0,Y_0)\in\overline{ B_i\times B_i}] \ge\gamma_0
$$
is uniform for all $(X_0, Y_{0}) \in\overline{B_i\times B_i}$, where $B^i_{\delta}=\{ x \in B_i \,|\, d(x, \partial B_i) > \delta\}$ denotes the $\delta$-interior of $B_i$;
\item[(b)] There exists 
			a  strongly convex  neighborhood $B^{i}_{c}$  of $x_i$ such that $B^{i}_{c}\subseteq B_i$.
\end{itemize}
		Then Assumption {\bf (H2)} holds. 
	\end{prop}

For the proof of Proposition \ref{pfH2}, refer to Appendix \ref{appendix_2}.	
    
\begin{remark}\label{rem:weaker_condition}
{\rm 
Assumption (b) in Proposition \ref{pfH2} holds if  $U$ has non-vanishing second-order derivatives at the minimum $x_{i}$. Assumption (a) asserts that if the two processes start from the same basin, the probability
that both strictly remain in that basin for an extended period of time of order $\mathcal O(-\log\var)$ is positive. While this is intuitive, a rigorous proof is technically challenging and beyond the scope of this paper. It would require specifying a normal form for $-\nabla U$ near the boundary and analyzing the exit behavior  from the separatrix of the  reflection-coupled processes governed by \eqref{SDE1}. Therefore, we choose to verify this assumption numerically. 
}
\end{remark}
			
The analysis becomes more intricate when $X_t$ and $Y_t$ initially  belong to different basins. The coupling process $(X_t,Y_t)$  typically evolves in two  stages during 
$(h,\tau_\var)$. In Stage 1, $X_t$ and $Y_t$  lie in  different basins until one of them, either $X_t$ or $Y_t,$  jumps out of its initial basin  and enters the other, making both of them stay in the same basin. Then Stage 2  begins, where $X_t$ and $Y_t$ are  in the same basin for a period of time, until  they are either successfully coupled or fail to couple with one of them  jumping out of the basin again. 
Accordingly,  write
\begin{eqnarray}\label{tau_equation}
\tau_\var=\tau_\var\circ\theta^{\tau_\var^{(1)}}+\tau^{(1)}_\var,\quad\P\text{- a.s.}
\end{eqnarray}
where $\tau^{(1)}_\var$ and $\tau_\var\circ\theta^{\tau_\var^{(1)}}$ correspond to the Stage 1 and Stage 2, respectively, and $\theta$ denotes the usual shift operator. 

We note that Stage 1 and Stage 2 exhibit  different time scales: Stage 1  corresponds to a slow time scale,  typically persisting over an exponentially long period,  with the tail exponent diminishing exponentially in terms of $\var$;  Stage 2 is associated with the fast time scale, and as shown in {\bf (H2)},  the  exponent of the tail distribution remains uniformly away from zero, independent of $\var$. 			

Based on the above analysis, we obtain the following estimate of  $\tau_\var$ when $X_t$ and $Y_t$ initially belong to  different basins.    

\begin{lemma}\label{lem:4}
		Let  $(X_t,Y_t)$ be  a  coupling  of two solutions of \eqref{SDE1}
		such that  $(X_0,Y_0)\in \overline{B_1\times  B_2}$ or
		$\overline{B_2\times B_1}.$ Assume  {\bf (H2)}. Then there exists
		$C_1>0$ such that for any $t>0$ and any $\var>0$ sufficiently small, 
		\begin{eqnarray*}
			\P[\tau_\var>t]\le C_1\lambda_\var^{-t}.
		\end{eqnarray*}  
		Consequently, by
		Proposition \ref{prop:prob_generate_fun2}, for any $\la\in(1,\la_{\var}),$
\begin{eqnarray}\label{ineq:13}\E[\lambda^{\tau_\var}]\le g(\lambda; C_1, \la_{\var},1)<\infty.
		\end{eqnarray}
	\end{lemma}
\begin{proof}				
According to \eqref{tau_equation},
 \begin{eqnarray}\label{ineq:8}
\P[\tau_\var>t]&=&\int_h^t\P[\tau_\var^{(1)}=s]\P[\tau_\var>t|\tau_\var^{(1)}=s]ds\nonumber\\
&\le&\int_h^t\P[\tau_\var^{(1)}>{s-\de}]\P[\tau_\var\circ\theta^s>t-s]ds
\end{eqnarray}
where $\delta\in(0,h)$ is sufficiently small.  
By \eqref{ineq:11}, there exists a constant $C_2>0$ such that
\begin{eqnarray*}
\P[\tau_\var^{(1)}>s-\de]\le C_2\lambda_\var^{-(s-\de)}.
\end{eqnarray*}
Moreover, since $X_t, Y_t$ belong to the same basin at  $t=\tau_\var^{(1)},$ it follows from {\bf (H2)}(ii) that there exists a constant $C_3>0$ such that 
\begin{eqnarray*}
\P[\tau_\var\circ\theta^s>t-s]
\le  C_3e^{-r_0(\var)(t-s)}. 
\end{eqnarray*}
Since $\delta$ is arbitrarily small, it follows that
\begin{eqnarray*}
\eqref{ineq:8}\
\le C_1\la_\var^{-t}\int_0^t\big(\la_\var e^{-r_0(\var)}\big)^{t-s}ds,
\end{eqnarray*}
where $C_1>0$ is a constant  independent of $t$ and $\var.$
				
Note that $\la_\var e^{-r_0(\var)}<1$ for $\var>0$ sufficiently small. Thus, by enlarging $C_1$ if necessary, 
\begin{eqnarray*}
\eqref{ineq:8}\le C_1\la_\var^{-t}.
\end{eqnarray*}
The lemma is proved.
\end{proof}

Combining {\bf(H2)}(ii) and 
Lemma \ref{lem:4},  we immediately obtain the following result.

\begin{prop}\label{lem:5}
Let $(X_t,Y_t)$ be a  coupling of two solutions of \eqref{SDE1} such that $(X_0,Y_0)$ is fully supported. 
Assume  {\bf (H2)}. 
Then for any $\var>0$ sufficiently small and  any $\la\in(1,\la_\var),$
\begin{eqnarray*}
\E[\la^{\tau_\var}]<\infty. 
\end{eqnarray*}
\end{prop}

\subsection{Proof of Theorem \ref{thm:2}}
Let $U$ be a double-well potential and $(X_t,Y_t)$ be a  coupling of two solutions of \eqref{SDE1}. In this section,  the coupling $(X_t,Y_t)$ is assumed to satisfy {\bf (H1)}-{\bf (H2)}. As in the proof of Theorem \ref{thm:1},  a sequence of random times is defined inductively as \begin{eqnarray}\label{iteration}
				\tau_\var^{0}=0,\  \tau_\var^{k}=\tau_\var^{k-1}+\tau_\var\circ\theta^{\tau_\var^{k-1}},\quad k\ge1.
			\end{eqnarray}
			where $\theta$ is the usual shift operator. Note that by the definition of $\tau_\var,$ for each $k\ge1,$ either $X_{\tau_\var^k}=Y_{\tau_\var^k},$  or $X_{\tau_\var^k}$ and $Y_{\tau_\var^k}$ belong to    different basins. Let
			\begin{eqnarray}\label{eta}
				\eta=\inf\{k\ge1:X_{\tau_\var^k}=Y_{\tau_\var^k}\}. 
			\end{eqnarray}
			
			\medskip
			
			The following Proposition \ref{prop:5} and Theorem \ref{thm:4} are   analogues of Proposition \ref{prop:1} and Theorem \ref{thm:3}, respectively,  in the  double-well setting.
			\begin{prop}\label{prop:5}
				For any $\var>0$ sufficiently small, the following hold:
				
				\quad 
				(i) $X_{\tau_\var^k}=Y_{\tau_\var^k}$ if and only if $k\ge\eta;$ 
				
				\quad (ii) For any $\var>0$ sufficiently small, it holds that 
				\[\P[X_{\tau_\var^k}\neq Y_{\tau_\var^k}|\cF_{\tau_\var^{k-1}}]<\ga_1,\quad\forall k\ge1,\]
				where $\ga_1\in(0,1)$ is as in {\bf (H2)}. 
			\end{prop}
			Note that  Proposition \ref{prop:5} (i)  yields  
			\[\tau_c=\tau_\var^\eta,\quad\text{$\P$-a.s.}\]
			Proposition \ref{prop:5} (ii)  directly follows from {\bf (H2)}(i) by the strong Markov property.

			\begin{theorem}\label{thm:4}
				Let $(X_t,Y_t)$ be a  coupling of two solutions of \eqref{SDE1} such that  $(X_0,Y_0)$ is fully supported. 
				Assume  {\bf (H2)}. 
				Then  for any $\var>0$ sufficiently small and any $\la\in(1,\la_\var),$ it holds that
\begin{eqnarray*}
\E[\la^{\tau_\var^\eta}]<\infty.
\end{eqnarray*}				
\end{theorem} 
			\begin{proof}
				The proof follows the same approach as that of Theorem \ref{thm:3}.  By replacing  $\I_{\{R_{\tau_\var^k>0}\}}$ in the proof of Theorem \ref{thm:3} with $\I_{\{X_{\tau_\var^k}\neq Y_{\tau_\var^k}\}},$ we have
\begin{eqnarray*}
\E[\lambda^{\tau_\var^\eta}]\le\E[\la^{\tau_\var}]\Big(1+ g(\lambda;C_1,\la_\var,1)\sum_{k=0}^\infty \Big(\ga_1^{1-p}g(\lambda^{1/p}; C_1,\la_\var,1)^{p}\Big)^k\Big),
\end{eqnarray*}
				where $C_1>0$ is the constant  given  in Lemma \ref{lem:4}, $\gamma_1$ is as in Proposition \ref{prop:5}(ii),  
				and $p\in(0,1)$ is an arbitrary  number.  Thus, $\E[\lambda^{\tau_\var^\eta}]<\infty$ holds if the inequality
				\begin{eqnarray}\label{ineq:3}
					g(\la^{1/p};C_1,\la_\var,1)<\ga_1^{-(1-p)/p}
				\end{eqnarray}
is satisfied.

By Proposition \ref{prop:prob_generate_fun1},
\eqref{ineq:3} holds for any $\lambda>1$ satisfying 
$\la^{1/p}\in(1,1+\beta(\la_\var-1)),$ 
where $\beta=\min\{1,\beta^*\}$ and $\beta^*$ is given by \eqref{beta1}. 
Note that 
\begin{eqnarray*}
   \beta^*\to\ga_1^{-(1-p)/p}-1,\quad\text{as $\var\to0.$}
\end{eqnarray*} 
Since $\ga_1^{-(1-p)/p}-1>0$ and diverges to infinity as $p\to0,$ one can choose $p\in(0,1)$ such that $\ga_1^{-(1-p)/p}-1>1.$ It then follows that $\beta^*>1$, and hence $\beta=1$, for any sufficiently small $\var>0$. 
Hence, $\lambda^{1/p}$ can be arbitrarily close to $\lambda_\var.$ 
By the definition of $\lambda_\var$ in \eqref{la_var},  it follows that
\begin{eqnarray*}
  \ln\la\simeq C_\var e^{-2H_U/\var^2}.
\end{eqnarray*}
\end{proof}
			
			\medskip
			
			\begin{proof}[Proof of Theorem \ref{thm:2}:]						
				Note that for any $\la>1$ satisfying $\E[\la^{	\tau_\var^\eta}]<\infty,$
				it holds that 
				\begin{eqnarray*}
					\P[\tau_\var^\eta>t]\la^t\le\E[\la^{\tau_\var^\eta}].
				\end{eqnarray*}
				 It then follows from Proposition \ref{prop:5} (i) and  Theorem \ref{thm:4}  
				that for any $t>0$ and $\var>0$ sufficiently small, 
				\begin{eqnarray*}
					\P[\tau_c>t]\lesssim \la^{-t},
				\end{eqnarray*}
                where $\lambda$ satisfies $\ln\la\simeq C_\var e^{-2H_U/\var^2}.$
				Hence, 
			\begin{eqnarray}\label{less1}
					\limsup_{t\to\infty}\dfrac{1}{t}
					\log\P[\tau_c>t]\lesssim-\ln\lambda\simeq -C_{\var}e^{-2H_U/{\var^2}}. 
			\end{eqnarray} 
				
For the reverse inequality,
since $(X_0,Y_0)$ is fully supported, it follows that for any $\de>0$ sufficiently small, 
\begin{eqnarray*}
					\P[\tau_c>t]\ge\P[\tau_c>t, (X_0,Y_0)\in  B_\de(x_1)\times B_\de(x_2)\  \mbox{or}\ B_\de(x_2)\times B_\de(x_1)].
				\end{eqnarray*}
				Note that when $X_0$ and $Y_0$ belong to different basins, it holds that  $\tau_c\ge\tau^{(1)}_\var.$ Thus, by \eqref{ineq:15}
\begin{eqnarray*}
&&\P[\tau_c>t, (X_0,Y_0)\in  B_\de(x_1)\times B_\de(x_2)\  \mbox{or}\ B_\de(x_2)\times B_\de(x_1)]\\
&\ge&\P[\tau_\var^{(1)}>t, (X_0,Y_0)\in  B_\de(x_1)\times B_\de(x_2)\  \mbox{or}\ B_\de(x_2)\times B_\de(x_1)]
\  \gtrsim\ \la_\var^{-t},
\end{eqnarray*}
and hence
\begin{eqnarray}\label{greater2}
\limsup_{t\to\infty}\dfrac{1}{t}\log\P[\tau_c>t]\gtrsim -C_{\var}e^{-2H_U/{\var^2}}.
\end{eqnarray} 
Combining \eqref{less1} and \eqref{greater2} completes the proof of Theorem \ref{thm:2}. 
\end{proof}

\subsection{Multi-well potential and proof of Theorem \ref{thm:5}}
In this section, we study the general case of multi-well potentials. Let  $U$ be a multi-well  potential satisfying {\bf (U3)},   
 with $L$$(L>2)$ local minima $x_1,\ldots,x_L$ and the corresponding basins $B_1,\ldots,B_L$. 
Let   $(X_t,Y_t)$ be a coupling  of two solutions of \eqref{SDE1}. 
			
Similar to the double-well case, several key stopping times need to be defined to 
estimate the coupling time. By a slight abuse of notation,  we continue to use  $\tau_\var^{(1)}$  to denote the infimum time at which $X_t$ and $Y_t$ lie in the same basin, i.e.,
			\begin{eqnarray*}
				\tau^{(1)}_\var=\inf\Big\{t>h:(X_t,Y_t)\in\cup_{1\le i\le L} B_i\times B_i\Big\}.
			\end{eqnarray*}
			Define
			\begin{eqnarray*}
				\tau_\var^{(3)}=&\inf&\Big\{t>h:(X_t,Y_t) \in\bigcup\nolimits_{\substack{1\le i,j\le L,i\neq j}}B_i\times B_j,\mbox{and for some}\ s\in(h,t),\\
				&&(X_s,Y_s)\in 
				\bigcup\nolimits_{1\le i\le L}B_i\times B_i \Big\},
			\end{eqnarray*}				
			and let $\tau_\var^{(3)}=\infty$ if the set is empty. 
Note that $\tau_\var^{(3)}$ generalizes $\tau_\var^{(2)}$ to the case of multiple wells and coincides with  $\tau_\var^{(2)}$ when $L=2.$	
			
In the multi-well setting, a key stopping time of interest is when both $X_t$ and $Y_t$ lie in the vicinity of the (unique) global minimum  $x_1.$	Let $\xi_1$ denote the infimum time at which  both $X_t$ and $Y_t$ lie in the basin  $B_1$.
Recall,  as defined in \eqref{kappa},  that $\ka_{\bs X}(B_1)$ ({\it resp.} $\ka_{\bs Y}(B_1)$) denotes the infimum time at which $X_t$ ({\it resp.} $Y_t$) enters $B_1$.  Then 
\begin{eqnarray}\label{ineq:17}
				\xi_1\ge \max\big\{\ka_{\bs X}(B_1),  \ka_{\bs Y}(B_1)\big\}.
\end{eqnarray}
			
Note that  $X_t$ and $Y_t$ may enter  and exit the basin $B_1$ multiple times  before $\xi_1.$ However, as long as $\var$ is sufficiently small, the typical  scenario is that  one of the two processes, say $X_t,$ first enters $B_1$ and ``waits" for $Y_t$ to arrive. Although $X_t$ may leave $B_1$ before $Y_t$ enters, it is highly probable  that $X_t$ will stay in  nearby basins and return to $B_1$ shortly after $Y_t$ enters. 

\medskip

The following {\bf (H3)} assumes that $\xi_1$ is no greater than  $\ka_{\bs X}(B_1)$ (or $\ka_{\bs Y}(B_1)$) up to an infinitesimal of the same order as  $\ka_{\bs X}(B_1)$ (or $\ka_{\bs Y}(B_1)$). 
 	
  \medskip
			
\noindent{\bf (H3)} 
Let $(X_t,Y_t)$ be a coupling of two solutions of \eqref{SDE1} such that  $(X_0, Y_0)\in\bigcup\nolimits_{1\le i,j\le L, i\neq j} \overline{B_i\times  B_j}.$ Then for any $\var>0$ sufficiently small, 			\begin{eqnarray}\label{ineq:18}
				\limsup_{t\to\infty}\dfrac{1}{t}\log\P\big[\big(\xi_1-\max\big\{\ka_{\bs X}(B_1),\ka_{\bs Y}(B_1)\big\}\big)>t\big]\lesssim e^{-2H_U/\var^2}.
			\end{eqnarray}
   \medskip

We note that in contrast to  {\bf (H2)}, which provides a local characterization of the coupling properties between $X_t$ and $Y_t$, {\bf(H3)} imposes a {\it global} condition on the coupling between $X_t$ and $Y_t$ as  both processes evolve across  the entire potential landscape. 
In Section 5.4, {\bf (H3)} is numerically verified for the reflection-maximal coupling scheme. 

\begin{remark}
{\rm 
    A rigorous justification of {\bf (H3)} is highly challenging, as it requires estimating the simultaneous hitting time, i.e., $\xi_1$, of the coupled process $(X_t, Y_t)$. Although there are some results on the simultaneous hitting time of independent processes \cite{lindvall_book}, to the best of knowledge of the authors, no such result exists for two reflection-coupled stochastic differential equations.
}
\end{remark}

For the multi-well case with $L>2$, {\bf (H1)}-{\bf (H3)} are  assumed. 
As in the double-well case, the quantity $\la_\var$ is defined as
\begin{eqnarray*}
	\la_\var=\exp\{C_\var  e^{-2H_U/\var^2}\},
\end{eqnarray*}
where $H_U$ now represents the essential barrier height in the general form \eqref{barrier_height_multi-well}, applicable to multi-well potentials. Still,  $C_\var>0$ is any constant,  not uniquely determined,  such that $\lim_{\var\to0}C_\var$ exists and depends only on $U.$ 

 Under assumption {\bf (H3)} and the initial condition $(X_0,Y_0)\in$ $\bigcup\nolimits_{1\le i,j\le L, i\neq j}\overline{B_i\times B_j},$
Lemma \ref{lem:expon_tail_multi_well} implies 
\begin{eqnarray*}
	\P[\xi_1>t]\lesssim \la_\var^{-t},
\end{eqnarray*}
which, since $\tau^{(1)}_\var\le\xi_1$,  further yields 
\begin{eqnarray}\label{ineq:22}						\P[\tau_\var^{(1)}>t]\lesssim\la_\var^{-t}.	
\end{eqnarray}

Still, similar to the double-well case, we denote
\begin{eqnarray*}
	\tau_\var=\tau_\var^{(3)}\wedge\tau_c.
\end{eqnarray*}   
The estimation of $\tau_\var$ follows the same reasoning as in the double-well case: If $X_t$ and $Y_t$ initially belong to the same basin, the result  directly follows from {\bf (H2)}(ii). If $X_t$ and $Y_t$ initially belong to different basins,  the  coupling process  $(X_t,Y_t)$ is typically decomposed into two stages over the time interval $(h,\tau_\var)$. In Stage 1, $X_t$ and $Y_t$ remain in different basins until  $\tau_\var^{(1)}$, at which time they are in the same basin.  Stage 2 then follows, during which  $X_t$ and $Y_t$ are either successfully coupled within the same basin or remain  uncoupled before being separated again by different basins. 
Hence, similar to  \eqref{tau_equation} for the double-well case, it  holds for the multi-well case as well that
	\begin{eqnarray*}
\tau_\var=\tau_\var\circ\theta^{\tau_\var^{(1)}}+\tau^{(1)}_\var,\quad\P\text{- a.s.}
	\end{eqnarray*}
where $\tau_\var^{(1)}$ and $\tau_\var$ are  defined in the setting of multi-well potential. 
 
\medskip

The following is the ``multi-well version"  of Lemma \ref{lem:4}
for the general case of $L\ge2$.
				
\begin{lemma}\label{lem:multi_well}
Let $(X_t,Y_t)$ be a coupling  of two solutions of \eqref{SDE1} such that   $(X_0,Y_0)\in\bigcup_{1\le i,j\le L, i\neq j} \overline{B_i\times  B_j}.$ Assume {\bf (H2)-(H3)}.
Then for any $t>0,$
\begin{eqnarray*}
\P[\tau_\var>t]\lesssim\la_\var^{-t}.					\end{eqnarray*}	
\end{lemma}
\begin{proof}
The proof follows similarly to that of Lemma \ref{lem:4}.
As in \eqref{ineq:8},
for any $\delta\in(0,h)$,
\begin{eqnarray}\label{ineq:80}
\P[\tau_\var>t]
\le\int_h^t\P[\tau_\var^{(1)}>{s-\de}]\P[\tau_\var\circ\theta^s>t-s]ds.
\end{eqnarray}
Applying \eqref{ineq:22} and {\bf (H2)}(ii) yields
\begin{eqnarray*}
\P[\tau_\var^{(1)}>s-\de]\lesssim \lambda_\var^{-(s-\de)}
,\quad
\P[\tau_\var\circ\theta^s>t-s]
\lesssim e^{-r_0(\var)(t-s)}.
\end{eqnarray*}
Since $\delta$ can be arbitrarily small, substituting into \eqref{ineq:80} yields
\begin{eqnarray*}
\P[\tau_\var>t]
\lesssim \la_\var^{-t}\int_0^t\big(\la_\var e^{-r_0(\var)}\big)^{t-s}ds
\end{eqnarray*}
Thus, for any $\var>0$ sufficiently small such that $\la_\var e^{-r_0(\var)}<1$, and therefore the integral remains bounded, we obtain
\begin{eqnarray*}
\P[\tau_\var>t]
\lesssim\la_\var^{-t}.
\end{eqnarray*}
\end{proof}

    \medskip

The proof of Theorem \ref{thm:5} is analogous to that of Theorem  \ref{thm:2}.
				
				\medskip

\noindent{\it Proof of Theorem \ref{thm:5}:} 
Analogous to Proposition \ref{lem:5} in the double-well case, a combination of assumption  {\bf(H2)}(ii) and Lemma \ref{lem:multi_well} yields the following result for the multi-well case: for any $\var>0$ sufficiently small and  any $\la\in(1,\la_\var),$
	\begin{eqnarray*}
	\E[\la^{\tau_\var}]<\infty.
\end{eqnarray*}
As demonstrated in both the single and double-well cases,  the coupling time $\tau_c$ can be written as a finite iteration of $\tau_\var$, specifically  $\tau_c=\tau_\var^\eta,$ where $\eta$ is defined in  \eqref{eta}. 
In analogy with the double-well case in Theorem \ref{thm:4}, it then follows that for $\var>0$ sufficiently small and  any $\la\in(1,\la_\var),$ 
\begin{eqnarray*}
\E[\la^{\tau_\var^\eta}]<\infty, 
\end{eqnarray*}	
which implies the exponential tail estimate 
\begin{eqnarray*}
	\P[\tau_c>t]\lesssim \la_\var^{-t}.
\end{eqnarray*}
				
To establish the corresponding lower bound, the assumption that the initial distribution of $(X_0, Y_0)$ is fully supported is employed. 
Consider the case where $X_t$ is initialized in the distant basin $B_2$, and $Y_t$ starts in the basin $B_1$ associated with the global minimum. Under the assumption {\bf (H1)}, the event that $Y_t$ does not exit the region $\bs B_1$ prior to the entrance of $X_t$  occurs with positive probability, uniformly in both $\var$ and $t.$ Hence,
\begin{eqnarray}\label{ineq:21}						\P[\tau_c>t]
&\ge&\P\big[Y_s\in\bs B_1\ \text{for all}\  s\in[0,t]\big|\ka_{\bs X}(\bs B_1)>t\big]\cdot\P[\ka_{\bs X}(\bs B_1)>t]\nonumber\\				
&\ge&\ga_0\cdot\P[\ka_{\bs X}(\bs B_1)>t]\ {\simeq}\ \la_\var^{-t}
\end{eqnarray} 			
where the last approximation follows from 
\eqref{genaralized_8}, 
the strengthened version of Lemma \ref{lem:expon_tail_multi_well}. 
This completes the proof of Theorem \ref{thm:5}. \qed

\section{Numerical Examples}
This section presents numerical examples to verify  the theoretical results and the assumptions {\bf (H1)}-{\bf(H3)} concerning  the coupling scheme introduced in the preceding sections.  An algorithm is first proposed in Section \ref{exptail} to obtain accurate numerical estimates of the exponential tails of the coupling times. For further  details on the coupling algorithm, the reader is referred to \cite{li-wang2020}. 	

\subsection{An algorithm for  exponential tail estimation}\label{exptail}
Let $\tau_{c}$ denote the coupling time. 
While the rigorous results only establish bounds on the limit superior  of $\frac{1}{t}\log \mathbb{P}[\tau_c > t]$,  numerical simulations consistently indicate convergence of $\frac{1}{t}\log \mathbb{P}[\tau_c > t]$ as $t$ increases. Therefore, the numerical investigation focuses on computing the exponential decay rate of $\mathbb{P}[\tau_{c} > t]$ with respect to $t$, that is,
$$
    r(\var) = -\lim_{t\rightarrow \infty} \frac{1}{t}\log \mathbb{P}[ \tau_{c} > t] \,,
$$    
where $\var$ denotes the  noise magnitude in \eqref{SDE1}.    
Since only a finite number of coupling events can be sampled,  an efficient algorithm is required both to provide statistical evidence for the existence of the exponential tail and  to estimate its decay rate with reasonable accuracy.  
				
The main challenge is that $\mathbb{P}[\tau_{c} > t]$ typically does not exhibit exponential decay until $t$ is sufficiently large. It is therefore necessary to identify a suitable threshold $t^{*}$ such that the tail $\mathbf{1}_{\{\tau_{c} > t^{*} \}}(\tau_{c} -t^{*})$  approximately follows an exponential distribution, while keeping $t^*$ as small as possible to ensure that enough samples with $\tau_{c} >t^{*}$ are available. However, most exponentiality tests that have been attempted yield a threshold $t^{*}$ that is too small, resulting in the failure of the log-linear plot of the tail to stabilize into a linear trend.  This  is likely due to  the sensitivity of the plots to small deviations in tail behavior. 

The goal of our algorithm is to determine a suitable $t^{*}$ such that the log-linear plot of $\mathbb{P}[\tau_{c} > t]$ is  approximately linear 
for all $t>t^*.$
That is,  the confidence interval of the estimated values of
$\mathbb{P}[\tau_{c} > t]$ should contain a straight line on the logarithmic scale  for all $t > t^*$. The algorithm proceeds as follows. 
First, select a sequence of times $t_{0}, t_{1}, \ldots t_{N}$, where $t_{N}$ is typically set as the maximum of the sampled coupling times. Let $M$ denote the total sample size,  and for each $i$, let $n_{i}$ be the number of samples satisfying $\tau_{c} >t_{i}$. The Agresti-Coull method \cite{Agresti-Coull1998} provides a confidence interval for each $i$ of the form
				\[[ \tilde{p}_{i}^{-}, \tilde{p}_{i}^{+}]:=					[ \tilde{p}_{i} - z \sqrt{\frac{\tilde{p}_{i}}{\tilde{M}}(1 - \tilde{p}_{i}) } \,,
				\tilde{p}_{i} + z \sqrt{\frac{\tilde{p}_{i}}{\tilde{M}}(1 -
					\tilde{p}_{i} )} ]\,,
				\]
				where 	$\tilde{M} = M + z^{2},$ $ \tilde{p}_{i} = {(n_{i} + \frac{z^{2}}{2})}/{\tilde{M}},$ and $z = \Phi^{-1}( 1 - \alpha/2)$ is the $\alpha$-quantile of the standard normal distribution. In practice, $z= 1.96$ and $\alpha = 0.05$ are commonly used. 
				
Given any $N_0\in\{1,\dots,N\}$, a weighted linear regression can be performed to fit the points $(t_{i}, \log\tilde{p}_{i})$ for $i = N_{0},\ldots, N$, where each point  $(t_{i}, \log\tilde{p}_{i})$ is assigned a weight of $n_i/M$. 
If the regression yields a linear function of the form $y = a t + b$, then $N_{0}$ is considered acceptable if it satisfies 		
\begin{eqnarray}\label{N0}
\big| \{ N_0\le i\le N \}:  a t_{i} + b \notin [ \tilde{p}_{i}^{-},\tilde{p}_{i}^{+}]\big| < \alpha (N - N_{0} + 1),
\end{eqnarray}			
ensuring the residuals $\mathbf{1}_{\{\tau_{c} > t_{N_0}\} }(\tau_{c} -t_{N_0})$ are statistically consistent  with an exponential tail beginning at $t_{N_0}.$
For each candidate   $N_0$, this procedure evaluates whether the tail distribution of $\tau_c$ beyond $t_{N_0}$ 
is approximately exponential.  The final choice of $N_0$ is the smallest index that satisfies  the condition \eqref{N0},  which can be efficiently found via  binary search over $\{1, \ldots, N\}$ in $\mathcal O(\log N)$ iterations.
The threshold $t^{*}$ is then defined by $t_{N_{0}}$, and  the exponential decay rate is given  by the slope $a$ of the corresponding weighted regression line. 
		
\subsection{Quadratic potential function}\label{sec52}
The first example considers the quadratic potential function. The primary objective is to numerically verify the theoretical result stated in Theorem \ref{thm:1}. This example will be revisited in Subsection \ref{numerical verification} to examine the consistency of the first passage times between the continuous-time process and its time-$h$ sampled chain as  the step size $h$ tends to zero, in accordance with the approximation \eqref{approximate} discussed in Remark \ref{rem:approximate}. 	
				
Consider the quadratic potential function
\[U(x) = \frac{1}{2} {x}^{T} \bs{A} {x},\quad  x\in\R^k,\]
where $\bs{A}$ is a $k\times k$ Lehmer matrix whose entries are given by $\bs{A}_{ij} =
\min(i,j)/\max(i,j).$  The matrix $\bs A$ is
symmetric and positive definite \cite{lehmer1958}. 
The associated SDE is
\begin{equation}\label{Lehmer}
\mathrm{d}Z_{t} = - \bs{A} Z_t\mathrm{d}t + \varepsilon\mathrm{d}W_{t} \,,
\end{equation}
where   $W_{t}$ is a
$k$-dimensional Wiener process, and $\varepsilon > 0$ denotes the noise magnitude. 
			
In the numerical simulations, the time step size $h$ is fixed at $0.001$, unless stated otherwise. 
Figure \ref{fig1} displays
the probability distribution of the coupling time $\tau_{c}$. The four panels show $\mathbb{P}[\tau_{c}
> t]$ versus $t$ on a log-linear scale for Lehmer matrices of size  $2\times 2$, $4\times 4$, $6\times 6$, and $8\times 8,$  respectively. 
For each case, the noise magnitude $\var$ is set to $0.02$, $0.1$, $0.5$, and $1.5$. The slopes and linear fitting in the log-linear plots are determined using the algorithm described in Subsection~\ref{exptail}.
The smallest eigenvalue of $\boldsymbol{A}$ is indicated in the subtitle of each subplot in Figure \ref{fig1}.
			
In all four cases,  although the probability distribution of $\tau_{c}$ varies significantly with the  noise magnitude, the slopes of the exponential tails remain unchanged. Moreover, the smallest eigenvalue of 	$\boldsymbol{A}$, which can be computed explicitly,  closely approximates the slope of the corresponding exponential tail, with an error of at most $0.01$. This observation is consistent with Theorem \ref{thm:1}, which asserts that the slope of the	exponential tail is determined by the convexity of the potential function and is independent of the noise magnitude. 
			
\begin{figure}[htbp]
\centerline{\includegraphics[width = \linewidth]{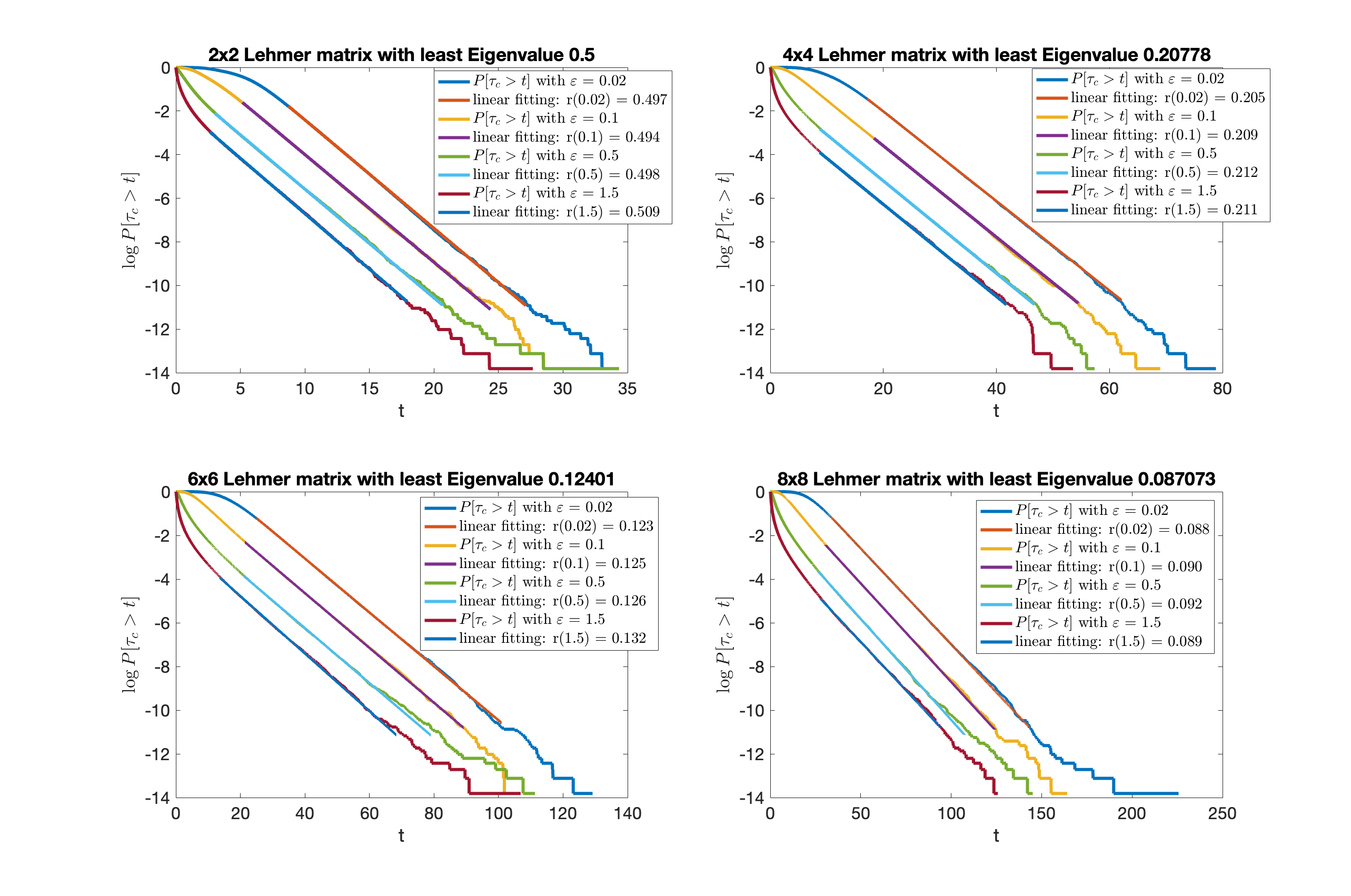}}
\caption{Log-linear plots of $\mathbb{P}[\tau_{c} > t]$ versus  $t$ and their exponential tails. The four panels correspond to  Lehmer matrices of sizes $2, 4, 6$, and $8$. The smallest eigenvalue of each matrix is indicated  in the title of the corresponding subplot.}
\label{fig1}
\end{figure}
	
\subsection{1D double-well potential} 
This subsection considers an asymmetric one-dimensional double-well potential given by
\[U(x) =  x^{4} - 2 x^{2} + 0.2x,\quad x\in\R.
\]
The potential  $U$ has two local minima located at $x=0.9740$ and $x=-1.0241.$ The barrier height that a trajectory must overcome to transition from the left well to the right is approximately $1.2074$, while the reverse transition requires overcoming a barrier of approximately $0.8076$; see the bottom left panel of Figure \ref{fig3}.
			
The purpose  of this example is to numerically verify the theoretical result of Theorem \ref{thm:2}, which asserts that the exponential  tail of the coupling time distribution is determined by the   {\it lower} of the two barrier heights. 
The time step size and  coupling method are the same as those used in the previous examples. The  noise magnitudes $\varepsilon$ are chosen as $0.32, 0.36, 0.4, 0.45, 0.5, 0.6$, and $0.7$. For each value of $\var$, the exponential tail $r(\var)$ is estimated using the weighted linear regression algorithm described in Subsection \ref{exptail}. The corresponding results are shown in  top panels of Figure \ref{fig3}. 
It is observed that the exponential decay rate   $r(\varepsilon)$ varies significantly with respect to  $\varepsilon$. 
   
In the bottom right panel of Figure \ref{fig3}, the quantity 
$y(\var):=-\var^2\log r(\var)$ is plotted against $\var^2,$ 
revealing an approximately linear relationship.   A linear extrapolation of $y(\var)$  as  $\var \rightarrow0$  yields the limiting value $y(0) = 1.617$, which closely agrees with  the theoretical value $y(0)=2H_U=1.615$, where $H_U$ denotes the lower barrier height of the potential. This confirms the validity of Theorem \ref{thm:2} in the asymmetric double-well setting. 

			\begin{figure}[htbp]
				\centerline{\includegraphics[width = \linewidth]{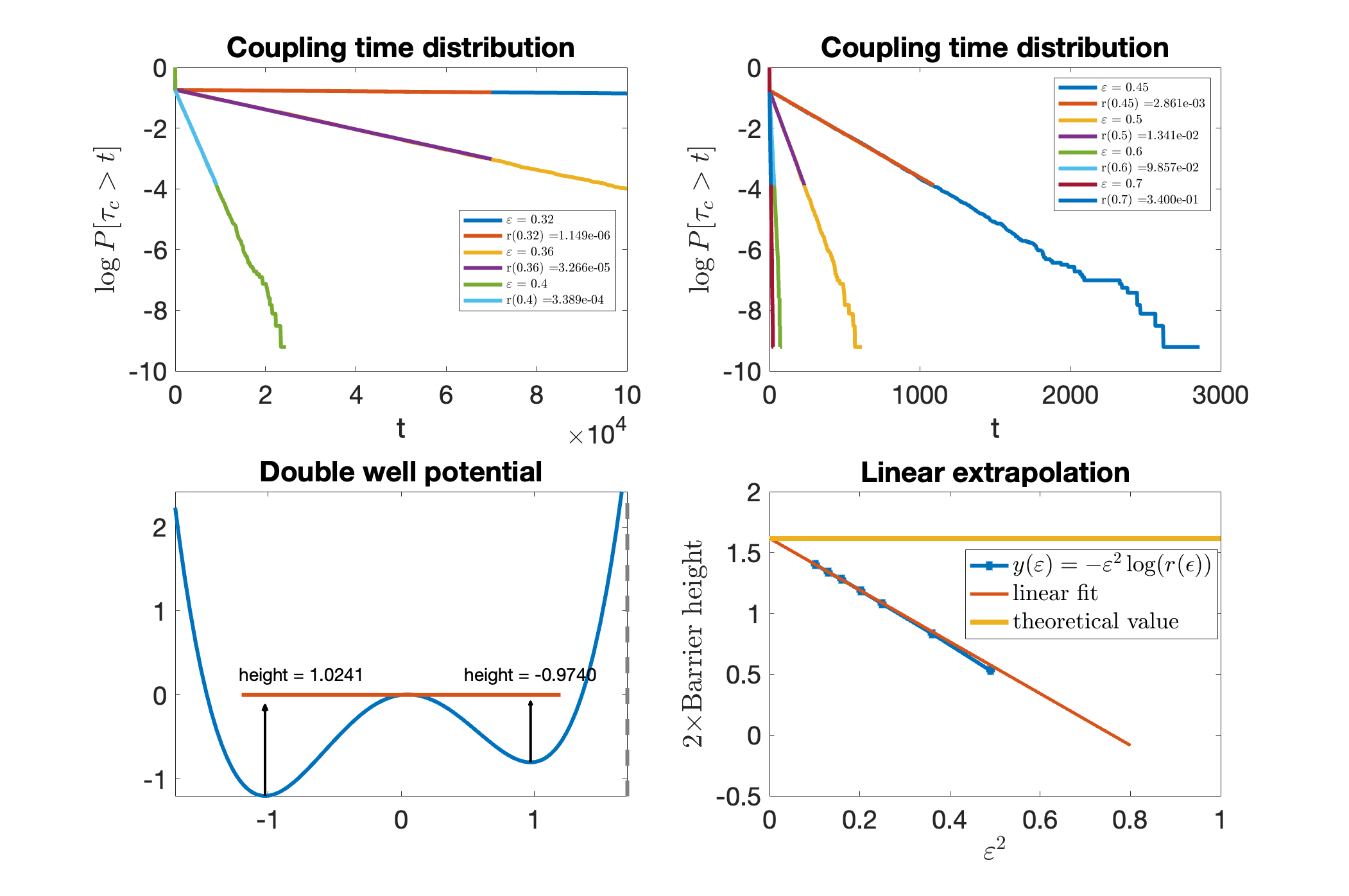}}
				\caption{Top: Coupling
					time distributions for different noise magnitudes. Bottom left: Asymmetric double-well potential. Bottom right: Linear extrapolation of the essential barrier
					height.}
				\label{fig3} 
			\end{figure}
			
\subsection{Interacting particle system in the double-well potential}\label{IPS}
This subsection considers a variation of the double-well potential introduced in the previous subsection. Let 
\[V(x) = x^{4} - 2 x^{2} + 0.2x,\quad x\in\R,\]
denote the double-well potential. Consider three particles moving along $V$ under  overdamped Langevin dynamics,  with additional pairwise interactions. The total energy potential is given by
\[U(x_{1}, x_{2}, x_{3}) = \sum\nolimits_{i = 1}^{3} V(x_{i}) + \sigma\sum\nolimits_{i,j =
1,2,3, \, i \neq j} (x_{i} - x_{j})^{2},\]
where $\sigma>0$ is the  interaction strength. 
			
The function $U$ has two trivial local minima at $x_{1} = x_{2} = x_{3} =0.9740$ and $x_{1} = x_{2} = x_{3} = -1.0241,$ 
corresponding to all  three particles occupying the same basin of $V$. For sufficiently small $\sigma>0$,  $U$ also admits six additional local minima, corresponding to configurations in which the particles are distributed across different basins; see the top  panel of Figure \ref{fig4} for a sample trajectory.
			
Two extreme regimes of interactions are notable. When $\sigma=0,$ i.e., when there are no interactions among the three particles, the particles move independently, 
resulting in a barrier height of the energy landscape identical to that of $V$. 
When $\sigma\to \infty$, the interaction is strong enough so that the three particles must move together as a single unit, making the barrier height of the energy potential $U$ three times that of $V$. For any fixed $\sigma>0$, the  essential barrier height $H_U$ lies between the barrier heights  of the two extreme cases, i.e., 
$0.8076\le H_U\le3 \times 0.8076 = 2.4228,$ with $H_U$ 
increasing as $\sigma$ increases. 
			
The estimate $H_U,$ and the distribution of the coupling time $\tau_{c}$ is computed for various values of $\var.$ For $\sigma =0.05$, values of $\var$ are chosen as $0.4,$ $0.41,$ $0.42,$ $0.43,$ $0.45,$ $0.47,$ $0.5,$ $0.55,$ $0.6,$ $0.7;$ for  $\sigma = 0.1$, values of $\var$ are $0.41,$ $0.42,$ $0.43,$ $0.44,$ $0.45,$ $0.47,$ $0.5,$ $0.55,$$0.6,$ $0.7$. The
decay rate $r(\varepsilon)$ of the exponential tails is estimated in both cases using linear weighted regression.
The relationship between $r(\varepsilon)$ and $\varepsilon$ exhibits a similar  trend to that observed for the double-well potential in the previous subsection.            
A linear extrapolation of $y(\var):=-\varepsilon^{2} \log r(\varepsilon)$ provides an estimate of the the essential barrier height.
As shown in the middle right panel of  Figure \ref{fig4},  the linear extrapolation yields $y(0) = 1.7374$ for $\sigma = 0.05$ and
$y(0) = 1.9598$ for $\sigma = 0.1$, both of which are expected to be approximately twice the barrier height $2H_U$, which will be computed using the String method below.         
As expected, the barrier
height increases with the interaction strength among the three particles. 
			
			\begin{figure}[htbp]
				\centerline{\includegraphics[width = 1.2\linewidth]{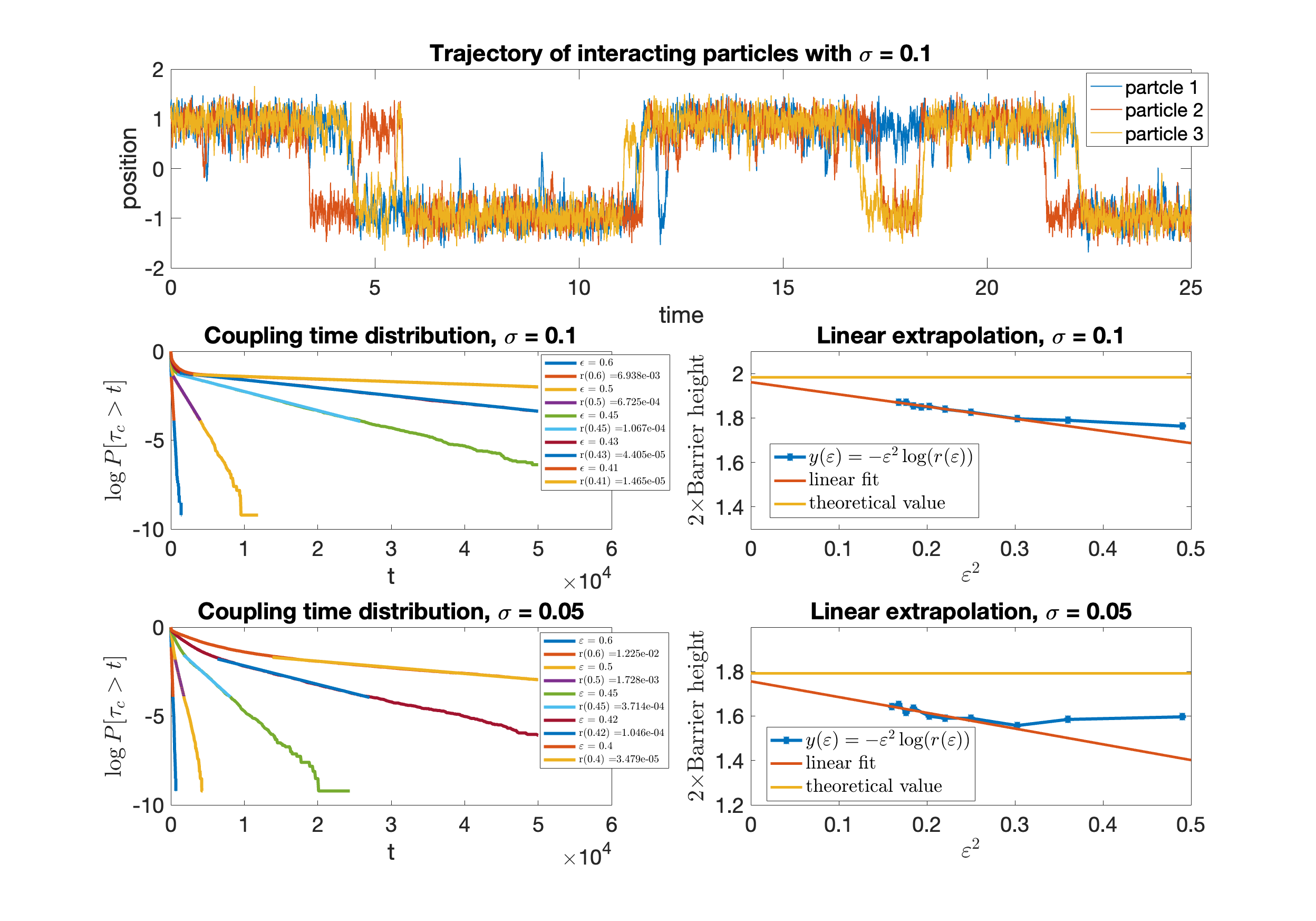}}
				\caption{Top: Sample trajectory of a three-particle interacting 
					system in a double-well potential. Middle: 
                     Coupling time
					distributions and $y(\var)=-\varepsilon^{2} \log r(\varepsilon)$ versus  $\varepsilon^2$ 
					for $\sigma = 0.1$. Bottom: Coupling time
					distributions and $y(\var)=-\varepsilon^{2} \log r(\varepsilon)$ versus $\varepsilon^2$ 
					for $\sigma = 0.05$. Theoretical values of $y(0)$ in the middle-right and bottom-right panels are obtained from the
					minimum energy path.}
				\label{fig4} 
			\end{figure}

To validate the  essential barrier height inferred from the coupling approach, the String method  (see, e.g., \cite{e2007simplified}) is employed to compute the heights of various barriers between the local minima $(0.9740, 0.9740, 0.9740)$
and $( -1.0241, -1.0241, -1.0241)$ in the energy landscape. 	As shown in Figure \ref{fig:MEPbyString}, the essential barriers, defined as the highest  barrier that a trajectory must overcome to enter the basin of the global minimum,
correspond to the leftmost barrier in the lower left panels of Figure \ref{fig:MEPbyString}: (A)  for  $\sigma=0.01$   and (B) for $\sigma=0.1$, respectively.

In this example, since the  three particles are indistinguishable, the energy potential exhibits significant symmetry: the eight local minima can be classified  into  two types, each consisting of four specific cases. These cases correspond to configurations where  (i) all  three particles reside in the same  basin (global or local), or (ii)  two of the three particles lie in one basin (global or local), while the remaining particle resides in the other basin. Figure \ref{fig:MEPbyString} shows that  the minimal energy path (MEP) connecting the two  minima  $(0.9740, 0.9740, 0.9740)$ and $( -1.0241, -1.0241, -1.0241)$   passes through  all four  cases. Thus, the essential barrier height $H_U$ can be attained along such an MEP although, 
	in principle, it should be determined by taking the supremum  over all  paths connecting any local minima to the global minima. In Figure \ref{fig:MEPbyString},  the computed values are  $H_U=0.8961$ for $\sigma = 0.05$ and $H_U=0.9916$ for $\sigma = 0.1$, which correspond to the theoretical values $y(0) = 1.7922$ for $\sigma = 0.05$ and $y(0) =1.9832$ for $\sigma = 0.1,$ respectively.

The result from the String method is further validated using the equivalent characterization  \eqref{equivalent_characterize_H_U} by numerically computing all  $27$ critical points of $U$, including all the minima and saddle points. The essential barrier heights obtained through this approach are $H_U=0.8962$ for $\sigma = 0.05$ and $H_U=0.9916$ for $\sigma = 0.1$, which are nearly identical to those computed using the String method. As shown in Figure \ref{fig4}, both values also  closely match $y(0)/2$, the estimate obtained via linear extrapolation  from the exponential tails of coupling times.

			\begin{figure}
				\begin{subfigure}{\textwidth}
					\hspace{-50pt}
					\includegraphics[width=1\textwidth]{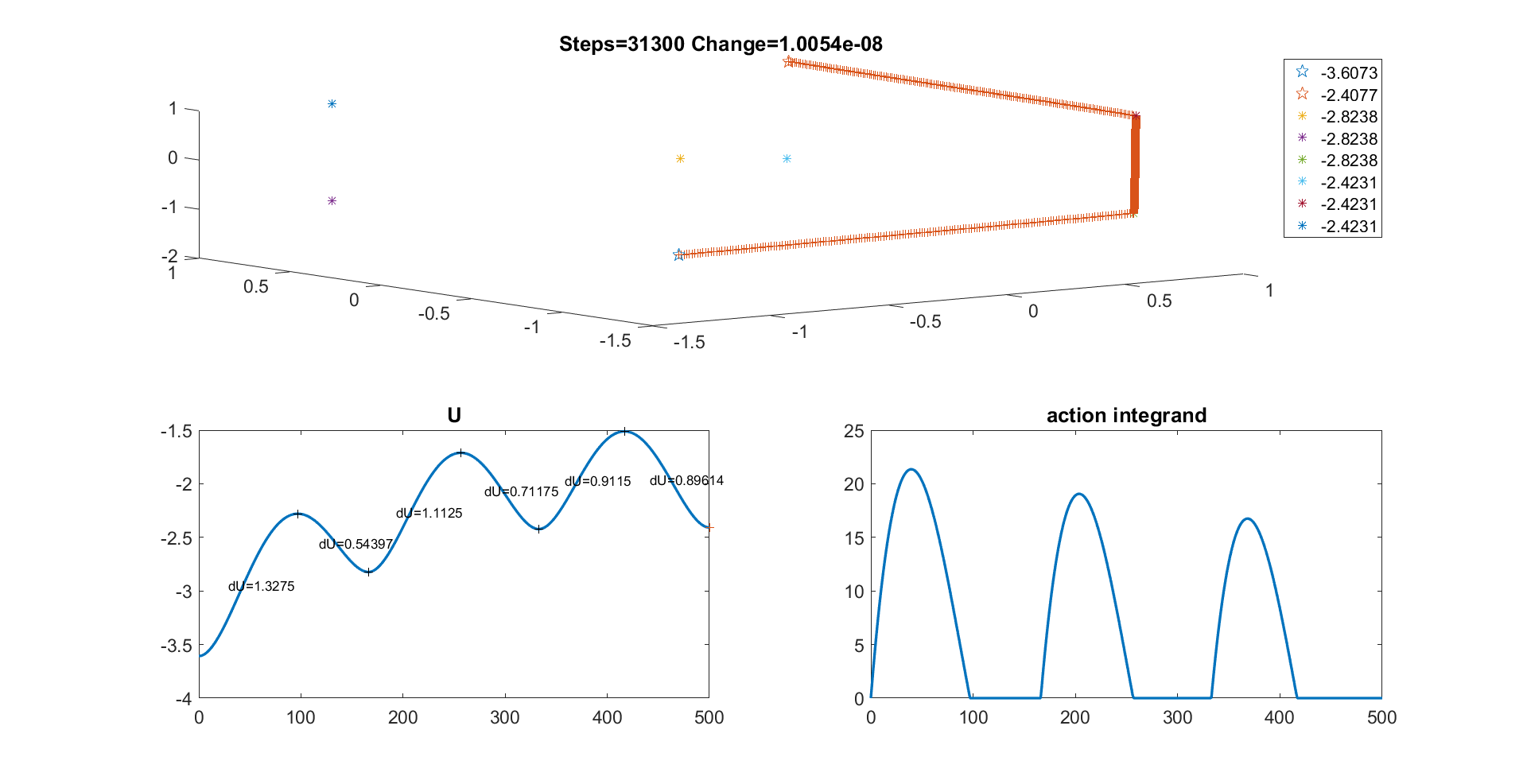}
					\caption{$\sigma=0.05$}	
				\end{subfigure}
				\begin{subfigure}{\textwidth}
					\hspace{-50pt}
					\includegraphics[width=1\textwidth]{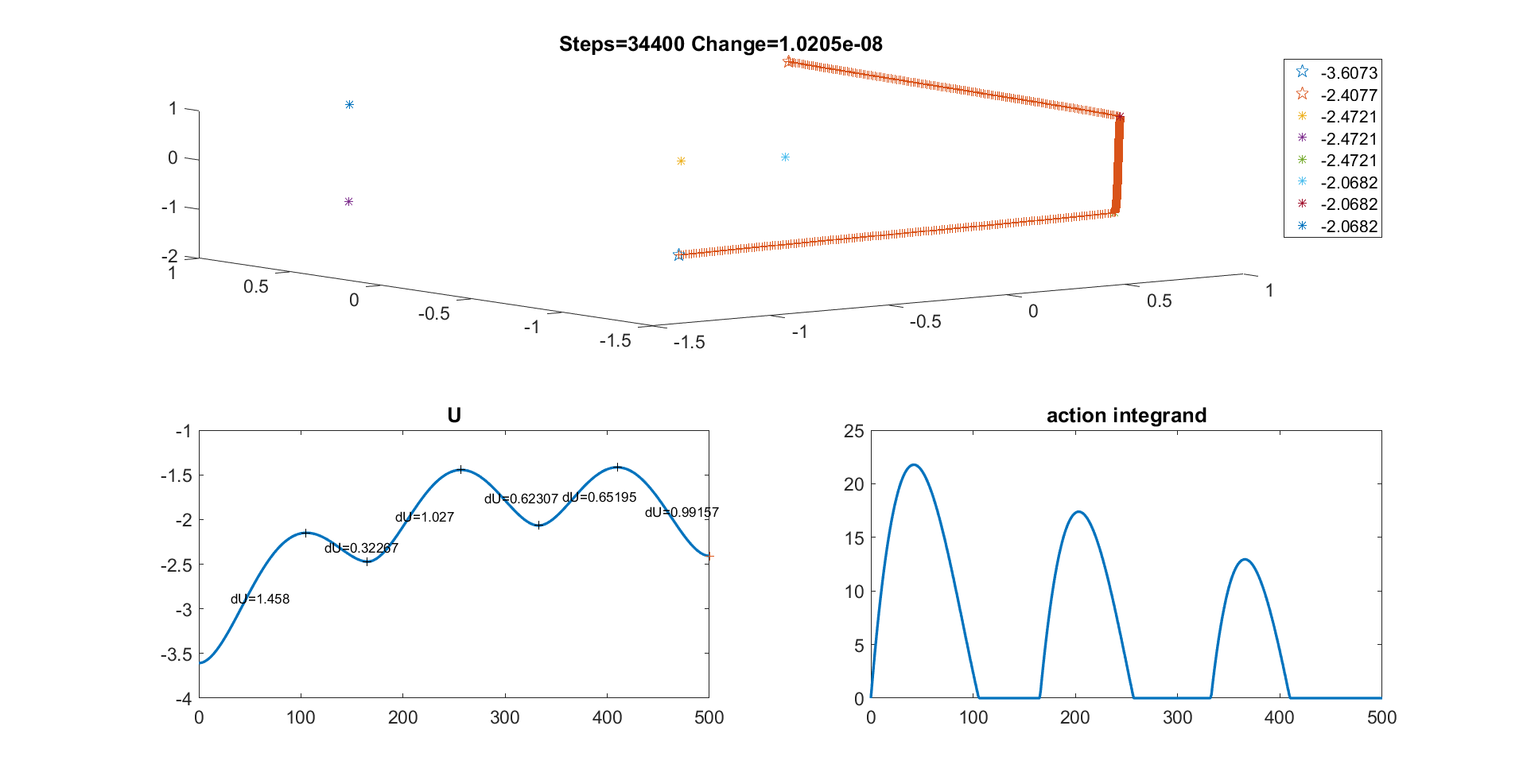}
					\caption{$\sigma=0.1$}	
				\end{subfigure}
				\caption{Minimum Energy Path (MEP) computed using the String method \cite{e2007simplified} at high numerical resolution. The MEP represents the most likely transition path between two metastable states  in the zero-temperature limit of overdamped Langevin dynamics. It is known (e.g., \cite{e2002string}) to reveal barrier heights and descent depths along the transition path, which are labeled by $\textrm{dU}$ values in the bottom-left panel. The top panel visualizes the MEP in three-dimensional space $(x_1,x_2,x_3)$, where the legend lists the values of the potential $U$ at each local minimum. The bottom-right panel displays the integrand of the Freidlin-Wentzell action functional  as a function of the arc-length parameterization of the path,  serving as a sanity check to verify the correctness of the computed MEP.}
				\label{fig:MEPbyString}
			\end{figure}				
			
\subsection{Rosenbrock function}
This example examines  the well-known non-convex landscape of the Rosenbrock function  in both two- and four-dimensional cases.  For $N\in\N_+$, the Rosenbrock
function is defined as
			\[
			R_{N}( \bs{x}) = \sum_{i = 1}^{N-1} [b (x_{i+1} - x_{i}^{2})^{2}
			+ (a - x_{i})^{2}],\quad\bs{x} \in \mathbb{R}^{N},
			\]
where $a$ and $b$ are  constants. In this study, the parameters are  chosen as $a = 1$ and $b = 20$. For $N = 2$, the function $R_{N}$ admits a unique 
minimum at $(1,1)$, while  for $N = 4$, it possesses a global minimum at
$(1,1,1,1)$ and a local minimum at $(-1,1,1,1)$. Figure
\ref{fig6} illustrates the function landscape: the top-left panel displays $\log R_{2}( \bs{x})$, and the bottom-left panel shows a slice of $\log R_{4}( \bs{x}) $ at $x_{3} = x_{4} = 1$. A logarithmic scale is used  to better visualize the detailed structure  near each minimum. In the vicinity of each minimum,  the  function exhibits a valley-like shape, remaining convex only within a very small neighborhood. The landscape of $R_{4}$ cannot be fully captured by a single heat map slice; however, it is straightforward to verify that the convex region of $R_{4}$ is relatively small. 
			
The noise magnitude is set as $\varepsilon = 0.001,$ $0.01,$ $0.1,$ $1.0,$ $1.5,$ and $2.0$  for $R_{2},$ and as $\varepsilon = 0.001,$ $0.003,$ $0.01,$ $0.03,$ $0.1,$ $0.3,$ and $1.0$ for $R_{4}$. The corresponding coupling time distributions are shown in the  two right  panels of Figure \ref{fig6}. It can be observed that for both cases, when the noise is sufficiently small, the tails of the coupling time distributions appear parallel in the log-linear plot. This behavior arises because  the	coupling time is primarily determined by the local convexity near the global minimum, in agreement with the result of Theorem \ref{thm:1}.  However, as the noise increases, trajectories are more likely to explore the entire valley rather than remaining confined to  the neighborhood of the global minimum. Consequently,  the coupling time distributions are altered.  
			
Another interesting phenomenon is that for the potential function $R_{4}$,  the coupling time distribution does not exhibit an exponentially small tail with respect to the noise magnitude, 
even when $\var=0.001$.
This contrasts with the theoretical results for the double-well potential.   Moreover, even when one of the coupled	processes is initialized  at the local minimum $(-1, 1,1,1)$, the tail of the coupling time distribution remains largely unchanged, as shown in the plot labeled  ``$\var= 0.001$ fixed''. This phenomenon occurs because the basin of the local minimum is shallow and separated by  a  low barrier, which  can be  easily crossed by a trajectory allowing it to quickly reach the valley of the global minimum. 		
			
			\begin{figure}[htbp]
				\centerline{\includegraphics[width = \linewidth]{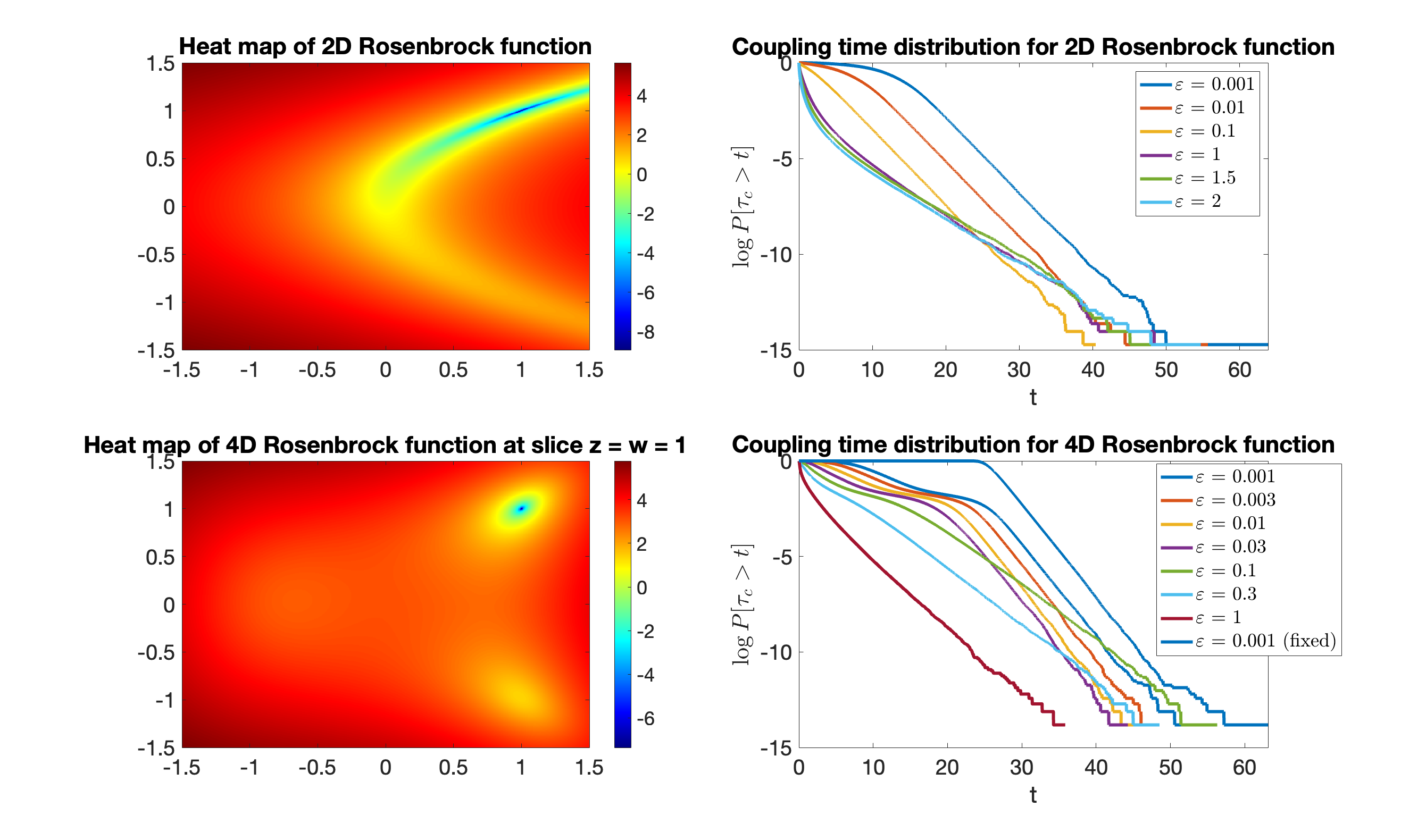}}
				\caption{Top left: Landscape of $R_{2}$. Top right:
					Coupling time distribution for $R_{2}$ under  different noise magnitudes. Bottom left: Landscape of $R_{4}$. Bottom right:
					Coupling time distribution for $R_{4}$ under  different noise magnitudes.}
				\label{fig6}
			\end{figure}

\subsection{Loss functions of artificial neural networks}
This subsection investigates the performance of the coupling method in a high-dimensional setting. Specifically, the training process of an artificial neural network (ANN) with two hidden layers is considered, where the first and second layers contain $N_1$  and $ N_{2}$ neurons, respectively. 
Let $\mathrm{ReLU}(z) = \max\{z ,0\}$ denote the rectified linear unit activation function.The ANN considered here is defined by
the following structure
			\begin{eqnarray}
				&&\mathbf{h}_{1} = \mbox{ReLU}(W_{1} \bs{x} + \mathbf{b}_{1})\label{NN1} \\
				&& \mathbf{h}_{2} = \mbox{ReLU}(W_{2} \mathbf{h}_{1} + \mathbf{b}_{2})\label{NN2}\\
				&&y = W_{3} \mathbf{h}_2 + b_{3}\label{NN3} \,,
			\end{eqnarray}
			where  $\bs{x} \in \mathbb{R}^{2}$ is the input and  $y \in \mathbb{R}$ is the output. The vectors    
            $\mathbf{b}_{1} \in \mathbb{R}^{N_{1}}, \mathbf{b}_{2} \in \mathbb{R}^{N_{2}}$ and $b_{3} \in
			\mathbb{R}$ denote  the bias terms. The weight matrices
			$W_{1}$, $W_{2}$, and $W_{3}$ have dimensions $N_{1} \times 2$, $N_{2}
			\times N_{1}$,  and $1 \times N_{2}$, respectively. Let  $\bm{\theta}$ denote the collection of  all trainable parameters, including the entries of  $W_{1}, W_{2}, W_{3}, \mathbf{b}_{1},
			\mathbf{b}_{2}$, and $b_{3}$.
			The total number of parameters is given by  \[\dim\bs\theta=(N_{1}N_{2} +
			3 N_{1} + 2 N_{2} + 1).\] 		
			For notational convenience,  the ANN defined by  \eqref{NN1} - \eqref{NN3} is denoted by $y =
			\mbox{NN}( \bm{\theta}, \bs{x}).$ 			
			
		The objective of the training process is to approximate the quadratic function $y = |\bm{x}|^2$ using the ANN.			
			Given a training set  $\{\bs{x}_{1}, \ldots, \bs{x}_{M}; y_{1}, \ldots y_{M}\}$, the  loss function is defined by 
			\[
			L( \bm{\theta}) =\dfrac{1}{M} \sum\nolimits_{i = 1}^{M} \left (y_{i} - \mbox{NN}(
			\bm{\theta}, \bs{x}_i) \right )^{2},
			\]		
			where the training set size is fixed at $M = 100$. The input points 
			$\bs{x}_{1}, \ldots, \bs{x}_{100}$ are uniformly sampled from
			$[-1\, , 1]^{2}$, and the corresponding target values are given by $y_{i} = | \bs{x}_{i} |^{2}$. The
			first column of Figure \ref{fig8} illustrates  the distribution of the collocation
			points and the target function $y =| \bs{x}|^{2}$.  The goal is to analyze the structure of the loss function $L(\bm{\theta})$.		
			
			The coupling method is applied to three ANNs with different hidden layer sizes: 
			$N_{1} = 4,N_{2} = 3$ (referred to as the ``small ANN''), $N_{1} = N_{2} = 10$ (the ``medium ANN''), and $N_{1} = N_{2} =
			20$ (the ``large ANN''). In this example, the small ANN  is under-parameterized, while the large ANN is over-parameterized. It is often believed that  over-parameterization tends to reduce  barrier heights in the loss landscape of ANNs (see, e.g., \cite{jacot2018neural, song2018mean, chizat2019lazy, mei2019mean, rotskoff2018trainability, sirignano2020mean, draxler2018essentially}). However, rigorous  justification remains elusive due to the complex structure of  high-dimensional loss functions. The  coupling-based approach proposed here may offer a viable tool in this regard by computing the essential barrier height  of such loss functions. 		
			
Figure  \ref{fig7} presents the coupling time distributions for the three neural networks under ten different noise magnitudes. For visual clarity, only five  noise levels are shown. As in  previous examples, the slopes are estimated via weighted linear regression. The six smallest values of $\varepsilon^{2}$ are used for the linear extrapolation of $y(\var):=-\var^2\log r(\varepsilon)$ versus $\varepsilon^{2}$, as displayed  in the lower panels.    
It is observed that the large ANN exhibits a lower essential barrier height.
More precisely, no significant barrier is detected within the region explored by the coupling method. This observation is consistent  with the findings in \cite{draxler2018essentially},  which adopts a different approach 
based on computing the MEPs between the local minima of the loss surface. 
Although it is theoretically possible that a high-barrier local minimum exists in a remote region not reached by the coupling trajectories, such cases have not been reported to the best of our knowledge. Moreover, practical ANN training is typically regularized, which prevents $|\bm{\theta}|$ from becoming excessively large.

The small ANN in this example is under-parameterized, as it contains only 31  parameters to be learned,  whereas the training set  comprises $100$ samples.  As illustrated in Figure \ref{fig7}, the loss function  of the small ANN exhibits a much larger essential barrier height compared to both the medium and large ANNs. Regarding the training performance, when initialized randomly,  the small ANN may converge to a ``bad" local minimum that fails to 
accurately approximate the target function (see the middle panels of Figure \ref{fig8}). 
In contrast, for all tested  initial	conditions, both the medium and large ANNs consistently converge to a ``good" local minimum of the training loss function, yielding satisfactory  
approximations of the target function  (see the right panels of Figure \ref{fig8}). This finding aligns with existing studies on the loss landscapes of ANNs \cite{choromanska2015loss, kawaguchi2019every, soltanolkotabi2018theoretical}.

			\begin{figure}[htbp]
				\centerline{\includegraphics[width = \linewidth]{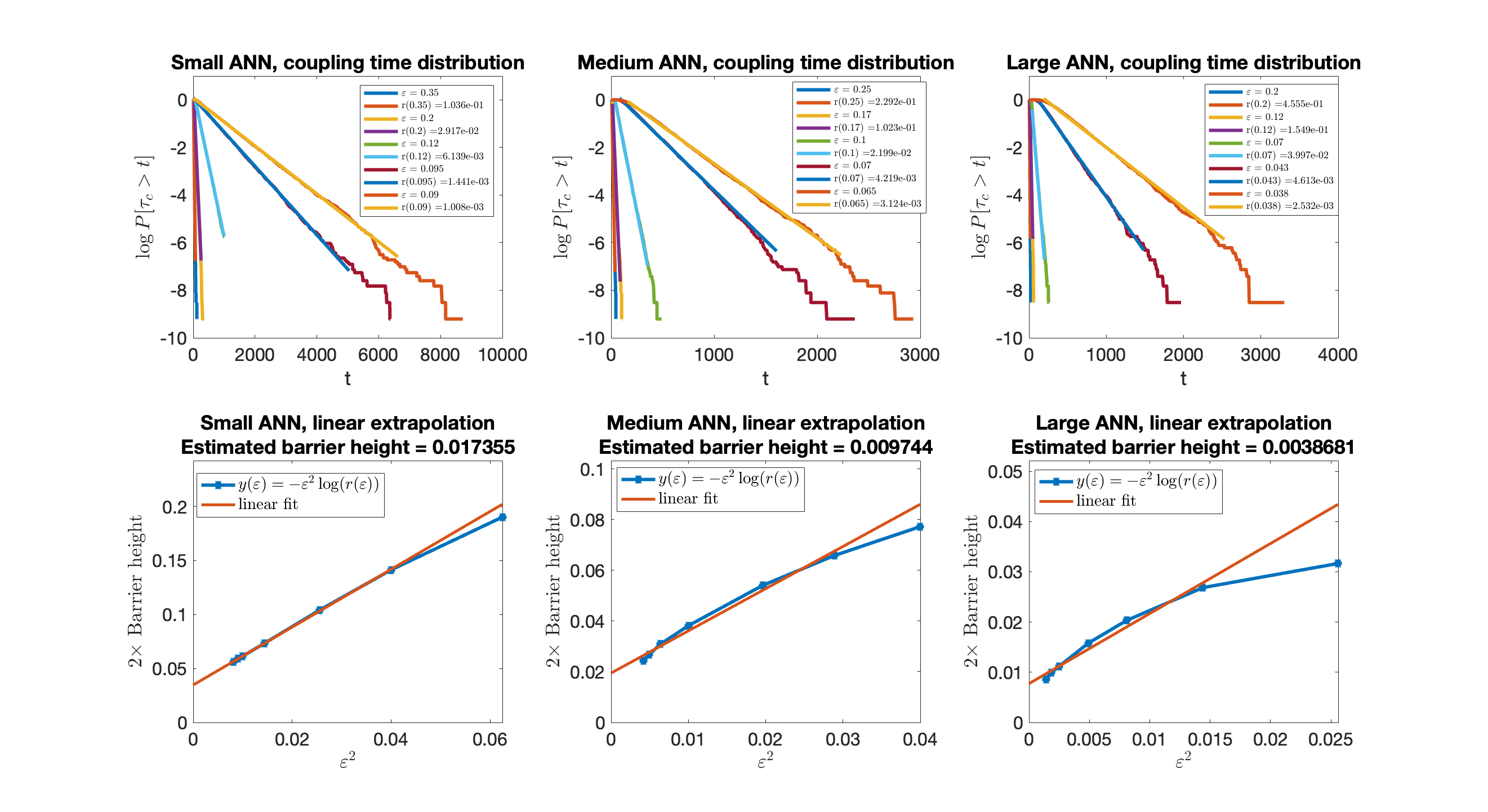}}
				\caption{Coupling time distributions and linear extrapolation of
					$y(\var)=-\var^2\log r(\varepsilon)$ versus $\varepsilon^{2}$. Left: small ANN;  middle: medium ANN; right: large ANN.}
				\label{fig7} 
			\end{figure}

			\begin{figure}[htbp]
				\centerline{\includegraphics[width = \linewidth]{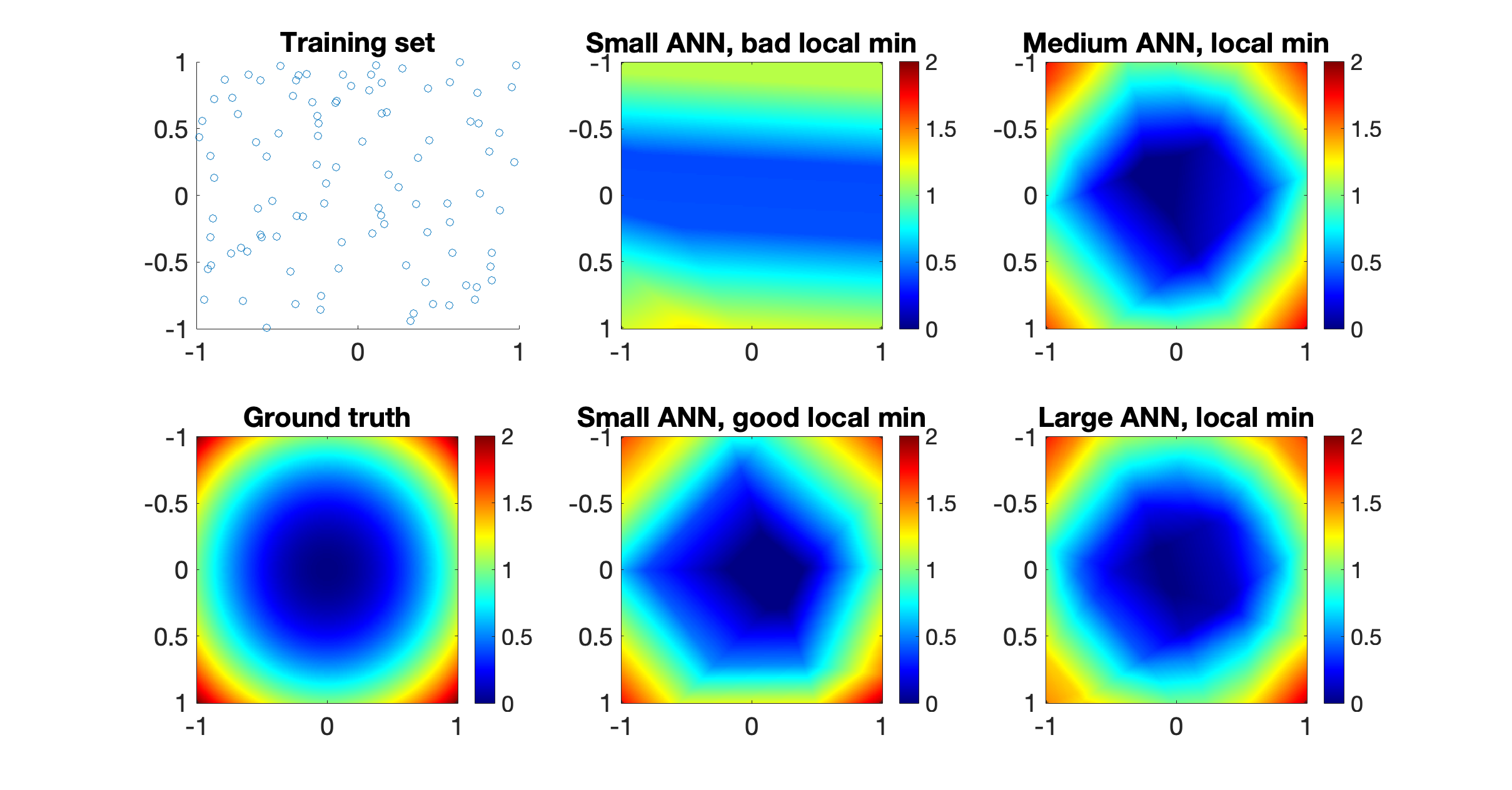}}
				\caption{Left column: Training set and target function. Middle column:
					$y =\mbox{NN}( \bm{\theta}, \bs{x})$ for $\bm{\theta}$ at the ``bad'' and ``good''
					local minima of the small neural network. Right column: $y =
					\mbox{NN}( \bm{\theta}, \bs{x})$ for the medium and large neural
					networks. 
				}
				\label{fig8} 
			\end{figure}

\subsection{Numerical verification of assumptions}\label{numerical verification}
In this subsection, assumptions {\bf (H1)}-{\bf(H3)} proposed in Section \ref{sec:double-multiple well} are numerically verified. In addition, the consistency between the first  passage times of the continuous-time process and its discrete-time counterpart 
is examined, as discussed in
Remark \ref{rem:approximate}.  		
		
\subsubsection{Numerical verification of Remark \ref{rem:approximate}.} We first numerically verify the  consistency of the first passage times defined in Remark \ref{rem:approximate} as two reflection-coupled trajectories approach each other. Specifically, we examine the assumption \eqref{approximate}
\[\lim\nolimits_{h \rightarrow 0} | \tau_{h}^{0} - \tau_{h}^{h} | = 0,\quad \text{$\P$-a.s.}\]
where $\tau_{h}^{0} = \inf\nolimits_{t > 0} \{ |X_{t} - Y_{t}| = 2 \sqrt{h}\}$ denotes the  continuous-time first passage time, and $\tau_{h}^{h} = h \cdot \inf\nolimits_{n > 0} \{ |X_{nh} - Y_{nh}| = 2 \sqrt{h}\}$ is its
discrete-time counterpart based on a time-$h$ sampled chain. 
			
This verification can be conducted  using an extrapolation argument. Let $h_{1} =h/n$ for some integer $n$,  and define the first passage time of the time-$h_1$ sampled chain by
\[\tau_{h}^{h_{1}} = h_{1} \cdot \inf\nolimits_{n > 0} \{ |X_{nh_{1}} - Y_{nh_{1}}| = 2 \sqrt{h} \}.\]
By the strong approximation property of the Euler-Maruyama scheme for SDEs, 	\[\lim\nolimits_{h_{1}\rightarrow 0} \tau_{h}^{h_{1}}  = \tau_{h}^{0},\quad\P\text{-a.s.},\]
so it suffices to compare  $\tau_{h}^{h}$ and $\tau_{h}^{h_{1}},$ with the latter serving as an approximation of  $\tau_h^0$ for the same trajectory. 
			
We apply the extrapolation method to  both the quadratic potential function  in Section \ref{sec52} and the interacting particle system  in Section \ref{IPS}. In the quadratic potential  case, the initial values of $X_t$ and $Y_t$ are set to $(0.5, 0.7)$ and $(-0.5, -0.6),$ respectively. For the interacting particle system,  $X_t$ and $Y_t$ are initialized at $(1,1,1)$ and $(-1, -1, -1)$, respectively, indicating  that they belong to the basins of different local minima.
			
In the top-left and bottom-left panels of Figure \ref{fig2}, the quantity $(\tau_{h}^{h}-\tau_{h}^{h_{1}})$ is plotted against $\sqrt{h_{1}}$ for five different values of $h.$
In both examples, this difference exhibits approximately linear behavior as $\sqrt{h_{1}}\to0.$ 
 An extrapolation at $h_{1} = 0$ provides an estimate of $(\tau_{h}^{h}-\tau_{h}^{0})$.  The top-right and bottom-right panels of Figure \ref{fig2}
 display 
 $(\tau_{h}^{h} - \tau_{h}^{0})$ 
 versus $\sqrt h$ for $h = 0.0002, 0.0005, 0.001, 0.005$, and $0.01$, respectively. The results show that this error  decreases as $h\to0,$ and a linear fit suggests  that $(\tau_{h}^{h} -\tau_{h}^{0})$  is approximately proportional to $\sqrt{h}$, consistent with the  findings of \cite{Gobet2004exact, Gobet2010stopped}. Although the error in estimating $\tau_h$ is larger for the interacting particle system due to  the presence of multiple local minima, the numerical results still exhibit the  expected convergence behavior as $h \rightarrow 0$.
			
			\begin{figure}[htbp]
				\centerline{\includegraphics[width = \linewidth]{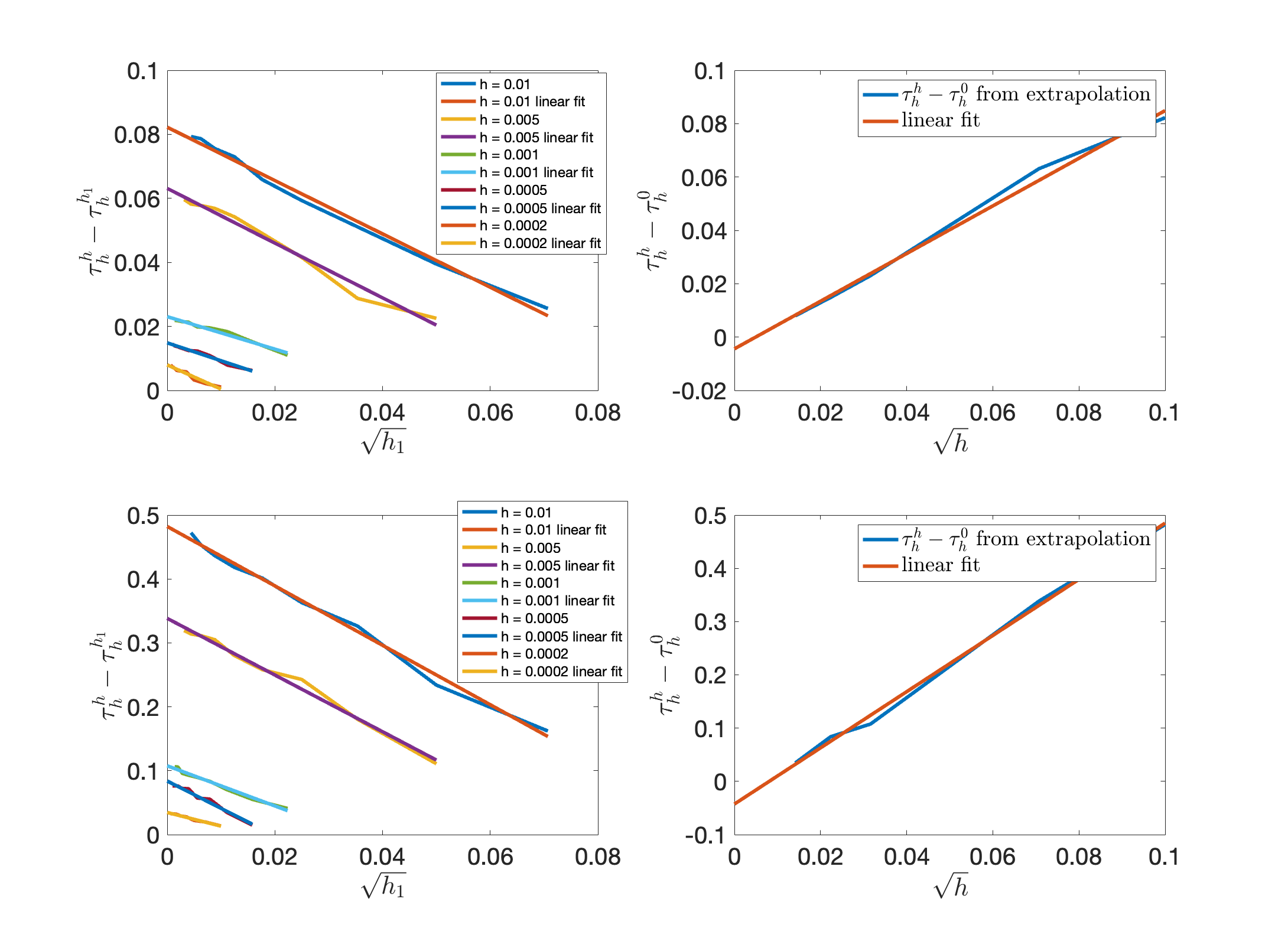}}
				\caption{Left: $(\tau_{h}^{h} - \tau_{h}^{h_{1}})$ vs. $\sqrt{h_{1}}$
					for five different values of $h$. 
                    Top left: Quadratic potential function. Bottom left: Interacting particle system. 
                    Right: $(\tau_{h}^{h} - \tau_{h}^{0})$					vs. $\sqrt{h}$ with a linear fit. 
                    Top right: Quadratic potential function. Bottom right: Interacting particle system.}
				\label{fig2}
			\end{figure}

			\medskip
			
In the following subsections, the interacting particle system described in Section \ref{IPS} is used to numerically  verify  assumptions {\bf (H1)-\bf(H3)}. The coupling strength is set to $\sigma = 0.05$.			
			
\subsubsection{Numerical verification of  {\bf (H1)}}\label{sec:H1}
Let $X_0 = (1,1,1)$ and  $Y_0 = (-1,-1, -1)$, ensuring that the trajectory $Y_t$ is initiated  near the global minimum. Based on the barrier heights illustrated in Figure \ref{fig:MEPbyString} and the definitions of $\mathcal I$,  the set $\bs B_1$ is identified as the complement of the basin containing $(1,1,1)$. The simulation is performed under four different noise magnitudes: $\varepsilon= 0.6, 0.65, 0.7, $ and $0.75$.  At each step of the  Euler-Maruyama scheme,it is numerically checked whether $X_t$ and $Y_t$ lie in $\bs B_1$.
The criterion for determining whether a point $\bs{x} = (x_1, x_2, x_3)$ belongs to $\bs B_1$
is as follows: for each $i = 1, 2, 3$, if either $x_i > 0.11$, or  $0 \leq x_i < 0.11$ while $-\partial U /\partial x_i > 0$, then $\bs x\notin\bs B_1$. This condition is sufficient for all  samples in our numerical simulation. 
			
Remarkably, across tens of millions of samples, $Y_t$ was  {\it never} observed to exit $\bs B_1$ before $X_t$ entered it. A similar phenomenon is observed in the one-dimensional double-well potential, where $Y_t$ remains in $B_1$ until $X_t$ enters.  This can be explained by noting that reflection-coupled Brownian motions have the same action functionals. When $X_t$ exits the basin $B_2$ and enters $\bs B_1$, the action functional of the associated driving Brownian motion of $X_t$ is highly likely to be close  to $H_U$. Consequently, the action functional corresponding to the Brownian motion term in $Y_t$ is unlikely to be  sufficiently large  to drive $Y_t$ out of $\bs B_1$. Therefore, assumption {\bf (H1)} is numerically verified  with the even stronger conclusion that
\begin{eqnarray*}
\P\big[Y_s\in \bs B_1\ \text{for all}\  s\in[0,t]\big| \ka_{\bs X}(\bs B_1)>t\big] \approx 1.\end{eqnarray*}

\subsubsection{Numerical verification of {\bf (H2)}}
Let $B_2$ denote the basin of attraction  containing  $(1,1,1)$. According to Proposition \ref{pfH2} and Remark \ref{rem:weaker_condition}, it suffices to verify condition (a) therein. Specifically,  this involves  numerically estimating  the probability that a trajectory  enters the interior of $B_2$.
The approximate boundary of $B_2$ is depicted in the left panel of Figure \ref{figIPSH2}.  
It suffices  to consider initial points from the  boundary, as the probability is expected to be higher when starting from the interior. 

Three initial values of $X_0$ 
are selected from a corner, an edge, and a face of $\partial B_2$, respectively, as marked in red in Figure \ref{figIPSH2} (Left). The initial value of $Y_0$ is fixed at $(1,1,1)$.
For each case, the probability that the coupled process $(X_t, Y_t)$ remains in the $\delta$-interior of $B_2 \times B_2$ (with $\delta = 0.01$) throughout the time interval $[h,T_0]$ is computed,  where $h=0.05,$ $T_0 = -\log\var$, and  $\varepsilon$ varies from $0.001$ to $0.01$. As shown in Figure \ref{figIPSH2}, this probability remains uniformly bounded from below as $\var \rightarrow 0$. This numerically confirms condition (a) of Proposition \ref{pfH2}, thereby verifying assumption {\bf (H2)}.

\begin{figure}\centering
\includegraphics[width = \linewidth]{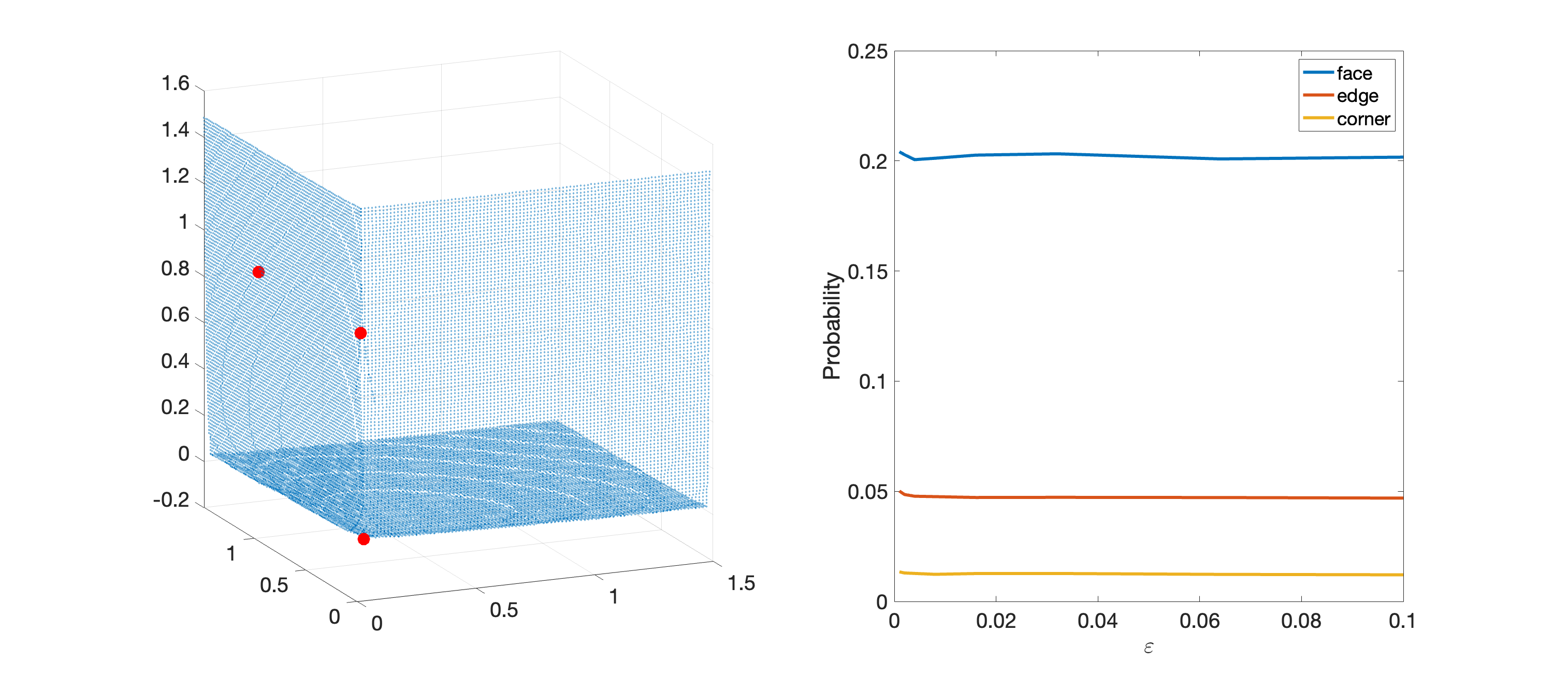}
\caption{Left: Approximate boundary of $\bs B_1$, with three initial values of $X_0$ marked in red.  Right: 
Probability that $(X_t,Y_t)$ remains in the $\delta$-interior of  $\bs B_1$ over the time interval $[h,T_0]=[0.05, -\log\var]$.}
\label{figIPSH2}
\end{figure}
			
\subsubsection{Numerical verification of  {\bf (H3)}} 
Assumption {\bf (H3)} is numerically verified by computing the overshoot time. The criterion for determining whether a trajectory  enters the basin $B_1$ is the same as that used in section \ref{sec:H1}.					
The noise magnitudes are set to  $0.5, 0.55,$ and $0.6$.  For each value of $\var$, the probability distribution of the overshoot time, given by $\xi_1 - \max \{ \kappa_{\bs X}, \kappa_{\bs Y}\}$, is estimated  using $1 \times 10^7$ samples.

As illustrated in Figure \ref{figH3}, the tail distribution of $\xi_1 - \max \{ \kappa_{\bs X}, \kappa_{\bs Y}\}$ exhibits a two-phase behavior.
The second phase corresponds to the scenario where one of the trajectories, ${X}_t$ or ${Y}_t$,  makes an excursion to other basins after entering  $B_1$ and subsequently returns, while the other trajectory remains within $B_1$.  
Due to the low probability of 
such an event, a large number of samples are required to capture the exponential tail. In Figure \ref{figH3}, the distributions of the  overshoot time and coupling time are compared. In the log-linear plot,  the slope of the overshoot time decreases rapidly  as the noise magnitude decreases, yet it  remains steeper than that of  the coupling time.  
As the theoretical result indicates that the tail of the coupling time distribution is close to the essential barrier height $H_U$, this numerical observation thereby verifies assumption {\bf (H3)}.

			\begin{figure}
				\centering
				\includegraphics[width = \linewidth]{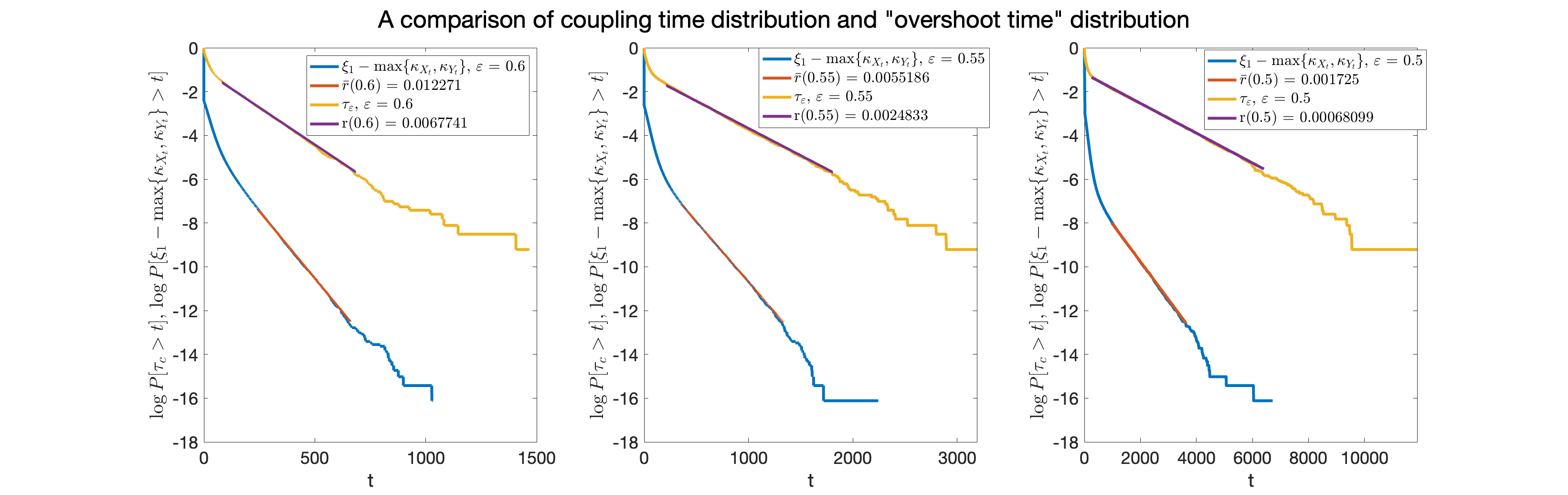}
				\caption{Comparison of probability distributions of the overshoot time $\xi_1 - \max \{ \kappa_{\bs X}, \kappa_{\bs Y}\}$ and the coupling time. The left, middle, and right panels correspond to $\varepsilon = 0.6, 0.55$, and $0.5$ respectively. The asymptotic slope of the overshoot time distribution in the log-linear plot is denoted by $\bar{r}$.}
				\label{figH3}
			\end{figure}		

\section{Conclusion and further discussions}	
This paper investigates  the relationship between the geometry of a multi-dimensional potential landscape and the distributions of coupling time for the overdamped Langevin system associated with the  potential. This study is motivated by the fact that the exponential tail of the coupling time  distribution provides a lower bound for the spectral gap of the Fokker-Planck operator governing the Langevin dynamics. It has long been believed that certain geometric properties of a region can  be inferred from the spectrum of an associated differential operator,  as famously illustrated by Kac's question   \cite{kac1966}, ``Can one hear the shape of a drum?"  
In a similar spirit, this work takes a preliminary step toward understanding the structure of a 
potential landscape by  establishing connections between its geometry and the statistical properties of coupling times.  

It is shown that, in the limit of vanishing noise, the exponential tails of the coupling time distributions exhibit {\it  qualitatively} distinct behaviors in the single-well potential and multi-well settings. 
Specifically, for a strongly convex single-well potential, the rate of exponential tail is uniformly bounded below by a constant that depends on the convexity of the potential. 
In contrast, for a multi-well potential, the rate of the exponential tail decays  exponentially as the noise strength tends to zero. These results are supported by both theoretical analysis  and numerical verification. 

The coupling scheme used in this paper combines  reflection coupling and maximal coupling  to improve efficiency.
It is observed that the upper bound on the tail distribution 
obtained through  this scheme is  close to  optimal, in the sense that it nearly achieves equality in the coupling inequality. 
To estimate the exponential decay rate of the tail in the small noise regime, a linear extrapolation is employed. 
This decay rate is governed by the essential barrier height, a concept introduced in this paper  to capture the global structural features of the  potential landscape.
In particular, the  essential barrier height is applied to analyze the loss landscape of artificial  neural networks, and the corresponding numerical observations are consistent with findings from related studies employing alternative methodologies.

 Although this work focuses on the distribution of coupling times, further information about the potential landscape 
 is expected to be extracted
 from the coupling-based analyses. 
For instance, the distribution of coupling locations may provide  additional insights into the geometry of the underlying landscape.
Furthermore, the present study  only concerns the tail of the coupling time distribution, which is associated with the principal eigenvalue of the Fokker-Planck operator. 
 An investigation of  conditional coupling times -- specifically,  
 those conditioned on avoiding  coupling in the deepest well --
 could  reveal spectral information associated with non-principal eigenvalues,  which correspond to the lower energy barriers.  
These  extensions represent promising directions for future research.

\appendix
\section{Proof of Lemma \ref{lem:expon_tail_multi_well}}\label{appendix_1}

To establish the exponential tail, more refined estimates of the eigenvalues of the Dirichlet operator are required. Specifically, let $\mathcal{L}_D^\var$ denote the infinitesimal generator of the process defined in \eqref{SDE1}, and let $D \subset\mathbb{R}^k$ be an open set with a regular boundary $\partial D$.
Based on potential theory, 
\cite{metastability2004, metastability2005} provide sharp estimates for the low-lying eigenvalues and the corresponding  eigenfunctions of 
the Dirichlet problem
\begin{align}
    \label{dirichlet}
    \mathcal{L}_D^\var u - \lambda u = 0,& \quad \mbox{ in } \quad D\\\nonumber
    u = 0,& \quad\  \mbox{in} \quad\  D^c \nonumber
\end{align}

The following Proposition \ref{prop:bovier_ev} summarizes the sharp bounds of the principal eigenvalues, as established in Proposition 3.2 of \cite{metastability2005} and Theorem 3.1 of \cite{metastability2004}. The corresponding bound on the principal eigenfunction is stated in Proposition \ref{prop:bovier_ef}, which follows from Proposition 3.3 in \cite{metastability2005}.
\begin{prop}\label{prop:bovier_ev}
\noindent{\bf(Sharp bound on eigenvalues \cite{metastability2004,metastability2005})}
Assume that $D \subseteq \mathbb{R}^d$ is open and let $U:D\to\R$ be a potential function satisfying {\bf (U3)}. Suppose that $D$ contains 
$L\geq 1$ local minima of $U$, and that there exists a unique minimum $x \in D$ such that
$$
U(z^*(x, D^c)) - U(x) = \max\nolimits_{1\le i\le L} \{ U(z^*(x_i, D^c)) - U(x_i)\}.
$$
Let $B\equiv B_x$ denote a neighborhood of $x$, and denote the first entrance time of $X_t$ into any subset $A \subseteq \mathbb{R}^d$ by $\tau_A$. Then there exist constants $\alpha > 0$, $C < \infty$, and $\delta > 0$, independent of $\var$, such that the principal eigenvalue $\lambda_1<0$  of $\mathcal{L}_D^\var$ satisfies
$$
\frac{\mathrm{cap}_B(D^c)}{\|h_{B, D^c}\|_2^2}(1 - C \var^\alpha)(1 - e^{-\delta/\var^2})\leq
|\lambda_1|\leq \frac{\mathrm{cap}_B(D^c)}{\|h_{B, D^c}\|_2^2}(1 + C \var^\alpha)(1 + e^{-\delta/\var^2}) \,,
$$
where $h_{B, D^c}(z) := \mathbb{P}_z[\tau_B < \tau_{D^c}]$, the norm $\|\cdot\|_2$ is taken with respect to the  invariant probability measure $\pi^\var$ of \eqref{SDE1}, and capacity
\[
\mathrm{cap}_B(D^c) = e^{-2U(z^*)/\var^2}\frac{(2 \pi \var)^d}{2 \pi} \frac{|\lambda_1^*(z^*)|}{\sqrt{|\mbox{det}(\nabla^2 U(z^*))|}}(1 + O(\var|\ln \var|))
\]
for $z^* = z^*(B, D^c)$. Here,  $\lambda^*_1(z^*)$ denotes the negative eigenvalue of the Hessian of $U$ at $z^*$.
\end{prop}

\begin{prop}\label{prop:bovier_ef}
\noindent{\bf(Sharp bounds on eigenfunctions \cite{metastability2004,metastability2005})}
Under the assumptions of Proposition \ref{prop:bovier_ev}, let $\phi_1$ be the eigenfunction of $\mathcal{L}^\var_D$ corresponding to $\lambda_1$, normalized such that $\inf_{x \in \partial B}\phi_1 = 1$. Then
\begin{eqnarray*}
h_{B, D^c}(z) \leq \phi_1(z) \leq h_{B, D^c}(z)(1 + C \var^\alpha)(1 + e^{-\delta/\var^2}).
\end{eqnarray*}
\end{prop}

\bigskip

These estimates on the principal eigenvalue and eigenfunction yield the exponential tail for the first hitting time in Lemma \ref{lem:expon_tail_multi_well}. 

\proof[Proof of Lemma \ref{lem:expon_tail_multi_well}]
The argument is a modification of Theorem 1.4 from \cite{metastability2005}. Assume $z \notin B_1$, since the bound is trivial otherwise. Set $D = B_1^c$ and $B = B_2$. Then
	\[
	\mathrm{cap}_B(D^c) = e^{-2U(z^*)/\var^2} \frac{(2\pi \var)^d}{2\pi} \frac{|\lambda_1^*(z^*)|}{\sqrt{|\det(\nabla^2 U(z^*))|}}(1 + O(\var|\ln \var|)),
	\]
	where $z^* = z^*(x_1, x_2)$. 

Note that for sufficiently small $\var > 0$, if $h_{B,D^c}(z) \simeq 1$ for $z \in B_i$, then it follows that $\Phi(x_i,x_2) < \Phi(x_i,x_1)$. This, in turn, implies $U(x_i) > U(x_2)$; otherwise, we would have
\[
\Phi(x_2,x_1) - U(x_2) < \Phi(x_i,x_1) - U(x_i),
\]
which contradicts \eqref{small_H_U}. In particular, $h_{B, D^c}(z) \simeq 1$ for $z \in B_2$. Therefore, integrating against $\pi^\var$, it follows that
\[
\|h_{B, D^c}\|_2^2 \simeq e^{-2U(x_2)/\var^2}.
\]
Noting that $H_U = U(z^*) - U(x_2)$, it follows from Proposition \ref{prop:bovier_ev} that
	\[
	|\lambda_1| = e^{-2 H_U / \var^2}(1 + O(\var^\alpha))(1 + O(e^{-\delta/\var^2})).
	\]
	Now observe that
	\[
	\mathbb{P}_z[\kappa_{\bs Z}(B_1) > t] = \left(e^{t \mathcal{L}_D^\var} \mathbf{1}_{B_1^c}\right)(z),
	\]
	so that
	\[
	\mathbb{P}_z[\kappa_{\bs Z}(B_1) > t] \le A_{z,\var} e^{-\lambda_1 t}
	\]
	for some constant $A_{z,\var}$ depending on $z$ and $\var$.

For the lower bound, note that $\mathcal{L}_D^\var$ is self-adjoint in the weighted space $L^2_{\pi^\var}$, and thus its eigenfunctions form an orthogonal basis. It follows that
	\[
	\mathbb{P}_z[\kappa_{\bs Z}(B_1) > t] = \left(e^{t \mathcal{L}_D^\var} \mathbf{1}_{B_1^c}\right)(z) 
	\ge \mathbb{E}_{\hat{\pi}^\var}[\phi_1 \mathbf{1}_{B_1^c}] \cdot e^{-\lambda_1 t} \cdot \phi_1(z),
	\]
	where $\hat{\pi}^\var$ denotes the normalized restriction of $\pi^\var$ to $D$. By Proposition \ref{prop:bovier_ef}, 
	\[
	\phi_1(z) \simeq h_{B, D^c}(z)(1 + O(\var^\alpha)).
	\]
Thus, for $z \in B$, one has $h_{B, D^c}(z) \simeq 1$, and hence
	\[
	\mathbb{P}_\mu[\kappa_{\bs Z}(B_1) > t] \ge \mathbb{E}_{\hat{\pi}^\var}[\phi_1 \mathbf{1}_{B_1^c}] \cdot e^{-\lambda_1 t} \cdot \E_\mu(\phi_1).
	\]
Since $\mu$ is fully supported, the term $\E_\mu(\phi_1)\simeq \E_\mu(h_{B, D^c}) > 0$, yielding the desired lower bound.
\qed

\begin{remark}
{\rm 
The leading-order term of \( A_{z, \var} \) is proportional to \( \phi_1(z) \). Furthermore, if \( z \) lies in the interior of \( B_2 \), then \( \mathbb{P}_z[\tau_{B_1} > \tau_{B_\var(x_2)}] \simeq 1 \). Consequently, the leading-order term of \(A_{z, \var}\) can be bounded by a constant that is independent of both \( \var \) and \( z \). Any dependence on \( \var \) and \( z\) arises exclusively through the coefficients associated with higher-order eigenfunctions in the spectral decomposition of the semigroup. On the time scale \( \kappa_{\bs Z}(B_1) = O(e^{2 H_U/\var^2}) \), these higher-order terms are of order \( O(e^{-\delta/\var^2}) \), making the prefactors \( A_{z, \var} \) and \( A_{\mu, \var} \) effectively independent of \( \var \).
}
\end{remark}

\begin{remark}\label{rem:hitting_time}
{\rm The result in Lemma \ref{lem:expon_tail_multi_well} still holds if \( B_1 \) is replaced by $\bs B_1$, as defined in \eqref{collection_B1}. This follows from \eqref{I_equivalent}, which ensures that, when \( D = \bs B_1^c \), the deepest local minimum remains to be \( x_2 \) and the height of the saddle \( z^*(B, D^c) \) remains to be \( U(z^*(x_1, x_2)) \). Hence, the proof of Lemma \ref{lem:expon_tail_multi_well} continues to hold. 
}
\end{remark}

\section{Proof of Proposition \ref{pfH2}}\label{appendix_2}
We first establish  that the event in which $X_{t}$ and $Y_{t}$ remain within the same basin and couple within a finite  time interval of order $\mathcal O(-\log \epsilon)$ occurs with a strictly positive probability.
This  directly implies {\bf (H2)}(i). 

By assumption (a), for any initial value $(X_0,Y_0)\in\overline{B_i\times B_i}$, the pair $(X_{T_{0}}, Y_{T_{0}})$ belongs to the $\delta$-interior $B^i_\delta\times B^i_\delta$ with probability
$\gamma_{0}$.  Since $B_{i}$ is the basin of attraction
of $x_{i}$, denote by
$\mathbf{x}_{t}$ the deterministic gradient flow $\dot{\mathbf{x}}_{t} = - \nabla U(\mathbf{x}_{t})$. Then there exists a constant   $T_{1}=\mathcal O(1)$ such that for any $\mathbf{x}_{0}
\in B^i_\delta$, 
the deterministic trajectory satisfies  $\mathbf{x}_{T_1} \in {B}^{i}_{c,1}\subset {B}^{i}_{c}$, where $B^i_{c,1}$ is an open subset in the interior of $B^i_c$.

By the standard small random perturbation argument (see, for instance, Chapter 4, Lemma 2.1 of \cite{freidlin1998random}), for any $\varepsilon > 0$ sufficiently small and any finite time interval, both processes $X_t$ and $ Y_t$ remain close to the deterministic trajectory $\mathbf{x}_t$ with high probability, say at least $0.9$. 
 Thus, combining this with assumption (a), define the event
\[
E_0 := \left\{ (X_t, Y_t) \in B_i \times B_i \ \text{for all } t \in [h, T_0 + T_1), \text{ and } (X_{T_0 + T_1}, Y_{T_0 + T_1}) \in B^i_{c,1} \times B^i_{c,1} \right\}.
\]
Then 
\[
\mathbb{P}[E_0 \mid (X_0, Y_0) \in \overline{B_i \times B_i}] \ge 0.9 \gamma_0.
\]

Let $\tilde{U}$ be a strongly convex potential satisfying \textbf{(U1)} such that $U = \tilde{U}$ on $B^i_c$. Denote by $(\tilde{X}_t, \tilde{Y}_t)$ the coupled process associated with $\tilde{U}$. Then $(X_t, Y_t)$ coincides with $(\tilde{X}_t, \tilde{Y}_t)$ as long as $(\tilde{X}_t, \tilde{Y}_t)$ remains in $B^i_c \times B^i_c$.
By Theorem~\ref{thm:1}, for any $\gamma_2 \in (0,1)$, there exists $T_2 = \mathcal{O}(1)$ such that
\[
\mathbb{P}[\tilde{X}_t \text{ and } \tilde{Y}_t \text{ couple before } T_2 \mid (\tilde{X}_0, \tilde{Y}_0) \in B^i_{c,1} \times B^i_{c,1}] \ge \gamma_2.
\]
Moreover, since $(\tilde{X}_t, \tilde{Y}_t)$ remains in $B^i_{c,1} \times B^i_{c,1}$ when $\varepsilon = 0$, the small random perturbation argument yields that for any $\var>0$ sufficiently small and any $\gamma_3 \in (0,1)$, there holds
\[
\mathbb{P}[(\tilde{X}_t, \tilde{Y}_t) \in B^i_c \times B^i_c \text{ for all } t \le T_2 \mid (\tilde{X}_0, \tilde{Y}_0) \in B^i_{c,1} \times B^i_{c,1}] \ge \gamma_3.
\]
Choose $\gamma_2$ and $\gamma_3$ such that $\gamma_2 + \gamma_3 > 1$, and define the events
\[
E_1 := \{ \tilde{X}_t \text{ and } \tilde{Y}_t \text{ couple before } T_2 \}, \quad E_2 := \{ (\tilde{X}_t, \tilde{Y}_t) \in B^i_c \times B^i_c \text{ for all } t \le T_2 \}.
\]
Then 
\begin{align*}
	\mathbb{P}[E_1 \cap E_2 \mid (\tilde{X}_0, \tilde{Y}_0) \in B^i_{c,1} \times B^i_{c,1}]
	&\ge \mathbb{P}[E_1 \mid (\tilde{X}_0, \tilde{Y}_0) \in B^i_{c,1} \times B^i_{c,1}] \\
	&\quad + \mathbb{P}[E_2 \mid (\tilde{X}_0, \tilde{Y}_0) \in B^i_{c,1} \times B^i_{c,1}] - 1 \\
	&\ge \gamma_2 + \gamma_3 - 1 > 0.
\end{align*}
Since $(X_t, Y_t)$ coincides with $(\tilde{X}_t, \tilde{Y}_t)$ on $E_1 \cap E_2$, it follows that
\[
\mathbb{P}[X_{T_0 + T_1 + t} = Y_{T_0 + T_1 + t} \text{ for some } t \in [0, T_2] \mid (X_{T_0 + T_1}, Y_{T_0 + T_1}) \in B^i_{c,1} \times B^i_{c,1}] \ge \gamma_2 + \gamma_3 - 1 > 0.
\]

Combining all estimates above, there exists a constant $T := T_0 + T_1 + T_2 = \mathcal{O}(-\log \varepsilon)$ such that
\begin{align}
	&\mathbb{P}[(X_t, Y_t) \in B_i \times B_i \ \text{for all } t \in (h, T],\ \text{and } X_t = Y_t \ \text{for some } t \in (h, T] \mid (X_0, Y_0) \in \overline{B_i \times B_i}]\nonumber\\
	&\ge \mathbb{P}[E_0 \mid (X_0, Y_0) \in \overline{B_i \times B_i}] \cdot \mathbb{P}[X_{T_0 + T_1 + t} = Y_{T_0 + T_1 + t} \text{ for some } t \in [0, T_2] \mid (X_{T_0 + T_1}, Y_{T_0 + T_1}) \in B^i_{c,1} \times B^i_{c,1}]\nonumber\\
	&\ge 0.9 \gamma_0 (\gamma_2 + \gamma_3 - 1) :=\gamma> 0.\label{P>ga}
\end{align}
This completes the verification of \textbf{(H2)}(i), where $\ga_1=1-\ga_0.$

\medskip

To prove {\bf (H2)}(ii), observe that $\tau_\var$ denotes the first time at which either $X_t$ and $Y_t$ couple or one of them exits the basin $B_i$. Hence, prior to time $\tau_\var$, the processes $X_t$ and $Y_t$ remain in the same basin and have not yet coupled. Consequently, for any $t \ge T$, it follows that
\[
\mathbb{P}[ \tau_\var > t \mid (X_{0}, Y_{0}) \in \overline{B_i \times B_i} ] 
\leq \prod_{n = 1}^{\lfloor t/T \rfloor} \mathbb{P}[ \tau_\var \circ \theta^{(n-1)T} > T \mid (X_{(n-1)T}, Y_{(n-1)T}) \in \overline{B_i \times B_i} ].
\]
Recalling that $T = \mathcal{O}(-\log \var)$ and applying \eqref{P>ga}, one obtains
\[
\mathbb{P}[ \tau_\var \circ \theta^{(n-1)T} > T \mid (X_{(n-1)T}, Y_{(n-1)T}) \in \overline{B_i \times B_i} ] \le \gamma_1,
\]
uniformly for all $(X_{(n-1)T}, Y_{(n-1)T}) \in \overline{B_i \times B_i}$. Therefore, for any $t > T$,
\[
\mathbb{P}[ \tau_\var > t \mid (X_{0}, Y_{0}) \in \overline{B_i \times B_i} ] \le \gamma_1^{\lfloor t/T \rfloor} \le e^{- r_0(\var) t},
\]
where $r_0(\var) = \mathcal{O}(T^{-1}) = \mathcal{O}(-1/\log\var)$.

For $t \in [0, T]$, since $\mathbb{P}[ \tau_\var > t \mid (X_{0}, Y_{0}) \in \overline{B_i \times B_i}] \le 1$, there exists a constant $C_0 = \mathcal{O}(1)$ independent of $t$ and $\var$ such that
\[
\mathbb{P}[ \tau_\var > t \mid (X_{0}, Y_{0}) \in \overline{B_i \times B_i} ] \le C_0 e^{- r_0(\var) t}.
\]
Thus, for all $t > 0$,
\[
\mathbb{P}[ \tau_\var > t \mid (X_{0}, Y_{0}) \in \overline{B_i \times B_i} ] \lesssim e^{- r_0(\var) t}.
\]

Finally, since the number of basins is finite, {\bf (H2)}(ii) follows by applying this estimate over all $L$ basins.

\section*{Acknowledgment}
The authors are grateful to the anonymous referees for their valuable comments, which significantly improved the manuscript. 

\bibliography{main}
\bibliographystyle{amsplain}		
\end{document}